\documentclass[11pt]{article}
\usepackage{amscd}
\usepackage{graphics}
\usepackage{color}
\usepackage{cancel}
\usepackage{stmaryrd}
\usepackage{tikz}
\usetikzlibrary{matrix,arrows}
\usepackage{amsmath}
\usepackage{amsthm}
\usepackage{amssymb}
\usepackage{proof}
\usepackage{verbatim,cmll}
\usepackage{setspace}
\usepackage{lscape}
\usepackage{latexsym,relsize}
\usepackage{amscd}
\usepackage{multicol}
\usepackage{bussproofs}
\usepackage{txfonts}
\usepackage{tikz}
\usepackage[position=top]{subfig}
\usepackage{leqno}
\usepackage{authblk}
\usepackage{marvosym}
\usepackage[left=3cm,top=2cm,right=3cm,bottom=2cm]{geometry}
\DeclareSymbolFont{extraup}{U}{zavm}{m}{n}
\DeclareMathSymbol{\varheart}{\mathalpha}{extraup}{86}
\DeclareMathSymbol{\vardiamond}{\mathalpha}{extraup}{87}
\DeclareMathSymbol{\vardiamond}{\mathalpha}{extraup}{87}
\newcommand{\Boxblack}{\blacksquare\,}

\newcommand{\commment}[1]{}

\renewcommand{\phi}{\varphi}
\renewcommand{\emptyset}{\varnothing}

\newcommand\twoheaduparrow{\mathrel{\rotatebox{90}{$\twoheadrightarrow$}}}

\renewcommand{\epsilon}{\varepsilon}

\newcommand{\diamdot}{{\Diamond}\!\!\!\cdot\ }

\newcommand{\diamdotb}{\Diamondblack\!\!\!{\color{white}{\cdot\ }}}
\newcommand{\boxdotb}{\blacksquare\!\!\!{\color{white}{\cdot\ }}}

\newcommand{\nomi}{\mathbf{i}}
\newcommand{\nomj}{\mathbf{j}}
\newcommand{\nomk}{\mathbf{k}}
\newcommand{\cnomm}{\mathbf{m}}

\newcommand{\bigamp}{\mathop{\mbox{\Large \&}}}

\newcommand{\marginnote}[1]{\marginpar{\raggedright\tiny{#1}}}
\theoremstyle{plain}
\newtheorem{thm}{Theorem}
\newtheorem{theorem}{Theorem}[section]

\newtheorem{corollary}[theorem]{Corollary}
\newtheorem{example}[theorem]{Example}

\newtheorem{proposition}[thm]{Proposition}

\newtheorem{lemma}[thm]{Lemma}

\theoremstyle{definition}

\newtheorem{definition}[thm]{Definition}

\newtheorem{remark}[thm]{Remark}

\title{Correspondence and Canonicity Theory of Quasi-Inequalities and $\Pi_2$-Statements in Modal Subordination Algebras}
\author{Zhiguang Zhao}

\date{}

\begin{document}
\maketitle
\begin{abstract}
In the present paper, we study the correspondence and canonicity theory of modal subordination algebras and their dual Stone space with two relations, generalizing correspondence results for subordination algebras in \cite{dR20,dRHaSt20,dRPa21,Sa16}. Due to the fact that the language of modal subordination algebras involves a binary subordination relation, we will find it convenient to use the so-called quasi-inequalities and $\Pi_2$-statements. We use an algorithm to transform (restricted) inductive quasi-inequalities and (restricted) inductive $\Pi_2$-statements to equivalent first-order correspondents on the dual Stone spaces with two relations with respect to arbitrary (resp.\ admissible) valuations.

\end{abstract}

\section{Introduction}\label{Sec:Intro}

In the study of the relations among logic, algebra and topology, a key tool is the dualities between  algebras and topological spaces. This line of research started from Stone \cite{St36}, who showed the duality between Boolean algebras and compact Hausdorff zero-dimensional spaces (which are now called Stone spaces). Later on, many other dualities are studied, e.g.\ Priestley duality for distributive lattices and Priestley spaces \cite{Pr70,Pr72}, Esakia duality for Heyting algebras and Esakia spaces \cite{Es74}, J\'onsson-Tarski duality for modal algebras and modal spaces \cite{JoTa51,JoTa52}, and de Vries duality for de Vries algebras and compact Hausdorff spaces \cite{dV62}.

De Vries duality connects compact Hausdorff spaces with de Vries algebras, which are complete Boolean algebras with a binary relation satisfying some additional properties. Subordination algebras \cite{BeBeSoVe17} are Boolean algebras with a binary relation (called the subordination relation) which generalizes de Vries algebras. Subordinations on Boolean algebras are in 1-1 correspondence with quasi-modal operators \cite{Ce01}. It is proved in \cite{BeBeSoVe17} that subordination algebras are dually equivalent to Stone spaces with a closed relation, which are called subordination spaces. Modal compact Hausdorff spaces are studied in \cite{BeBeHa15} as a the compact Hausdorff generalization of modal spaces, whose counterpart of de Vries duality yields modal de Vries algebras. Therefore, it is natural to consider the generalizations of modal de Vries algebras, namely modal subordination algebras, as the background theory of studying modal compact Hausdorff spaces. Indeed, modal subordination algebras are dually equivalent to Stone spaces with two relations, one closed relation corresponding to the subordination relation, and one clopen relation corresponding to the modality.

It is natural to consider the correspondence theory of propositional formulas on subordination spaces with their first-order corresponding conditions on subordination spaces. Indeed, in the literature, there are works of correspondence theory for subordination algebras and subordination spaces. In \cite{dR20,dRHaSt20,dRPa21}, de Rudder et al. studied correspondence theory of subordination algebras in the perspective of quasi-modal operators. In \cite{Sa16}, Santoli studied the topological correspondence theory between conditions on algebras and first-order conditions on the dual subordination spaces, in the language of a binary connective definable from the squigarrow associated with the subordination relation, using the so-called $\forall\exists$-statements \cite{BeBeSaVe19,BeCaGhLa22,BeGhLa20} (which we call $\Pi_2$-statements in the present paper). Balbiani and Kikot \cite{BaKi12} investigated the Sahlqvist theory in the language of region-based modal logics of space, which uses a contact relation. 

Since the language of modal subordination algebras involves a binary subordination relation, and many natural conditions on subordination algebras and modal subordination algebras involves quasi-inequalities and $\Pi_2$-statements, we will we follow the approach of Santoli in the sense of using $\Pi_2$-statements, and partly follow the approach of de Rudder in the sense of using quasi-modal operators which are not closed under the admissible subsets. We use an algorithm to transform (restricted) inductive quasi-inequalities and (restricted) inductive $\Pi_2$-statements to equivalent first-order correspondents on the dual Stone spaces with two relations with respect to arbitrary (resp.\ admissible) valuations.

The paper is organized as follows: Section \ref{Sec:Modal:Subord:Alg} gives preliminaries on modal subordination algebras and Stone space with two relations. In Part I, we study the correspondence and canonicity theory for quasi-inequalities. Section \ref{Sec:Syntax:Semantics} gives the syntax and semantics of the logic language we are considering. Section \ref{Sec:Preliminaries} gives preliminaries on algorithmic correspondence. We give the definition of inductive quasi-inequalities in Section \ref{sec:Sahlqvist}, define a version of the algorithm $\mathsf{ALBA}$ in Section \ref{Sec:ALBA}, show its soundness in Section \ref{Sec:Soundness}, success on inductive quasi-inequalities in Section \ref{Sec:Success} and the canonicity of restricted inductive quasi-inequalities in Section \ref{Sec:Canonicity}. Section \ref{Sec:Examples} gives some examples of the execution of $\mathsf{ALBA}$ on quasi-inequalities. In Part II, we study the correspondence and canonicity theory for $\Pi_2$-statements. Section \ref{Subsec:Syn:Sem:Pi_2} gives the syntax and semantics for $\Pi_2$-statements. Section \ref{Subsec:Inductive:Pi_2} defines inductive $\Pi_2$-statements and restricted inductive $\Pi_2$-statements. Section \ref{SubSec:ALBA:Pi_2} gives a version of the algorithm $\mathsf{ALBA}^{\Pi_2}$ for $\Pi_2$-statements. We state its soundness with respect to arbitrary valuations in Section \ref{Subsec:Soundness:Pi_2}. We prove its success on inductive and restricted inductive $\Pi_2$-statements in Section \ref{Subsec:Success:Pi_2} and canonicity of restricted inductive $\Pi_2$-statements in Section \ref{Subsec:Canonicity:Pi_2}. Section \ref{Subsec:Examples:Pi_2} gives an example of the execution of $\mathsf{ALBA}$ on a $\Pi_2$-statement. We give some comparisons with existing works in Section \ref{Sec:Comparison}.

\section{Modal Subordination Algebras and Stone Space with Two Relations} \label{Sec:Modal:Subord:Alg}

In this section, we give the definitions of modal subordination algebras and special classes of modal subordination algebras, as well as their dual Stone space with two relations. Basically, we generalize upper continuous modal de Vries algebras (abbreviation: UMDVs, see \cite{BeBeHa15}) to modal subordination algebras. The reason why we choose UMDVs is that the diamond there is finitely additive (see Proposition \ref{Prop:finite:additive}), which makes it easier to develop a correspondence theory.

\subsection{Subordination Algebras, de Vries Algebras and Modal de Vries Algebras}\label{Subsec:MSA}

\begin{definition}[Subordination Algebra, Definition 2.1.1 in \cite{Sa16}]
A \emph{subordination algebra} is a pair $(B,\prec)$ where $B$ is a Boolean algebra and $\prec$ is a \emph{subordination}, which is a binary relation on $B$ satisfying the following properties:

\begin{itemize}
\item $0\prec 0$ and $1\prec 1$;
\item $a\prec b$ and $a\prec c$ implies $a\prec b\land c$;
\item $a\prec c$ and $b\prec c$ implies $a\lor b\prec c$; 
\item $a\leq b\prec c\leq d$ implies $a\prec d$.
\end{itemize}

Equivalently, we can describe a subordination $\prec$ on a Boolean algebra $B$ by an operation $\rightsquigarrow:B\times B\to \{0,1\}\subseteq B$\label{page:squigarrow} such that
\begin{itemize}
\item $a\rightsquigarrow b\in\{0,1\}$;
\item $0\rightsquigarrow 0=1\rightsquigarrow 1=1$; 
\item $a\rightsquigarrow b=1$ and $a\rightsquigarrow c=1$ implies $a\rightsquigarrow b\land c=1$;
\item $a\rightsquigarrow c=1$ and $b\rightsquigarrow c=1$ implies $a\lor b\rightsquigarrow c=1$;
\item $b\rightsquigarrow c=1$, $a\leq b$ and $c\leq d$ implies $a\rightsquigarrow d=1$.
\end{itemize}

Given a subordination $\prec$ on $B$, we can obtain an operation $\rightsquigarrow:B\times B\to\{0,1\}\subseteq B$ satysfying the properties above by defining $a\rightsquigarrow b=1$ iff $a\prec b$. Given an operation $\rightsquigarrow:B\times B\to\{0,1\}\subseteq B$ satysfying the properties above, we can obtain a subordination $\prec$ by defining $\prec=\{(a,b)\in B\times B: a\rightsquigarrow b=1\}$. Therefore, we have a 1-1 correspondence between subordinations $\prec$ and operations $\rightsquigarrow$ satisfying the properties above.
\end{definition}

\begin{definition}[Contact Algebra, Definition 2.1.3 in \cite{Sa16}]
A \emph{contact algebra} is a subordination algebra $(B,\prec)$ where $\prec$ satisfies the following two additional properties:
\begin{itemize}
\item $a\prec b$ implies $a\leq b$;
\item $a\prec b$ implies $\neg b\prec\neg a$.
\end{itemize}
\end{definition}

\begin{definition}[Compingent Algebra, Definition 2.1.4 in \cite{Sa16}]
A \emph{compingent algebra} is a contact algebra $(B,\prec)$ where $\prec$ satisfies the following two additional properties:
\begin{itemize}
\item $a\prec b$ implies $\exists c: a\prec c\prec b$;
\item $a\neq 0$ implies $\exists b\neq 0:b\prec a$.
\end{itemize}
\end{definition}

\begin{definition}[de Vries Algebra, Definition 2.2.1 in \cite{Sa16}]
A \emph{de Vries algebra} is a compingent algebra $(B,\prec)$ where $B$ is a complete Boolean algebra.
\end{definition}

\begin{definition}[Modal de Vries Algebra, Definition 4.7 in \cite{BeBeHa15}]
A \emph{modal de Vries algebra} is a tuple $(B,\prec,\Diamond)$ where $(B,\prec)$ is a de Vries algebra and $\Diamond$ is \emph{de Vries additive}, i.e.\ it satisfies the following two properties:
\begin{itemize}
\item $\Diamond 0=0$;
\item for all  $a_1,a_2,b_1,b_2\in B$, $a_1\prec a_2$ and $b_1\prec b_2$ implies that $\Diamond(a_1\lor a_2)\prec \Diamond(b_1\lor b_2)$.
\end{itemize}
\end{definition}

\begin{proposition}[Proposition 4.8 in \cite{BeBeHa15}]
In a modal de Vries algebra $(B,\prec,\Diamond)$, $\Diamond$ is proximity preserving, i.e.\ $a\prec b$ implies $\Diamond a\prec\Diamond b$ for all $a,b\in B$.
\end{proposition}

\begin{proposition}[Proposition 4.10 in \cite{BeBeHa15}]
In a de Vries algebra $(B,\prec)$, if $\Diamond$ is finitely additive (i.e.\ $\Diamond(a\lor b)=\Diamond a\lor\Diamond b$ for all $a,b\in B$) and proximity preserving, then it is de Vries additive.
\end{proposition}

Notice that in a modal de Vries algebra $(B,\prec,\Diamond)$, the de Vries additive operation $\Diamond$ is not necessarily order-preserving or finitely additive. However, in upper continuous modal de Vries algebras, $\Diamond$ is finitely additive.

\begin{definition}[Upper continuity, Definition 4.14 in \cite{BeBeHa15}]
A modal de Vries algebra $(B,\prec,\Diamond)$ is \emph{upper continuous}, if it satiesfies the following property:
$$\Diamond a=\bigwedge\{\Diamond b:a\prec b\}\mbox{ for any }a,b\in B.$$
\end{definition}

\begin{proposition}[Proposition 4.15 in \cite{BeBeHa15}]\label{Prop:finite:additive}
In an upper continuous modal de Vries algebra $(B,\prec,\Diamond)$, $\Diamond$ is both order-preserving and finitely additive.
\end{proposition}

\subsection{Modal Subordination Algebras, Stone Spaces with Two Relations and their Object-Level Duality}

In this subsection we define the modal subordination algebras, and their dual Stone spaces with two relations, and give their object-level duality. The reason why we choose to make modal subordination algebras as subordination algebras with normal and finitely additive operations is that it is easier to develop the correspondence theory for normal and finitely additive operations.
 
\begin{definition}[Modal Subordination Algebra]
A \emph{modal subordination algebra} is a tuple $(B,\prec,\Diamond)$ where $(B,\prec)$ is a subordination algebra, $\Diamond$ satisfies the following conditions: for all $a,b\in B$,
\begin{itemize}
\item $\Diamond$ is normal, namely $\Diamond 0=0$;
\item $\Diamond$ is finitely additive, namely $\Diamond(a\lor b)=\Diamond a\lor\Diamond b$.\end{itemize}
\end{definition}

\begin{definition}
A modal subordination algebra is 
\begin{itemize}
\item a \emph{modal contact algebra}, if $(B,\prec)$ is a contact algebra; 
\item a \emph{modal compingent algebra}, if $(B,\prec)$ is a compingent algebra;
\item \emph{proximity preserving}, if for all $a,b\in B$, $a\prec b$ implies that $\Diamond a\prec\Diamond b$;
\item an \emph{additive modal de Vries algebra}, if $(B,\prec)$ is a de Vries algebra and $\Diamond$ is proximity preserving;
\item an \emph{upper continuous modal de Vries algebra}, if $(B,\prec,\Diamond)$ is an additive modal de Vries algebra and $\Diamond$ is upper continuous. 
\end{itemize}
\end{definition}

In what follows we will define the dual topological structure of a modal subordination algebra. Since there is the duality between subordination algebras and Stone spaces with closed relations (see \cite[Section 2.1.1]{Sa16}), and there is the duality between modal algebras and modal spaces (see \cite[Chapter 5]{BRV01}), we can put the two together to obtain the duality between modal subordination algebras and Stone spaces with \emph{two relations}.

\begin{definition}[Definition 2.1.9 in \cite{Sa16}]
Take any modal subordination algebra $(B,\prec,\Diamond)$ and any $S\subseteq B$. We define $\twoheaduparrow S$ to be the upset of $S$ with respect to the relation $\prec$, i.e.\ 
$$\twoheaduparrow S:=\{b\in B:\exists s\in S\mbox{ such that }s\prec b\}.$$
\end{definition}

\begin{definition}[Stone Spaces with two relations]\label{Def:StRR'}
A \emph{Stone space with two relations} $(X,\tau,R,R')$ is defined as follows:
\begin{itemize}
\item $(X,\tau)$ is a Stone space;
\item $R$ is a closed relation on $X$, i.e.\ for each closed subset $F$ of $X$, both $R[F]$ and $R^{-1}[F]$ are closed;
\item $R'[x]$ is closed for all $x\in X$ and $R'^{-1}[U]$ is clopen for all clopen $U\subseteq X$.
\end{itemize}
\end{definition}
Given a modal subordination algebra $(B,\prec,\Diamond)$, its dual Stone space with two relations $(X,\tau,R,R')$ is defined as follows:

\begin{itemize}
\item $(X,\tau)$ is the dual Stone space of $B$;
\item $R$ is such that $xRy$ iff $\twoheaduparrow x\subseteq y$ (It is easy to check that $R$ is a closed relation on $X$);
\item $R'$ is such that $xR'y$ iff for all $a\in y$, $\Diamond a\in x$ (It is easy to check that $R'[x]$ is closed for all $x\in X$ and $R'^{-1}[U]$ is clopen for all clopen $U\subseteq X$).
\end{itemize}

Given a Stone space with two relations $(X,\tau,R,R')$, its dual modal subordination algebra $(B,\prec,\Diamond)$ is defined as follows:

\begin{itemize}
\item $B$ is the dual Boolean algebra of the Stone space $(X,\tau)$, i.e.\ $B$ consists of the clopen subsets of $X$;
\item $U\prec V$ iff $R[U]\subseteq V$;
\item $\Diamond U=R'^{-1}[U]$.
\end{itemize}

\subsection{Canonical Extensions}\label{Subsec:Canonical:Extensions}

In this subsection, we define the canonical extensions of modal subordination algebras, as well as give the semantic environment of the correspondence and canonicity theory for modal subordination algebras.

\subsubsection{Canonical Extensions of Boolean Algebras}\label{SubSubSec:CE:BA}

\begin{definition}[Canonical Extension of Boolean Algebras, cf.\ Chapter 6, Definition 104 in \cite{BeBlWo06}]\label{Def:CE:BA}
The \emph{canonical extension} of a Boolean algebra $B$ is a complete Boolean algebra $B^\delta$ containing $B$ as a sub-Boolean algebra such that the following two conditions hold:
\begin{itemize}
\item[](\emph{denseness}) each element of $B^\delta$ can be represented both as a join of meets and as a meet of joins of elements from $B$;
\item[](\emph{compactness}) for all $X,Y \subseteq B$ with $\bigwedge X \leq \bigvee Y$ in $B^\delta$, there are finite subsets $X_0 \subseteq X$ and $Y_0\subseteq Y$ such that $\bigwedge X_0 \leq \bigvee Y_0.$\footnote{In fact, this is an equivalent formulation of the definition in \cite{BeBlWo06}.}
\end{itemize}
\end{definition}
An element $x\in B^\delta$ is \emph{open} (resp.\ \emph{closed}) if it is the join (resp.\ meet) of some $X\subseteq B$. We use $\mathsf{O}(B^\delta)$ (resp.\ $\mathsf{K}(B^\delta)$) to denote the set of open (resp.\ closed) elements of $B^\delta$. It is easy to see that elements in $B$ are exactly the ones which are both closed and open (i.e.,\ \emph{clopen}).

It is well-known that for any given $B$, its canonical extension is unique up to isomorphism and that assuming the axiom of choice, the canonical extension of a Boolean algebra is a perfect Boolean algebra, i.e.,\ a complete and atomic Boolean algebra.

\subsubsection{Canonical Extensions of Maps}\label{SubSubSec:Maps}

Let $A,B$ be Boolean algebras. There are two canonical ways to extend an order-preserving map $f:A\rightarrow B$ to a map $A^{\delta}\to B^{\delta}$:
\begin{definition}[$\sigma$- and $\pi$-extension](\cite[page 375]{BeBlWo06})\label{def:canonical:extension:maps}
For any order-preserving map $f:A\rightarrow B$ and $u\in A^\delta$, we define
\[f^\sigma(u)=\bigvee\{\bigwedge\{f(a):x\leq a\in A\}:u\geq x\in\mathsf{K}(A^\delta)\}\]
\[f^\pi(u)=\bigwedge\{\bigvee\{f(a):y\geq a\in A\}:u \leq y\in\mathsf{O}(A^\delta)\}.\]
\end{definition}

\subsubsection{Canonical Extensions of Modal Subordination Algebras}\label{SubSubSec:MSA}

Since in a modal subordination algebra $(B,\prec,\Diamond)$, $\Diamond$ is normal and finitely additive, it is \emph{smooth}, i.e.,\ $\Diamond^{\sigma}=\Diamond^{\pi}$ (cf.\ \cite[Proposition 111(3)]{BeBlWo06}). Therefore, in the canonical extension we can take either of the two extensions.

For the $\prec$ relation, we define its canonical extension by defining the $\pi$-extension $\rightsquigarrow_{\prec}^{\pi}:B^{\delta}\times B^{\delta}\to B^{\delta}$ of the assocciated strict implication $\rightsquigarrow_{\prec}$ and then take the associated subordination $\prec_{\rightsquigarrow^{\pi}_{\prec}}$ (which we also denote $\prec^{\pi}$) associated with $\rightsquigarrow^{\pi}_{\prec}$. 

\begin{proposition}[Folklore.]
\begin{itemize}
\item $\Diamond^{\sigma}$ is completely join-preserving, i.e.\ it preserves arbitrary (including empty) joins.
\item $\rightsquigarrow^{\pi}_{\prec}$ is completely join-reversing in the first coordinate and completely meet-preserving in the second coordinate, i.e.\ 
\begin{itemize}
\item $\bigvee\{a_i:i\in I\}\rightsquigarrow^{\pi}_{\prec}b=\bigwedge\{a_i\rightsquigarrow^{\pi}_{\prec}b:i\in I\}$;
\item $a\rightsquigarrow^{\pi}_{\prec}\bigwedge\{b_i:i\in I\}=\bigwedge\{a\rightsquigarrow^{\pi}_{\prec}b_i:i\in I\}$.
\end{itemize}
\end{itemize}
\end{proposition}

Indeed, there is another way of obtaining the canonical extension of the modal subordination algebras, namely first take its dual space $(X,\tau,R,R')$, and then drop off the topological structure to obtain a \emph{birelational Kripke frame} $(X,R,R')$, and then define the canonical extension $(B^{\circ},\prec^{\circ},\Diamond^{\circ})$ as follows:

\begin{itemize}
\item $B^{\circ}$ is the power set Boolean algebra of $X$, i.e.\ $B^{\circ}$ consists of all subsets of $X$;
\item $U\prec^{\circ} V$ iff $R[U]\subseteq V$;
\item $\Diamond^{\circ} U=R'^{-1}[U]$.
\end{itemize}

\begin{proposition}(Folklore.)
The two definitions are equivalent, i.e.\ $(B^{\circ},\prec^{\circ},\Diamond^{\circ})=(B^{\delta},\prec^{\pi},\Diamond^{\pi})$.
\end{proposition}

Therefore, the following diagram describes the relation between modal subordination algebras and their canonical extensions, as well as their dual Stone space with two relations, and the birelational Kripke frames obtained by dropping the topology:
\begin{center}
\begin{tikzpicture}[node/.style={circle, draw, fill=black}, scale=1]\label{atable:canonical:extension}
\node (BFD) at (-1.5,-1.5) {$(B,\prec,\Diamond)$};
\node (BRO) at (-1.5,1.5) {$(B^{\delta}, \prec^{\pi}, \Diamond^{\pi})$};
\node (FD) at (1.5,-1.5) {$(X,\tau,R,R')$};
\node (RO) at (1.5,1.5) {$(X,R,R')$};
\draw [right hook->] (BFD) to node[left]{$(\cdot)^{\delta}$} (BRO);
\draw [->] (FD) to node[right]{$U$} (RO);
\draw [<->] (BFD) to node[above] {$\cong^{\partial}$} (FD);
\draw [<->] (BRO) to node[above] {$\cong^{\partial}$} (RO);
\end{tikzpicture}
\end{center}

Here $\cong^{\partial}$ means dual equivalence, $U$ is the forgetful functor dropping the topology and replacing the clopen set Boolean algebra by the powerset Boolean algebra, and $(\cdot)^{\delta}$ is taking the canonical extension.

\section*{Part I: Correspondence and Canonicity for Quasi-Inequalities}

\section{Syntax and Semantics}\label{Sec:Syntax:Semantics}

In this section, we introduce the syntax and semantics of the language of modal subordination algebras.

\subsection{Language and Syntax}\label{Subsec:Lan:Syn}

\begin{definition}
Given a countably infinite set $\mathsf{Prop}$ of propositional variables, the modal subordination language $\mathcal{L}$ is defined as follows:
$$\varphi::=p \mid \bot \mid \top \mid \neg\varphi \mid (\varphi\land\varphi) \mid (\varphi\lor\varphi) \mid (\varphi\to\varphi) \mid \Box\varphi \mid \Diamond\varphi \mid {\diamdot}\phi\mid{\boxdot}\phi,$$
where $p\in \mathsf{Prop}$. We will follow the standard rules for omission of the parentheses. We also use $\mathsf{Prop}(\phi)$ to denote the propositional variables occuring in $\phi$. We use the notation $\vec p$ to denote a set of propositional variables and use $\phi(\vec p)$ to indicate that the propositional variables occur in $\phi$ are all in $\vec p$. We call a formula \emph{pure} if it does not contain propositional variables. We use the notation $\overline{\theta}$ to indicate a finite list of formulas. We use the notation $\theta(\eta/p)$ to indicate uniformly substituting $p$ by $\eta$. 
\end{definition}

In the language of formulas, we use $\Diamond$ and $\Box$ as the syntax for the normal and finitely additive operation, and use ${\diamdot}$ and ${\boxdot}$ as the syntax for the subordination on the Boolean algebra. We will also use the following syntactic binary relations $\leq$ and $\prec$ to formalize the ``less than or equal to'' relation and the subordination relation, respectively. Therefore, we use both ${\diamdot}, {\boxdot}$ and $\prec$ to denote the subordination relation.

\begin{definition}
$\ $
\begin{itemize}
\item An \emph{inequality} is of the form $\phi\leq\psi$ or $\phi\prec\psi$, where $\phi$ and $\psi$ are formulas.

\item A \emph{meta-conjunction of inequalities} is of the form $\phi_1\triangleleft_1\psi_1\ \&\ \ldots \ \&\ \phi_n\triangleleft_n\psi_n$, where $n\geq 1$, $\triangleleft_i\in \{\leq,\prec\}$. Typically a meta-conjunction of inequalities is abbreviated as $\overline{\phi}\triangleleft\overline{\psi}$, and if all $\triangleleft$s are $\leq$ (resp.\ $\prec$), then it is written as $\overline{\phi}\leq\overline{\psi}$ (resp.\ $\overline{\phi}\prec\overline{\psi}$).

\item A \emph{quasi-inequality} is of the form $\overline{\phi}\triangleleft\overline{\psi}\  \Rightarrow\ \overline{\gamma}\triangleleft\overline{\delta}$. 

\end{itemize}
\end{definition}

\subsection{Semantics}\label{Subsec:Seman}

We interpret formulas on the dual Stone space with two relations, with two kinds of valuations, namely admissible valuations which interpret propositional variables as clopen subsets of the space (i.e.\ interpret them as elements of the dual modal subordination algebras), and arbitrary valuations which interpret propositional variables as arbitrary subsets of the space (i.e.\ interpret them as elements of the canonical extensions of the dual modal subordination algebras).

\begin{definition}

In a Stone space with two relations $(X,\tau,R,R')$ (we abuse notation to use $X$ to denote the space), we call $X$ the \emph{domain} of the space, and $R$ is the relation corresponding to the subordination as well as ${\diamdot}$ and ${\boxdot}$, and $R'$ is the relation corresponding to the modalities $\Diamond$ and $\Box$.

\begin{itemize}
\item A \emph{pointed Stone space with two relations} is a pair $(X, w)$ where $w\in X$.
\item An \emph{admissible model} is a pair $M=(X,V)$ where $V:\mathsf{Prop}\to Clop(X)$ is an \emph{admissible valuation} on $X$ such that for all propositional variables $p$, $V(p)$ is a clopen subset of $X$. 
\item An \emph{arbitrary model} is a pair $M=(X,V)$ where $V:\mathsf{Prop}\to P(X)$ is an \emph{arbitrary valuation} on $X$ such that for all propositional variables $p$, $V(p)$ is an arbitrary subset of $X$. 
\end{itemize}

Given a valuation $V$, a propositional variable $p\in\mathsf{Prop}$, a subset $A\subseteq X$, we can define $V^{p}_{A}$, the \emph{$p$-variant of $V$} as follows: $V^{p}_{A}(q)=V(q)$ for all $q\neq p$ and $V^{p}_{A}(p)=A$.

Now the satisfaction relation can be defined as follows: given any Stone space with two relations $(X,\tau,R,R')$, any (admissible or arbitrary) valuation $V$ on $X$, any $w\in X$, 

\begin{center}
\begin{tabular}{l c l}
$X,V,w\Vdash p$ & iff & $w\in V(p)$;\\
$X,V,w\Vdash \bot$ & : & never;\\
$X,V,w\Vdash \top$ & : & always;\\
$X,V,w\Vdash \neg\varphi$ & iff & $X,V,w\nVdash\varphi$;\\
$X,V,w\Vdash\varphi\land\psi$ & iff & $X,V,w\Vdash \varphi$ and $X,V,w\Vdash\psi$;\\
$X,V,w\Vdash\varphi\lor\psi$ & iff & $X,V,w\Vdash \varphi$ or $X,V,w\Vdash\psi$;\\
$X,V,w\Vdash\varphi\to\psi$ & iff & $X,V,w\nVdash \varphi$ or $X,V,w\Vdash\psi$;\\
$X,V,w\Vdash \Box\varphi$ & iff & $\forall v(R'wv\ \Rightarrow\ X,V,v\Vdash\varphi)$;\\
$X,V,w\Vdash\Diamond\varphi$ & iff & $\exists v(R'wv\ \mbox{ and }\ X,V,v\Vdash\varphi)$;\\
$X,V,w\Vdash{\boxdot}\varphi$ & iff & $\forall v(Rvw\ \Rightarrow\ X,V,v\Vdash\varphi)$;\\
$X,V,w\Vdash{\diamdot}\varphi$ & iff & $\exists v(Rvw\ \mbox{ and }\ X,V,v\Vdash\varphi)$.\\
\label{page:downarrow}\\
\end{tabular}
\end{center}

Notice that here ${\boxdot}$ and ${\diamdot}$ are interpreted in the reverse direction of $R$. This is because we would like to make $\phi\prec\psi$ (which is interpreted as $R[V(\phi)]\subseteq V(\psi)$) equivalent to ${\diamdot}\phi\leq\psi$, so we need to make ${\diamdot}\phi$ interpreted as $R[V(\phi)]$.

For any formula $\phi$, we let $\llbracket\varphi\rrbracket^{X,V}=\{w\in X\mid X,V,w\Vdash\varphi\}$ denote the \emph{truth set} of $\varphi$ in $(X,V)$ (sometimes we just use $V(\phi)$ when the space is clear from the context). 
\begin{itemize}
\item The formula $\varphi$ is \emph{globally true} on $(X,V)$ (notation: $X,V\Vdash\varphi$) if $X,V,w\Vdash\varphi$ for every $w\in W$. 
\item We say that $\varphi$ is \emph{admissibly valid} on a Stone space with two relations $X$ (notation: $X\Vdash_{Clop}\varphi$) if $\varphi$ is globally true on $(X,V)$ for every admissible valuation $V$.
\item We say that $\varphi$ is \emph{valid} on a Stone space with two relations $X$ (notation: $X\Vdash_{P}\varphi$) if $\varphi$ is globally true on $(X,V)$ for every arbitrary valuation $V$.
\end{itemize}
\end{definition}

For the semantics of inequalities, meta-conjunctions of inequalities, quasi-inequalities, they are given as follows:

\begin{definition}
$\ $
\begin{itemize}
\item An inequality is interpreted as follows:
$$X,V\Vdash\phi\leq\psi\mbox{ iff }V(\phi)\subseteq V(\psi);$$
$$X,V\Vdash\phi\prec\psi\mbox{ iff }R[V(\phi)]\subseteq V(\psi);$$

\item A meta-conjunction of inequalities is interpreted as follows:
$$X,V\Vdash\phi_1\triangleleft_1\psi_1\ \&\ \ldots \ \&\ \phi_n\triangleleft_n\psi_n\mbox{ iff }X,V\Vdash\phi_i\triangleleft_i\psi_i\mbox{ for all }i=1,\ldots, n;$$

\item A quasi-inequality is interpreted as follows:
$$X,V\Vdash\overline{\phi}\triangleleft\overline{\psi}\  \Rightarrow\ \overline{\gamma}\triangleleft\overline{\delta}\mbox{ iff }$$
$$X,V\Vdash\overline{\gamma}\triangleleft\overline{\delta}\mbox{ holds whenever }X,V\Vdash\overline{\phi}\triangleleft\overline{\psi}.$$
\end{itemize}
\end{definition}

The definitions of validity are similar to formulas.

\section{Preliminaries on Algorithmic Correspondence}\label{Sec:Preliminaries}
In this section, we give necessary preliminaries on the correspondence algorithm $\mathsf{ALBA}$ in the style of \cite{CoGoVa06,CoPa12,Zh21c}. The algorithm $\mathsf{ALBA}$ transforms the input quasi-inequality  $$\phi_1\leq\psi_1\ \&\ldots\&\ \phi_n\leq\psi_n\ \&\ \gamma_1\prec\delta_1\ \&\ldots\&\ \gamma_m\prec\delta_m\ \Rightarrow\ \alpha\triangleleft\beta\mbox{ (where }\triangleleft\in\{\prec,\leq\})$$ into an equivalent set of pure quasi-inequalities which does not contain occurrences of propositional variables, and therefore can be translated into the first-order correspondence language via the standard translation of the expanded language (see page \pageref{Sub:FOL:ST}).

The ingredients for the algorithmic correspondence proof to go through can be listed as follows:

\begin{itemize}
\item An expanded language as the syntax of the algorithm, as well as its semantics;
\item An algorithm $\mathsf{ALBA}$ which transforms a given quasi-inequality into equivalent pure quasi-inequalities;
\item A soundness proof of the algorithm;
\item A syntactically identified class of quasi-inequalities (namely the \emph{inductive quasi-inequalities}) on which the algorithm is successful;
\item A first-order correspondence language and first-order translation which transforms pure quasi-inequalities into their equivalent first-order correspondents.
\item A syntactically identified class of quasi-inequalities (namely the \emph{restricted inductive quasi-inequalities}) which are canonical.
\end{itemize}

In the remainder of this part, we will define an expanded language which the algorithm will manipulate (Section \ref{Sub:expanded:language}), define the first-order correspondence language of the expanded language and the standard translation (Section \ref{Sub:FOL:ST}). We give the definition of inductive quasi-inequalities (Section \ref{sec:Sahlqvist}), define a version of the algorithm $\mathsf{ALBA}$ (Section \ref{Sec:ALBA}), and show its soundness (Section \ref{Sec:Soundness}), success on inductive quasi-inequalities (Section \ref{Sec:Success}) and the canonicity of restricted inductive quasi-inequalities (Section \ref{Sec:Canonicity}).

\subsection{The expanded hybrid modal language}\label{Sub:expanded:language}

In the present subsection, we give the definition of the expanded language, which will be used in the execution of the algorithm:
$$\varphi::=p \mid \nomi \mid \bot \mid \top \mid \neg\varphi \mid (\varphi\land\varphi) \mid (\varphi\lor\varphi) \mid (\varphi\to\varphi) \mid \Box\varphi \mid \Diamond\varphi \mid {\diamdot}\phi\mid{\boxdot}\phi\mid\blacksquare\phi \mid \Diamondblack\phi\mid{\boxdotb}\phi\mid{\diamdotb}\phi$$

where $\nomi\in\mathsf{Nom}$ is called a \emph{nominal}. For $\nomi$, it is interpreted as a singleton set. For $\blacksquare$ and $\Diamondblack$, they are interpreted as the box and diamond modality on the inverse relation $R'^{-1}$, and for ${\boxdotb}$ and ${\diamdotb}$, they are interpreted as the box and diamond modality on the relation $R$.

For the semantics of the expanded language, the valuation $V$ is extended to $\mathsf{Prop}\cup\mathsf{Nom}$ such that $V(\nomi)$ is a singleton for each $\nomi\in\mathsf{Nom}$. \footnote{Notice that we allow admissible valuations to interpret nominals as singletons, even if singletons might not be clopen. The admissibility restrictions are only for the propositional variables.} The additional semantic clauses can be given as follows:
\begin{center}
\begin{tabular}{l c l}
$X,V,w\Vdash\nomi$ & iff & $V(\nomi)=\{w\}$;\\
$X,V,w\Vdash \blacksquare\varphi$ & iff & $\forall v(R'vw\ \Rightarrow\ X,V,v\Vdash\varphi)$;\\
$X,V,w\Vdash\Diamondblack\varphi$ & iff & $\exists v(R'vw\ \mbox{ and }\ X,V,v\Vdash\varphi)$;\\
$X,V,w\Vdash{\boxdotb}\varphi$ & iff & $\forall v(Rwv\ \Rightarrow\ X,V,v\Vdash\varphi)$;\\
$X,V,w\Vdash{\diamdotb}\varphi$ & iff & $\exists v(Rwv\ \mbox{ and }\ X,V,v\Vdash\varphi)$.\\
\end{tabular}
\end{center}

\subsection{The first-order correspondence language and the standard translation}\label{Sub:FOL:ST}

In the first-order correspondence language, we have two binary predicate symbols $R$ and $R'$ corresponding to the two binary relations in the Stone space with two relations, a set of unary predicate symbols $P$ corresponding to each propositional variable $p$.

\begin{definition}
The standard translation of the expanded language is defined as follows:
\begin{itemize}
\item $ST_{x}(p):=Px$;
\item $ST_{x}(\bot):=\bot$;
\item $ST_{x}(\top):=\top$;
\item $ST_{x}(\nomi):=x=i$;
\item $ST_{x}(\neg\phi):=\neg ST_{x}(\phi)$;
\item $ST_{x}(\phi\land\psi):=ST_{x}(\phi)\land ST_{x}(\psi)$;
\item $ST_{x}(\phi\lor\psi):=ST_{x}(\phi)\lor ST_{x}(\psi)$;
\item $ST_{x}(\phi\to\psi):=ST_{x}(\phi)\to ST_{x}(\psi)$;
\item $ST_{x}(\Box\phi):=\forall y(R'xy\to ST_{y}(\phi))$;
\item $ST_{x}(\Diamond\phi):=\exists y(R'xy\land ST_{y}(\phi))$;
\item $ST_{x}(\blacksquare\phi):=\forall y(R'yx\to ST_{y}(\phi))$;
\item $ST_{x}(\Diamondblack\phi):=\exists y(R'yx\land ST_{y}(\phi))$;
\item $ST_{x}({\boxdot}\phi):=\forall y(Ryx\to ST_{y}(\phi))$;
\item $ST_{x}({\diamdot}\phi):=\exists y(Ryx\land ST_{y}(\phi))$;
\item $ST_{x}({\boxdotb}\phi):=\forall y(Rxy\to ST_{y}(\phi))$;
\item $ST_{x}({\diamdotb}\phi):=\exists y(Rxy\land ST_{y}(\phi))$;
\item $ST(\phi\leq\psi):=\forall x(ST_{x}(\phi)\to ST_{x}(\psi))$;
\item $ST(\phi\prec\psi):=\forall x(ST_{x}({\diamdot}\phi)\to ST_{x}(\psi))$;
\item $ST(\phi_1\triangleleft_1\psi_1\ \&\ \ldots \ \&\ \phi_n\triangleleft_n\psi_n):=ST(\phi_1\triangleleft_1\psi_1)\land\ldots\land ST(\phi_n\triangleleft_n\psi_n)$;
\item $ST(\overline{\phi}\triangleleft\overline{\psi}\  \Rightarrow\ \overline{\gamma}\triangleleft\overline{\delta}):=ST(\overline{\phi}\triangleleft\overline{\psi})\to ST(\overline{\gamma}\triangleleft\overline{\delta})$.
\end{itemize}
\end{definition}

It is easy to see that this translation is correct:

\begin{proposition}
For any Stone space with two relations $X$, any valuation $V$ on $X$, any $w\in X$ and any expanded language formula $\phi$, 
$$X,V,w\Vdash\phi\mbox{ iff }X,V\vDash ST_{x}(\phi)[w].$$
\end{proposition}

\begin{proposition}\label{Prop:ST:ineq:quasi:mega}
For any Stone space with two relations $X$, any valuation $V$ on $X$, and inequality $\mathsf{Ineq}$, meta-conjunction of inequalities $\mathsf{MetaConIneq}$, quasi-inequality $\mathsf{Quasi}$,
$$X,V\Vdash\mathsf{Ineq}\mbox{ iff }X,V\vDash ST(\mathsf{Ineq});$$
$$X,V\Vdash\mathsf{MetaConIneq}\mbox{ iff }X,V\vDash ST(\mathsf{MetaConIneq});$$
$$X,V\Vdash\mathsf{Quasi}\mbox{ iff }X,V\vDash ST(\mathsf{Quasi}).$$
\end{proposition}

\section{Inductive Quasi-Inequalities for Modal Subordination Algebras}\label{sec:Sahlqvist}

In this section, we define inductive quasi-inequalities for modal subordination algebras and Stone spaces with two relations. Here we consider quasi-inequalities of the form $$\phi_1\leq\psi_1\ \&\ldots\&\ \phi_n\leq\psi_n\ \&\ \gamma_1\prec\delta_1\ \&\ldots\&\ \gamma_m\prec\delta_m\ \Rightarrow\ \alpha\triangleleft\beta\mbox{ (where }\triangleleft\in\{\prec,\leq\}),$$

where $n,m\geq 0$. We follow the presentation of \cite{CPZ:Trans}.

\begin{definition}[Order-type of propositional variables](cf.\ \cite[page 346]{CoPa12})
For an $n$-tuple $(p_1, \ldots, p_n)$ of propositional variables, an order-type $\epsilon$ of $(p_1, \ldots, p_n)$ is an element in $\{1,\partial\}^{n}$. We say that $p_i$ has order-type 1 if $\epsilon_i=1$, and denote $\epsilon(p_i)=1$ or $\epsilon(i)=1$; we say that $p_i$ has order-type $\partial$ if $\epsilon_i=\partial$, and denote $\epsilon(p_i)=\partial$ or $\epsilon(i)=\partial$.
\end{definition}

\begin{definition}[Signed generation tree]\label{adef: signed gen tree}(cf.\ \cite[Definition 4]{CPZ:Trans})
The \emph{positive} (resp.\ \emph{negative}) {\em generation tree} of any given formula $\theta$ is defined by first labelling the root of the generation tree of $\theta$ with $+$ (resp.\ $-$) and then labelling the children nodes as follows:
\begin{itemize}
\item Assign the same sign to the children nodes of any node labelled with $\Box, \Diamond, \boxdot, \diamdot$, $\blacksquare, \Diamondblack, {\boxdotb}, {\diamdotb}, \lor, \land$;
\item Assign the opposite sign to the child node of any node labelled with $\neg$;
\item Assign the opposite sign to the first child node and the same sign to the second child node of any node labelled with $\to$.
\end{itemize}
Nodes in signed generation trees are \emph{positive} (resp.\ \emph{negative}) if they are signed $+$ (resp.\ $-$).
\end{definition}

Signed generation trees will be used in the quasi-inequalities 
$$\phi_1\leq\psi_1\ \&\ldots\&\ \phi_n\leq\psi_n\ \&\ \gamma_1\prec\delta_1\ \&\ldots\&\ \gamma_m\prec\delta_m\ \Rightarrow\ \alpha\triangleleft\beta\mbox{ (where }\triangleleft\in\{\prec,\leq\}),$$
where the positive generation trees $+\psi_i,+\delta_j,+\alpha$ and the negative generation trees $-\phi_i,-\gamma_j,-\beta$ will be considered. We will also say that a quasi-inequality is \emph{uniform} in a variable $p_i$ if all occurrences of $p_i$ in $+\psi_i,+\delta_j,+\alpha, -\phi_i,-\gamma_j,-\beta$ have the same sign, and that a quasi-inequality is $\epsilon$-\emph{uniform} in an array $\vec{p}$ if it is uniform in $p_i$, occurring with the sign indicated by $\epsilon$ (i.e., $p_i$ has the sign $+$ if $\epsilon(p_i)=1$, and has the sign $-$ if $\epsilon(p_i)=\partial$), for each propositional variable $p_i$ in $\vec{p}$.

For any given formula $\theta(p_1,\ldots, p_n)$, any order-type $\epsilon$ over $n$, and any $1 \leq i \leq n$, an \emph{$\epsilon$-critical node} in a signed generation tree of $\theta$ is a leaf node $+p_i$ when $\epsilon_i = 1$ or $-p_i$ when $\epsilon_i = \partial$. An $\epsilon$-{\em critical branch} in a signed generation tree is a branch from an $\epsilon$-critical nodes. The $\epsilon$-critical occurrences are intended to be those which the algorithm $\mathsf{ALBA}$ will solve for. We say that $+\theta$ (resp.\ $-\theta$) is $\epsilon$-uniform, and write $\epsilon(+\theta)$ (resp.\ $\epsilon(-\theta)$), if every leaf node in the signed generation tree of $+\theta$ (resp.\ $-\theta$) is $\epsilon$-critical.

We will also use the notation $+\iota\ll\ast\theta$ (resp.\ $-\iota\ll \ast\theta$) to indicate that an occurence of a subformula $\iota$ inherits the positive (resp.\ negative) sign from the signed generation tree $\ast\theta$, where $\ast\in\{+,-\}$. We will write $\epsilon(\iota)\ll\ast\theta$ (resp.\ $\epsilon^\partial(\iota)\ll\ast\theta$) to indicate that the signed generation subtree $\iota$, with the sign inherited from $\ast\theta$, is $\epsilon$-uniform (resp.\ $\epsilon^\partial$-uniform). We say that a propositional variable $p$ is \emph{positive} (resp.\ \emph{negative}) in $\theta$ if $+p\ll+\theta$ (resp.\ $-p\ll+\theta$).\label{page:epsilon:subtree}

In what follows, we will use the following classification of nodes:

\begin{definition}\label{adef:good:branches}(Classification of nodes, cf.\ \cite[Definition 5]{CPZ:Trans})
Nodes in signed generation trees are called \emph{$\Delta$-adjoints}, \emph{syntactically left residual (SLR)}, \emph{syntactically right adjoint (SRA)}, and \emph{syntactically right residual (SRR)}, according to Table \ref{aJoin:and:Meet:Friendly:Table}.\footnote{For a detailed explanation why these names are used, see \cite[Remark 3.24]{PaSoZh16}.}
\begin{table}
\begin{center}
\begin{tabular}{| c | c |}
\hline
Skeleton &PIA\\
\hline
$\Delta$-adjoints & SRA \\
\begin{tabular}{ c c c c c c}
$+$ &$\vee$ &\\
$-$ &$\wedge$ &\\
\end{tabular}
&
\begin{tabular}{c c c c c c c }
$+$ &$\wedge$ & $\neg$ & $\Box$ & $\blacksquare$ & ${\boxdot}$ & ${\boxdotb}$\\
$-$ &$\vee$ & $\neg$ & $\Diamond$ & $\Diamondblack$ & ${\diamdot}$ & ${\diamdotb}$ \\
\end{tabular}\\
\hline
SLR &SRR\\
\begin{tabular}{c c c c c c c c}
$+$ & $\wedge$ & $\neg$ & $\Diamond$ & $\Diamondblack$ & ${\diamdot}$ & ${\diamdotb}$\\
$-$ & $\vee$ & $\neg$ & $\Box$ & $\blacksquare$ & ${\boxdot}$  & ${\boxdotb}$& $\to$\\
\end{tabular}
&\begin{tabular}{c c c c}
$+$ &$\vee$ &$\to$\\
$-$ & $\wedge$ &\\
\end{tabular}
\\
\hline
\end{tabular}
\end{center}
\caption{Skeleton and PIA nodes.}\label{aJoin:and:Meet:Friendly:Table}
\vspace{-1em}
\end{table}
\end{definition}

\begin{definition}[Good/PIA/Skeleton branches]
A branch in a signed generation tree is called a 
\begin{itemize}
\item \emph{good branch} if it is the concatenation of two paths $P_1$ and $P_2$, one of which might be of length $0$, such that $P_1$ is a path from the leaf consisting (apart from variable nodes) of PIA-nodes only, and $P_2$ consists (apart from variable nodes) of Skeleton-nodes only;
\item \emph{PIA branch} if it is a good branch and $P_2$ is of length 0;
\item \emph{Skeleton branch} if it is a good branch and $P_1$ is of length 0;
\end{itemize}
\end{definition}

\begin{definition}[Inductive signed generation trees]\label{aInducive:Ineq:Def}(cf.\ \cite[Definition 6]{CPZ:Trans})
For any order-type $\epsilon$ and any irreflexive and transitive binary relation $<_\Omega$ on $p_1,\ldots p_n$ (called \emph{dependence order} on the variables), the signed generation tree $*\theta$ $(* \in \{-, + \})$ of a formula $\theta(p_1,\ldots p_n)$ is \emph{$(\Omega, \epsilon)$-inductive} if
\begin{enumerate}
\item for all $1 \leq i \leq n$, every $\epsilon$-critical branch with leaf $p_i$ is good;
\item every SRR-node in an $\epsilon$-critical branch is either $\bigstar(\iota,\eta)$ or $\bigstar(\eta,\iota)$, where $\bigstar$ is a binary connective, the $\epsilon$-critical branch goes through $\eta$, and
\begin{enumerate}
\item $\epsilon^\partial(\iota) \ll \ast \theta$ (cf.\ page \pageref{page:epsilon:subtree}), and
\item $p_k <_{\Omega} p_i$ for every $p_k$ that occurs in $\iota$.
\end{enumerate}
\end{enumerate}
\end{definition}

The definitions above are mostly straightforward adaptations from standard $\mathsf{ALBA}$ setting for inequalities to the present paper. The following two definitions are specific in the quasi-inequality setting. Intuitively, the receiving inequality $\phi\leq\psi$ (resp.\ $\gamma\prec\delta$) is used to receive minimal valuations, while the solvable inequality $\phi\leq\psi$ (resp.\ $\gamma\prec\delta$) will be transformed into inequalities with minimal valuation $\theta\leq p$ (if $\epsilon(p)=1$) or $p\leq\theta$ (if $\epsilon(p)=\partial$), such that all propositional variables $q$ occurring in $\theta$ has dependence order below $p$, i.e.\ $q<_{\Omega}p$.

\begin{definition}[Receiving inequality]
An inequality $\phi\leq\psi$ (resp.\ $\gamma\prec\delta$) is said to be $(\Omega, \epsilon)$-receiving, if both of $-\phi,+\psi$ (resp.\ $-\gamma,+\delta$) are $\epsilon^{\partial}$-uniform;
\end{definition}

\begin{definition}[Solvable inequality]
An inequality $\phi\leq\psi$ (resp.\ $\gamma\prec\delta$) is said to be $(\Omega, \epsilon)$-solvable, if

\begin{itemize}
\item exactly one of $-\phi,+\psi$ (resp.\ $-\gamma,+\delta$) is $\epsilon^{\partial}$-uniform (without loss of generality we denote the $\epsilon^{\partial}$-uniform one $\star\theta$ and the other one $*\iota$);
\item $*\iota$ is $(\Omega, \epsilon)$-inductive, and all $\epsilon$-critical branches in $*\iota$ are PIA branches;
\item for all the $\epsilon$-critical branches in $*\iota$ ending with $p$, all propositional variables $q$ in $\star\theta$, we have $q<_{\Omega} p$.
\end{itemize}
\end{definition}

\begin{definition}
A quasi-inequality 
$$\phi_1\leq\psi_1\ \&\ldots\&\ \phi_n\leq\psi_n\ \&\ \gamma_1\prec\delta_1\ \&\ldots\&\ \gamma_m\prec\delta_m\ \Rightarrow\ \alpha\triangleleft\beta\mbox{ (where }\triangleleft\in\{\prec,\leq\})$$ is \emph{$(\Omega, \epsilon)$-inductive} if

\begin{itemize}
\item each inequality $\phi\leq\psi$ and $\gamma\prec\delta$ is either $(\Omega, \epsilon)$-receiving or $(\Omega, \epsilon)$-solvable;
\item $+\alpha$ and $-\beta$ are $(\Omega, \epsilon)$-inductive signed generation trees;
\end{itemize}

A quasi-inequality is \emph{inductive} if it is $(\Omega, \epsilon)$-inductive for some ($\Omega$, $\epsilon$).
\end{definition}

\section{Algorithm}\label{Sec:ALBA}
In the present section, we define the algorithm $\mathsf{ALBA}$ which compute the first-order correspondence of the input quasi-inequality in the style of \cite{CoPa12}. The algorithm $\mathsf{ALBA}$ proceeds in three stages. Firstly, $\mathsf{ALBA}$ receives a quasi-inequality $$\phi_1\leq\psi_1\ \&\ldots\&\ \phi_n\leq\psi_n\ \&\ \gamma_1\prec\delta_1\ \&\ldots\&\ \gamma_m\prec\delta_m\ \Rightarrow\ \alpha\triangleleft\beta\mbox{ (where }\triangleleft\in\{\prec,\leq\})$$
as input, which do not contain nominals or black connectives $\Diamondblack,\blacksquare,{\boxdotb},{\diamdotb}$.
\begin{enumerate}
\item \textbf{Preprocessing and first approximation}:
\begin{enumerate}
\item In each inequality $\theta\triangleleft\eta$ (where $\triangleleft\in\{\leq,\prec\}$) in the quasi-inequality, consider the signed generation trees $+\theta$ and $-\eta$, apply the distribution rules:
\begin{enumerate}
\item Push down $+\Diamond,+{\diamdot}, -\neg, +\land, -\to$ by distributing them over nodes labelled with $+\lor$ which are Skeleton nodes (see Figure \ref{Figure:distribution:rules}), and
\item Push down $-\Box,-{\boxdot},+\neg, -\lor, -\to$ by distributing them over nodes labelled with $-\land$ which are Skeleton nodes (see Figure \ref{Figure:distribution:rules:2}).
\end{enumerate}

\begin{figure}[htb]
\centering
\begin{multicols}{8}
\begin{tikzpicture}[scale=0.7]
\tikzstyle{level 1}=[level distance=1cm, sibling distance=1cm]
\tikzstyle{level 2}=[level distance=1cm, sibling distance=1cm]
\tikzstyle{level 3}=[level distance=1cm, sibling distance=1cm]
 \node {$+\Diamond$}         
              child{node{$+\lor$}
                     child{node{$+\alpha$}}
                           child{node{$+\beta$}}}
;
\end{tikzpicture}
\columnbreak

$\Rightarrow$
\columnbreak

\begin{tikzpicture}[scale=0.7]
\tikzstyle{level 1}=[level distance=1cm, sibling distance=1cm]
\tikzstyle{level 2}=[level distance=1cm, sibling distance=1cm]
\tikzstyle{level 3}=[level distance=1cm, sibling distance=1cm]
 \node {$+\lor$}
              child{node{$+\Diamond$}
                     child{node{$+\alpha$}}}
              child{node{$+\Diamond$}
                           child{node{$+\beta$}}}
 ;
\end{tikzpicture}

\columnbreak

$\ $
\columnbreak

$\ $
\columnbreak

\begin{tikzpicture}[scale=0.7]
\tikzstyle{level 1}=[level distance=1cm, sibling distance=1cm]
\tikzstyle{level 2}=[level distance=1cm, sibling distance=1cm]
\tikzstyle{level 3}=[level distance=1cm, sibling distance=1cm]
 \node {$-\neg$}         
              child{node{$+\lor$}
                     child{node{$+\alpha$}}
                           child{node{$+\beta$}}}
 ;
\end{tikzpicture}

\columnbreak

$\Rightarrow$
\columnbreak

\begin{tikzpicture}[scale=0.7]
\tikzstyle{level 1}=[level distance=1cm, sibling distance=1cm]
\tikzstyle{level 2}=[level distance=1cm, sibling distance=1cm]
\tikzstyle{level 3}=[level distance=1cm, sibling distance=1cm]
 \node {$-\land$}
              child{node{$-\neg$}
                     child{node{$+\alpha$}}}
              child{node{$-\neg$}
                           child{node{$+\beta$}}}
 ;
\end{tikzpicture}
\end{multicols}

\centering
\begin{multicols}{8}
\begin{tikzpicture}[scale=0.7]
\tikzstyle{level 1}=[level distance=1cm, sibling distance=1cm]
\tikzstyle{level 2}=[level distance=1cm, sibling distance=1cm]
\tikzstyle{level 3}=[level distance=1cm, sibling distance=1cm]
 \node {$+\land$}         
              child{node{$+\lor$}
                     child{node{$+\alpha$}}
                           child{node{$+\beta$}}}
              child{node{+$\gamma$}}
;
\end{tikzpicture}
\columnbreak

$\ \ \ \ \ \ \ \ \Rightarrow$
\columnbreak

\begin{tikzpicture}[scale=0.7]
\tikzstyle{level 1}=[level distance=1cm, sibling distance=2cm]
\tikzstyle{level 2}=[level distance=1cm, sibling distance=1cm]
\tikzstyle{level 3}=[level distance=1cm, sibling distance=1cm]
 \node {$+\lor$}
              child{node{$+\land$}
                     child{node{$+\alpha$}}
                     child{node{$+\gamma$}}}
              child{node{$+\land$}
                           child{node{$+\beta$}}
                           child{node{$+\gamma$}}}
 ;
\end{tikzpicture}

\columnbreak

$\ $
\columnbreak

$\ $
\columnbreak

\begin{tikzpicture}[scale=0.7]
\tikzstyle{level 1}=[level distance=1cm, sibling distance=1cm]
\tikzstyle{level 2}=[level distance=1cm, sibling distance=1cm]
\tikzstyle{level 3}=[level distance=1cm, sibling distance=1cm]
 \node {$+\land$}         
              child{node{+$\alpha$}}
              child{node{$+\lor$}
                     child{node{$+\beta$}}
                           child{node{$+\gamma$}}}
 ;
\end{tikzpicture}

\columnbreak

$\ \ \ \ \ \ \ \ \Rightarrow$
\columnbreak

\begin{tikzpicture}[scale=0.7]
\tikzstyle{level 1}=[level distance=1cm, sibling distance=2cm]
\tikzstyle{level 2}=[level distance=1cm, sibling distance=1cm]
\tikzstyle{level 3}=[level distance=1cm, sibling distance=1cm]
 \node {$+\lor$}
              child{node{$+\land$}
                     child{node{$+\alpha$}}
                     child{node{$+\beta$}}}
              child{node{$+\land$}
                           child{node{$+\alpha$}}
                           child{node{$+\gamma$}}}
 ;
\end{tikzpicture}
\end{multicols}

\centering
\begin{multicols}{8}
\begin{tikzpicture}[scale=0.7]
\tikzstyle{level 1}=[level distance=1cm, sibling distance=1cm]
\tikzstyle{level 2}=[level distance=1cm, sibling distance=1cm]
\tikzstyle{level 3}=[level distance=1cm, sibling distance=1cm]
 \node {$+{\diamdot}$}         
              child{node{$+\lor$}
                     child{node{$+\alpha$}}
                           child{node{$+\beta$}}}
;
\end{tikzpicture}
\columnbreak

$\Rightarrow$
\columnbreak

\begin{tikzpicture}[scale=0.7]
\tikzstyle{level 1}=[level distance=1cm, sibling distance=1cm]
\tikzstyle{level 2}=[level distance=1cm, sibling distance=1cm]
\tikzstyle{level 3}=[level distance=1cm, sibling distance=1cm]
 \node {$+\lor$}
              child{node{$+{\diamdot}$}
                     child{node{$+\alpha$}}}
              child{node{$+{\diamdot}$}
                           child{node{$+\beta$}}}
 ;
\end{tikzpicture}
\columnbreak

$\ $
\columnbreak

$\ $
\columnbreak

\begin{tikzpicture}[scale=0.7]
\tikzstyle{level 1}=[level distance=1cm, sibling distance=1cm]
\tikzstyle{level 2}=[level distance=1cm, sibling distance=1cm]
\tikzstyle{level 3}=[level distance=1cm, sibling distance=1cm]
 \node {$-\to$}         
              child{node{$+\lor$}
                     child{node{$+\alpha$}}
                           child{node{$+\beta$}}}
              child{node{$-\gamma$}}
;
\end{tikzpicture}
\columnbreak

$\ \ \ \ \ \ \ \ \Rightarrow$
\columnbreak

\begin{tikzpicture}[scale=0.7]
\tikzstyle{level 1}=[level distance=1cm, sibling distance=2cm]
\tikzstyle{level 2}=[level distance=1cm, sibling distance=1cm]
\tikzstyle{level 3}=[level distance=1cm, sibling distance=1cm]
 \node {$-\land$}
              child{node{$-\to$}
                     child{node{$+\alpha$}}
                     child{node{$-\gamma$}}}
              child{node{$-\to$}
                           child{node{$+\beta$}}
                           child{node{$-\gamma$}}}
 ;
\end{tikzpicture}
\end{multicols}
\caption{Distribution rules for $+\lor$}
\label{Figure:distribution:rules}
\end{figure}

\begin{figure}[htb]

\centering
\begin{multicols}{8}
\begin{tikzpicture}[scale=0.7]
\tikzstyle{level 1}=[level distance=1cm, sibling distance=1cm]
\tikzstyle{level 2}=[level distance=1cm, sibling distance=1cm]
\tikzstyle{level 3}=[level distance=1cm, sibling distance=1cm]
 \node {$-\Box$}         
              child{node{$-\land$}
                     child{node{$-\alpha$}}
                           child{node{$-\beta$}}}
 ;
\end{tikzpicture}

\columnbreak

$\ \ \ \ \ \ \ \ \Rightarrow$
\columnbreak

\begin{tikzpicture}[scale=0.7]
\tikzstyle{level 1}=[level distance=1cm, sibling distance=1cm]
\tikzstyle{level 2}=[level distance=1cm, sibling distance=1cm]
\tikzstyle{level 3}=[level distance=1cm, sibling distance=1cm]
 \node {$-\land$}
              child{node{$-\Box$}
                     child{node{$-\alpha$}}}
              child{node{$-\Box$}
                           child{node{$-\beta$}}}
 ;
\end{tikzpicture}

\columnbreak

$\ $
\columnbreak

$\ $
\columnbreak

\begin{tikzpicture}[scale=0.7]
\tikzstyle{level 1}=[level distance=1cm, sibling distance=1cm]
\tikzstyle{level 2}=[level distance=1cm, sibling distance=1cm]
\tikzstyle{level 3}=[level distance=1cm, sibling distance=1cm]
 \node {$+\neg$}         
              child{node{$-\land$}
                     child{node{$-\alpha$}}
                           child{node{$-\beta$}}}
;
\end{tikzpicture}
\columnbreak

$\Rightarrow$
\columnbreak

\begin{tikzpicture}[scale=0.7]
\tikzstyle{level 1}=[level distance=1cm, sibling distance=1cm]
\tikzstyle{level 2}=[level distance=1cm, sibling distance=1cm]
\tikzstyle{level 3}=[level distance=1cm, sibling distance=1cm]
 \node {$+\lor$}
              child{node{$+\neg$}
                     child{node{$-\alpha$}}}
              child{node{$+\neg$}
                           child{node{$-\beta$}}}
 ;
\end{tikzpicture}
\end{multicols}

\centering
\begin{multicols}{8}
\begin{tikzpicture}[scale=0.7]
\tikzstyle{level 1}=[level distance=1cm, sibling distance=1cm]
\tikzstyle{level 2}=[level distance=1cm, sibling distance=1cm]
\tikzstyle{level 3}=[level distance=1cm, sibling distance=1cm]
 \node {$-\lor$}         
              child{node{$-\land$}
                     child{node{$-\alpha$}}
                           child{node{$-\beta$}}}
              child{node{$-\gamma$}}
 ;
\end{tikzpicture}

\columnbreak

$\Rightarrow$
\columnbreak

\begin{tikzpicture}[scale=0.7]
\tikzstyle{level 1}=[level distance=1cm, sibling distance=2cm]
\tikzstyle{level 2}=[level distance=1cm, sibling distance=1cm]
\tikzstyle{level 3}=[level distance=1cm, sibling distance=1cm]
 \node {$-\land$}
              child{node{$-\lor$}
                     child{node{$-\alpha$}}
                     child{node{$-\gamma$}}}
              child{node{$-\lor$}
                           child{node{$-\beta$}}
                           child{node{$-\gamma$}}}
 ;
\end{tikzpicture}

\columnbreak

$\ $
\columnbreak

$\ $
\columnbreak

\begin{tikzpicture}[scale=0.7]
\tikzstyle{level 1}=[level distance=1cm, sibling distance=1cm]
\tikzstyle{level 2}=[level distance=1cm, sibling distance=1cm]
\tikzstyle{level 3}=[level distance=1cm, sibling distance=1cm]
 \node {$-\lor$}         
              child{node{$-\alpha$}}
              child{node{$-\land$}
                     child{node{$-\beta$}}
                           child{node{$-\gamma$}}}
;
\end{tikzpicture}
\columnbreak

$\Rightarrow$
\columnbreak

\begin{tikzpicture}[scale=0.7]
\tikzstyle{level 1}=[level distance=1cm, sibling distance=2cm]
\tikzstyle{level 2}=[level distance=1cm, sibling distance=1cm]
\tikzstyle{level 3}=[level distance=1cm, sibling distance=1cm]
 \node {$-\land$}
              child{node{$-\lor$}
                     child{node{$-\alpha$}}
                     child{node{$-\beta$}}}
              child{node{$-\lor$}
                           child{node{$-\alpha$}}
                           child{node{$-\gamma$}}}
 ;
\end{tikzpicture}
\end{multicols}

\centering
\begin{multicols}{8}
\begin{tikzpicture}[scale=0.7]
\tikzstyle{level 1}=[level distance=1cm, sibling distance=1cm]
\tikzstyle{level 2}=[level distance=1cm, sibling distance=1cm]
\tikzstyle{level 3}=[level distance=1cm, sibling distance=1cm]
 \node {$-{\boxdot}$}         
              child{node{$-\land$}
                     child{node{$-\alpha$}}
                           child{node{$-\beta$}}}
 ;
\end{tikzpicture}

\columnbreak

$\ \ \ \ \ \ \ \ \Rightarrow$
\columnbreak

\begin{tikzpicture}[scale=0.7]
\tikzstyle{level 1}=[level distance=1cm, sibling distance=1cm]
\tikzstyle{level 2}=[level distance=1cm, sibling distance=1cm]
\tikzstyle{level 3}=[level distance=1cm, sibling distance=1cm]
 \node {$-\land$}
              child{node{$-{\boxdot}$}
                     child{node{$-\alpha$}}}
              child{node{$-{\boxdot}$}
                           child{node{$-\beta$}}}
 ;
\end{tikzpicture}
\columnbreak

$\ $
\columnbreak

$\ $
\columnbreak

\begin{tikzpicture}[scale=0.7]
\tikzstyle{level 1}=[level distance=1cm, sibling distance=1cm]
\tikzstyle{level 2}=[level distance=1cm, sibling distance=1cm]
\tikzstyle{level 3}=[level distance=1cm, sibling distance=1cm]
 \node {$-\to$}         
              child{node{+$\alpha$}}
              child{node{$-\land$}
                     child{node{$-\beta$}}
                           child{node{$-\gamma$}}}
 ;
\end{tikzpicture}

\columnbreak

$\Rightarrow$
\columnbreak

\begin{tikzpicture}[scale=0.7]
\tikzstyle{level 1}=[level distance=1cm, sibling distance=2cm]
\tikzstyle{level 2}=[level distance=1cm, sibling distance=1cm]
\tikzstyle{level 3}=[level distance=1cm, sibling distance=1cm]
 \node {$-\land$}
              child{node{$-\to$}
                     child{node{$+\alpha$}}
                     child{node{$-\beta$}}}
              child{node{$-\to$}
                           child{node{$+\alpha$}}
                           child{node{$-\gamma$}}}
 ;
\end{tikzpicture}
\end{multicols}
\caption{Distribution rules for $-\land$}
\label{Figure:distribution:rules:2}
\end{figure}

\item Apply the splitting rules to each inequality occurring in the quasi-inequality:

$$\infer{\theta\leq\eta\ \&\ \theta\leq\iota}{\theta\leq\eta\land\iota}
\qquad
\infer{\theta\leq\iota\ \&\ \eta\leq\iota}{\theta\lor\eta\leq\iota}
$$

$$\infer{\theta\prec\eta\ \&\ \theta\prec\iota}{\theta\prec\eta\land\iota}
\qquad
\infer{\theta\prec\iota\ \&\ \eta\prec\iota}{\theta\lor\eta\prec\iota}
$$

\item Apply the monotone and antitone variable-elimination rules to the whole quasi-inequality:

$$\infer{\overline{\phi}(\bot)\leq\overline{\psi}(\bot)\ \&\ \overline{\gamma}(\bot)\prec\overline{\delta}(\bot)\ \Rightarrow\ \overline{\alpha}(\bot)\leq\overline{\beta}(\bot)\ \&\ \overline{\xi}(\bot)\prec\overline{\chi}(\bot)}{\overline{\phi}(p)\leq\overline{\psi}(p)\ \&\ \overline{\gamma}(p)\prec\overline{\delta}(p)\ \Rightarrow\ \overline{\alpha}(p)\leq\overline{\beta}(p)\ \&\ \overline{\xi}(p)\prec\overline{\chi}(p)}$$

if $p$ is positive in $\overline{\phi},\overline{\gamma},\overline{\beta},\overline{\chi}$ and negative in $\overline{\psi},\overline{\delta},\overline{\alpha},\overline{\xi}$;

$$\infer{\overline{\phi}(\top)\leq\overline{\psi}(\top)\ \&\ \overline{\gamma}(\top)\prec\overline{\delta}(\top)\ \Rightarrow\ \overline{\alpha}(\top)\leq\overline{\beta}(\top)\ \&\ \overline{\xi}(\top)\prec\overline{\chi}(\top)}{\overline{\phi}(p)\leq\overline{\psi}(p)\ \&\ \overline{\gamma}(p)\prec\overline{\delta}(p)\ \Rightarrow\ \overline{\alpha}(p)\leq\overline{\beta}(p)\ \&\ \overline{\xi}(p)\prec\overline{\chi}(p)}$$

if $p$ is negative in $\overline{\phi},\overline{\gamma},\overline{\beta},\overline{\chi}$ and positive in $\overline{\psi},\overline{\delta},\overline{\alpha},\overline{\xi}$;

\item Apply the subordination rewritting rule to each inequality with $\prec$ in order to turn it into $\leq$:
$$\infer{{\diamdot}\theta\leq\eta}{\theta\prec\eta}$$
\end{enumerate}

Now we have a quasi-inequality of the form 
$$\overline{\phi}\leq\overline{\psi}\ \Rightarrow\ \overline{\alpha}\leq\overline{\beta},$$
which we split into a meta-conjunction of quasi-inequalities
$$\overline{\phi}\leq\overline{\psi}\ \Rightarrow\ \alpha\leq\beta,$$
where $\alpha\leq\beta$ belong to $\overline{\alpha}\leq\overline{\beta}$.

We denote Preprocess:=$\{\overline{\phi}\leq\overline{\psi}\ \Rightarrow\ \alpha\leq\beta: \alpha\leq\beta\mbox{ belong to }\overline{\alpha}\leq\overline{\beta}\}$.

Now for each quasi-inequality in Preprocess, we apply the following first-approximation rule:

$$\infer{\overline{\phi}\leq\overline{\psi}\ \&\ \nomi_0\leq\alpha\ \&\ \beta\leq\neg\nomi_1 \Rightarrow\ \nomi_0\leq\neg\nomi_1}{\overline{\phi}\leq\overline{\psi}\ \Rightarrow\ \alpha\leq\beta}$$

Now for each quasi-inequality, we focus on the set of its antecedent inequalities $\{\overline{\phi}\leq\overline{\psi}, \nomi_0\leq\alpha, \beta\leq\neg\nomi_1\}$, which we call a \emph{system}.

\item \textbf{The reduction-elimination cycle}:

In this stage, for each $\{\overline{\phi}\leq\overline{\psi}, \nomi_0\leq\alpha, \beta\leq\neg\nomi_1\}$, we apply the following rules together with the splitting rules in the previous stage to eliminate all the propositional variables in the set of inequalities:

\begin{enumerate}

\item Residuation rules:\label{Page:Residuation:Rules}

\begin{prooftree}
\AxiomC{$\Diamond\theta\leq\iota$}
\RightLabel{($\Diamond$-Res)}
\UnaryInfC{$\theta\leq\blacksquare\iota$}
\AxiomC{$\neg\theta\leq\iota$}
\RightLabel{($\neg$-Res-Left)}
\UnaryInfC{$\neg\iota\leq\theta$}
\noLine\BinaryInfC{}
\end{prooftree}

\begin{prooftree}
\AxiomC{$\theta\leq\Box\iota$}
\RightLabel{($\Box$-Res)}
\UnaryInfC{$\Diamondblack\theta\leq\iota$}
\AxiomC{$\theta\leq\neg\iota$}
\RightLabel{($\neg$-Res-Right)}
\UnaryInfC{$\iota\leq\neg\theta$}
\noLine\BinaryInfC{}
\end{prooftree}

\begin{prooftree}
\AxiomC{$\theta\land\iota\leq\eta$}
\RightLabel{($\land$-Res-1)}
\UnaryInfC{$\theta\leq\iota\to\eta$}
\AxiomC{$\theta\leq\iota\lor\eta$}
\RightLabel{($\lor$-Res-1)}
\UnaryInfC{$\theta\land\neg\iota\leq\eta$}
\noLine\BinaryInfC{}
\end{prooftree}

\begin{prooftree}
\AxiomC{$\theta\land\iota\leq\eta$}
\RightLabel{($\land$-Res-2)}
\UnaryInfC{$\iota\leq\theta\to\eta$}
\AxiomC{$\theta\leq\iota\lor\eta$}
\RightLabel{($\lor$-Res-2)}
\UnaryInfC{$\theta\land\neg\eta\leq\iota$}
\noLine\BinaryInfC{}
\end{prooftree}

\begin{prooftree}
\AxiomC{$\theta\leq\iota\to\eta$}
\RightLabel{($\to$-Res-1)}
\UnaryInfC{$\theta\land\iota\leq\eta$}
\AxiomC{$\theta\leq\iota\to\eta$}
\RightLabel{($\to$-Res-2)}
\UnaryInfC{$\iota\leq\theta\to\eta$}
\noLine\BinaryInfC{}
\end{prooftree}

\begin{prooftree}
\AxiomC{${\diamdot}\theta\leq\iota$}
\RightLabel{(${\diamdot}$-Res)}
\UnaryInfC{$\theta\leq{\boxdotb}\iota$}
\AxiomC{$\theta\leq{\boxdot}\iota$}
\RightLabel{(${\boxdot}$-Res)}
\UnaryInfC{${\diamdotb}\theta\leq\iota$}
\noLine\BinaryInfC{}
\end{prooftree}

\item Approximation rules:

$$\infer{\nomj\leq\theta\ \ \ \nomi\leq\Diamond \nomj}{\nomi\leq\Diamond\theta}
\qquad
\infer{\theta\leq  \neg\nomj\ \ \ \Box  \neg\nomj\leq \neg\nomi}{\Box\theta\leq \neg\nomi}
$$
$$
\infer{\nomj\leq\theta\ \ \ \nomi\leq{\diamdot}\nomj}{\nomi\leq{\diamdot}\theta}
\qquad
\infer{\theta\leq  \neg\nomj\ \ \ {\boxdot} \neg\nomj\leq \neg\nomi}{{\boxdot}\theta\leq \neg\nomi}
$$

$$
\infer{\nomj\leq\alpha\ \ \ \ \ \ \ \beta\leq\neg\nomk\ \ \ \ \ \ \ \nomj\rightarrow\neg\nomk\leq\neg\nomi}{\alpha\rightarrow\beta\leq\neg\nomi}
$$

The nominals introduced by the approximation rules must not occur in the system before applying the rule.

\item The Ackermann rules. These two rules are the core of $\mathsf{ALBA}$, since their application eliminates propositional variables. In fact, all the preceding steps are aimed at reaching a shape in which the rules can be applied. Notice that an important feature of these rules is that they are executed on the whole set of inequalities, and not on a single inequality.\\

The right-handed Ackermann rule:

The system 
$\left\{ \begin{array}{ll}
\theta_1\leq p \\
\vdots\\
\theta_n\leq p \\
\eta_1\leq\iota_1\\
\vdots\\
\eta_m\leq\iota_m\\
\end{array} \right.$ 
is replaced by 
$\left\{ \begin{array}{ll}
\eta_1((\theta_1\lor\ldots\lor\theta_n)/p)\leq\iota_1((\theta_1\lor\ldots\lor\theta_n)/p) \\
\vdots\\
\eta_m((\theta_1\lor\ldots\lor\theta_n)/p)\leq\iota_m((\theta_1\lor\ldots\lor\theta_n)/p) \\
\end{array} \right.$
where:
\begin{enumerate}
\item $p$ does not occur in $\theta_1, \ldots, \theta_n$;
\item Each $\eta_i$ is positive, and each $\iota_i$ negative in $p$, for $1\leq i\leq m$.
\end{enumerate}

The left-handed Ackermann rule:

The system
$\left\{ \begin{array}{ll}
p\leq\theta_1 \\
\vdots\\
p\leq\theta_n \\
\eta_1\leq\iota_1\\
\vdots\\
\eta_m\leq\iota_m\\
\end{array} \right.$
is replaced by
$\left\{ \begin{array}{ll}
\eta_1((\theta_1\land\ldots\land\theta_n)/p)\leq\iota_1((\theta_1\land\ldots\land\theta_n)/p) \\
\vdots\\
\eta_m((\theta_1\land\ldots\land\theta_n)/p)\leq\iota_m((\theta_1\land\ldots\land\theta_n)/p) \\
\end{array} \right.$
where:
\begin{enumerate}
\item $p$ does not occur in $\theta_1, \ldots, \theta_n$;
\item Each $\eta_i$ is negative, and each $\iota_i$ positive in $p$, for $1\leq i\leq m$.
\end{enumerate}
\end{enumerate}

\item \textbf{Output}: If in the previous stage, for some set of inequalities, the algorithm gets stuck, i.e.\ some propositional variables cannot be eliminated by the application of the reduction rules, then the algorithm halts and output ``failure''. Otherwise, each initial set of inequalities after the first approximation has been reduced to a set of pure inequalities Reduce$(\overline{\phi}\leq\overline{\psi}, \nomi_0\leq\alpha, \beta\leq\neg\nomi_1)$, and then the output is a set of quasi-inequalities $\{\&$Reduce$(\overline{\phi}\leq\overline{\psi}, \nomi_0\leq\alpha, \beta\leq\neg\nomi_1)\Rightarrow \nomi_0\leq \neg\nomi_1\mid \overline{\phi}\leq\overline{\psi}\ \Rightarrow\ \alpha\leq\beta\in$Preprocess$\}$. Then we can use the standard translation of the set of quasi-inequalities to obtain the first-order correspondence.
\end{enumerate}

\section{Soundness}\label{Sec:Soundness}

In this section we show the soundness of the algorithm with respect to the arbitrary valuations. The soundness proof follows the style of \cite{CoPa12}. For some of the rules, the soundness proofs are the same to existing literature and hence are omitted, so we only give details for the proofs which are different.

\begin{theorem}[Soundness]\label{Thm:Soundness}
If $\mathsf{ALBA}$ runs successfully on an input quasi-inequality $\mathsf{Quasi}$ and outputs a first-order formula $\mathsf{FO(Quasi)}$, then for any Stone space with two relations $(X,\tau,R,R')$, $$X\Vdash_{P}\mathsf{Quasi}\mbox{ iff }X\vDash\mathsf{FO(Quasi)}.$$
\end{theorem}

\begin{proof}
The proof goes similarly to \cite[Theorem 8.1]{CoPa12}. Let $$\overline{\phi}\leq\overline{\psi}\ \Rightarrow\ \alpha\triangleleft\beta$$ denote the input quasi-inequality $\mathsf{Quasi}$, let 
$$\overline{\phi}_{i}\leq\overline{\psi}_{i}\ \Rightarrow\ \alpha_{i}\leq\beta_{i}, i\in I$$ denote the quasi-inequalities before the first-approximation rule, let
$$\overline{\phi}_{i}\leq\overline{\psi}_{i}\ \&\ \nomi_{i,0}\leq\alpha_{i}\ \&\ \beta_{i}\leq\neg\nomi_{i,1}\Rightarrow \nomi_{i,0}\leq\neg\nomi_{i,1}, i\in I$$ denote the quasi-inequalities after the first-approximation rule, let
$$\mbox{Reduce}(\overline{\phi}_{i}\leq\overline{\psi}_{i}, \nomi_{i,0}\leq\alpha, \beta\leq\neg\nomi_{i,1}), i\in I$$
denote the sets of inequalities after Stage 2, let 
$$\mathsf{FO(Quasi)}$$ denote the standard translation of the quasi-inequalities into first-order formulas, then it suffices to show the equivalence from (\ref{Crct:Eqn0}) to (\ref{Crct:Eqn4}) given below:

\begin{eqnarray}
&&X\Vdash_{P}\overline{\phi}\leq\overline{\psi}\ \Rightarrow\ \alpha\triangleleft\beta\label{Crct:Eqn0}\\
&&X\Vdash_{P}\overline{\phi}_{i}\leq\overline{\psi}_{i}\ \Rightarrow\ \alpha_{i}\leq\beta_{i},\mbox{ for all }i\in I\label{Crct:Eqn1}\\
&&X\Vdash_{P}\overline{\phi}_{i}\leq\overline{\psi}_{i}\ \&\ \nomi_{i,0}\leq\alpha_{i}\ \&\ \beta_{i}\leq\neg\nomi_{i,1}\Rightarrow \nomi_{i,0}\leq\neg\nomi_{i,1}, i\in I\label{Crct:Eqn2}\\
&&X\Vdash_{P}\mbox{Reduce}(\overline{\phi}_{i}\leq\overline{\psi}_{i}, \nomi_{i,0}\leq\alpha, \beta\leq\neg\nomi_{i,1})\ \Rightarrow\ \nomi_{i,0}\leq\neg\nomi_{i,1}, i\in I\label{Crct:Eqn3}\\
&&X\vDash\mathsf{FO(Quasi)}\label{Crct:Eqn4}
\end{eqnarray}

\begin{itemize}
\item The equivalence between (\ref{Crct:Eqn0}) and (\ref{Crct:Eqn1}) follows from Proposition \ref{prop:Soundness:stage:1};
\item the equivalence between (\ref{Crct:Eqn1}) and (\ref{Crct:Eqn2}) follows from Proposition \ref{prop:Soundness:first:approximation};
\item the equivalence between (\ref{Crct:Eqn2}) and (\ref{Crct:Eqn3}) follows from Propositions \ref{Prop:Stage:2}, \ref{Prop:Ackermann};
\item the equivalence between (\ref{Crct:Eqn3}) and (\ref{Crct:Eqn4}) follows from Proposition \ref{Prop:ST:ineq:quasi:mega}.
\end{itemize}
\end{proof}

In the remainder of this section, we prove the soundness of the rules in Stage 1, 2 and 3.

\begin{proposition}[Soundness of the rules in Stage 1]\label{prop:Soundness:stage:1}
For the distribution rules, the splitting rules, the monotone and antitone variable-elimination rules and the subordination rewritting rule, they are sound in both directions in $X$.
\end{proposition}

\begin{proof}
\begin{enumerate}
\item For the soundness of the distribution rules, it follows from the fact that the corresponding distribution laws are valid in $X$, which can be found in \cite[Proposition 6.2]{Zh21c}.

\item For the soundness of the splitting rules, the rules involving $\leq$ are the same to the same rules in \cite[Proposition 6.2]{Zh21c}. For the rules involving $\prec$, it follows from the following fact:
\begin{center}
\begin{tabular}{c l}
& $X,V\Vdash\theta\prec\eta\land\iota$\\
iff & $R[V(\theta)]\subseteq V(\eta\land\iota)$\\
iff & $R[V(\theta)]\subseteq V(\eta)\cap V(\iota)$\\
iff & $R[V(\theta)]\subseteq V(\eta)$ and $R[V(\theta)]\subseteq V(\iota)$\\
iff & $X,V\Vdash\theta\prec\eta$ and $X,V\Vdash\theta\prec\iota$.\\
\end{tabular}
\end{center}
\begin{center}
\begin{tabular}{c l}
& $X,V\Vdash\theta\lor\eta\prec\iota$\\
iff & $R[V(\theta)\cup V(\eta)]\subseteq V(\iota)$\\
iff & $R[V(\theta)]\cup R[V(\eta)]\subseteq V(\iota)$\\
iff & $R[V(\theta)]\subseteq V(\iota)$ and $R[V(\eta)]\subseteq V(\iota)$\\
iff & $X,V\Vdash\theta\prec\iota$ and $X,V\Vdash\eta\prec\iota.$\\
\end{tabular}
\end{center}
\item For the soundness of the monotone and antitone variable elimination rules, we show the soundness for the first rule. Suppose $p$ is positive in $\overline{\phi},\overline{\gamma},\overline{\beta},\overline{\chi}$ and negative in $\overline{\psi},\overline{\delta},\overline{\alpha},\overline{\xi}$. 

$(\Downarrow)$: If $$X\Vdash_{P}\overline{\phi}(p)\leq\overline{\psi}(p)\ \&\ \overline{\gamma}(p)\prec\overline{\delta}(p)\ \Rightarrow\ \overline{\alpha}(p)\leq\overline{\beta}(p)\ \&\ \overline{\xi}(p)\prec\overline{\chi}(p),$$ then for all valuations $V$, we have $$X,V\Vdash\overline{\phi}(p)\leq\overline{\psi}(p)\ \&\ \overline{\gamma}(p)\prec\overline{\delta}(p)\ \Rightarrow\ \overline{\alpha}(p)\leq\overline{\beta}(p)\ \&\ \overline{\xi}(p)\prec\overline{\chi}(p),$$ then for the valuation $V^{p}_{\emptyset}$ such that $V^{p}_{\emptyset}$ is the same as $V$ except that $V^{p}_{\emptyset}(p)=\emptyset$, 
$$X,V^{p}_{\emptyset}\Vdash\overline{\phi}(p)\leq\overline{\psi}(p)\ \&\ \overline{\gamma}(p)\prec\overline{\delta}(p)\ \Rightarrow\ \overline{\alpha}(p)\leq\overline{\beta}(p)\ \&\ \overline{\xi}(p)\prec\overline{\chi}(p),$$ therefore $$X, V\Vdash\overline{\phi}(\bot)\leq\overline{\psi}(\bot)\ \&\ \overline{\gamma}(\bot)\prec\overline{\delta}(\bot)\ \Rightarrow\ \overline{\alpha}(\bot)\leq\overline{\beta}(\bot)\ \&\ \overline{\xi}(\bot)\prec\overline{\chi}(\bot),$$ so $$X\Vdash_{P}\overline{\phi}(\bot)\leq\overline{\psi}(\bot)\ \&\ \overline{\gamma}(\bot)\prec\overline{\delta}(\bot)\ \Rightarrow\ \overline{\alpha}(\bot)\leq\overline{\beta}(\bot)\ \&\ \overline{\xi}(\bot)\prec\overline{\chi}(\bot).$$ 

$(\Uparrow)$: For the other direction, consider any valuation $V$ on $X$, by the fact that $p$ is positive in $\overline{\phi},\overline{\gamma},\overline{\beta},\overline{\chi}$ and negative in $\overline{\psi},\overline{\delta},\overline{\alpha},\overline{\xi}$, we have that 
$$X,V\vDash\overline{\psi}(p)\leq\overline{\psi}(\bot)$$
$$X,V\vDash\overline{\delta}(p)\leq\overline{\delta}(\bot)$$
$$X,V\vDash\overline{\alpha}(p)\leq\overline{\alpha}(\bot)$$ 
$$X,V\vDash\overline{\xi}(p)\leq\overline{\xi}(\bot),$$
and $$X,V\vDash\overline{\phi}(\bot)\leq\overline{\phi}(p)$$
$$X,V\vDash\overline{\gamma}(\bot)\leq\overline{\gamma}(p)$$
$$X,V\vDash\overline{\beta}(\bot)\leq\overline{\beta}(p)$$
$$X,V\vDash\overline{\chi}(\bot)\leq\overline{\chi}(p).$$
Suppose that $$X,V\Vdash\overline{\phi}(p)\leq\overline{\psi}(p)\ \&\ \overline{\gamma}(p)\prec\overline{\delta}(p),$$ then
$$X,V\Vdash\overline{\phi}(\bot)\leq\overline{\phi}(p)\leq\overline{\psi}(p)\leq\overline{\psi}(\bot)\ \&\ \overline{\gamma}(\bot)\leq\overline{\gamma}(p)\prec\overline{\delta}(p)\leq\overline{\delta}(\bot),$$
so $$X,V\Vdash\overline{\phi}(\bot)\leq\overline{\psi}(\bot)\ \&\ \overline{\gamma}(\bot)\prec\overline{\delta}(\bot),$$
therefore by assumption that 
$$X\Vdash_{P}\overline{\phi}(\bot)\leq\overline{\psi}(\bot)\ \&\ \overline{\gamma}(\bot)\prec\overline{\delta}(\bot)\ \Rightarrow\ \overline{\alpha}(\bot)\leq\overline{\beta}(\bot)\ \&\ \overline{\xi}(\bot)\prec\overline{\chi}(\bot),$$
we have 
$$X,V\Vdash\overline{\alpha}(\bot)\leq\overline{\beta}(\bot)\ \&\ \overline{\xi}(\bot)\prec\overline{\chi}(\bot),$$
so $$X,V\Vdash\overline{\alpha}(p)\leq\overline{\alpha}(\bot)\leq\overline{\beta}(\bot)\leq\overline{\beta}(p)\ \&\ \overline{\xi}(p)\leq\overline{\xi}(\bot)\prec\overline{\chi}(\bot)\leq\overline{\chi}(p),$$
then $$X,V\Vdash\overline{\alpha}(p)\leq\overline{\beta}(p)\ \&\ \overline{\xi}(p)\prec\overline{\chi}(p),$$
so we get that for all $V$, 
$$X,V\Vdash\overline{\phi}(p)\leq\overline{\psi}(p)\ \&\ \overline{\gamma}(p)\prec\overline{\delta}(p)\ \Rightarrow\ \overline{\alpha}(p)\leq\overline{\beta}(p)\ \&\ \overline{\xi}(p)\prec\overline{\chi}(p),$$
so 
$$X\Vdash_{P}\overline{\phi}(p)\leq\overline{\psi}(p)\ \&\ \overline{\gamma}(p)\prec\overline{\delta}(p)\ \Rightarrow\ \overline{\alpha}(p)\leq\overline{\beta}(p)\ \&\ \overline{\xi}(p)\prec\overline{\chi}(p).$$
The soundness of the other variable elimination rule is similar.

\item For the subordination rewritting rule, its soundness follows from the following fact:
\begin{center}
\begin{tabular}{c l}
& $X,V\Vdash\theta\prec\eta$\\
iff & $R[V(\theta)]\subseteq V(\eta)$\\
iff & $V({\diamdot}\theta)\subseteq V(\eta)$\\
iff & $X,V\Vdash{\diamdot}\theta\leq\eta$.\\
\end{tabular}
\end{center}
\end{enumerate}
\end{proof}
\begin{proposition}\label{prop:Soundness:first:approximation}
(\ref{Crct:Eqn1}) and (\ref{Crct:Eqn2}) are equivalent, i.e.\ the first-approximation rule is sound in $\mathbb{F}$.
\end{proposition}

\begin{proof}
(\ref{Crct:Eqn1}) $\Rightarrow$ (\ref{Crct:Eqn2}): Suppose $X\Vdash_{P}\overline{\phi}\leq\overline{\psi}\ \Rightarrow\ \alpha\leq\beta$. Then for any valuation $V$, if $V(\overline{\phi})\subseteq V(\overline{\psi})$, $X,V\Vdash\nomi_{0}\leq\alpha$ and $X,V\Vdash\beta\leq\neg\nomi_{1}$, then $X,V,V(\nomi_{0})\Vdash\alpha$ and $X,V,V(\nomi_{1})\nVdash\beta$, so by $X\Vdash_{P}\overline{\phi}\leq\overline{\psi}\ \Rightarrow\ \alpha\leq\beta$ we have $X,V,V(\nomi_{0})\Vdash\beta$, so $V(\nomi_{0})\neq V(\nomi_{1})$, so $X,V\Vdash\nomi_{0}\leq\neg\nomi_{1}$.

(\ref{Crct:Eqn2}) $\Rightarrow$ (\ref{Crct:Eqn1}): Suppose $X\Vdash_{P}\overline{\phi}\leq\overline{\psi}\ \&\ \nomi_0\leq\alpha\ \&\ \beta\leq\neg\nomi_1 \Rightarrow\ \nomi_0\leq\neg\nomi_1$. If $X\nVdash_{P}\overline{\phi}\leq\overline{\psi}\ \Rightarrow\ \alpha\leq\beta$, there is a valuation $V$ and a $w\in X$ such that $V(\overline{\phi})\subseteq V(\overline{\psi})$, $X,V,w\Vdash\alpha$ and $X,V,w\nVdash\beta$. Then by taking $V^{\nomi_{0}, \nomi_{1}}_{w, w}$ to be the valuation which is the same as $V$ except that $V^{\nomi_{0}, \nomi_{1}}_{w, w}(\nomi_{0})=V^{\nomi_{0}, \nomi_{1}}_{w, w}(\nomi_{1})=\{w\}$, since $\nomi_{0}, \nomi_{1}$ do not occur in $\alpha$ and $\beta$, we have that $X,V^{\nomi_{0},\nomi_{1}}_{w,w},w\Vdash\alpha$ and $X,V^{\nomi_{0}\nomi_{1}}_{w,w},w\nVdash\beta$, therefore $X,V^{\nomi_{0},\nomi_{1}}_{w,w}\Vdash\nomi_{0}\leq\alpha$ and $X,V^{\nomi_{0},\nomi_{1}}_{w,w}\Vdash\beta\leq\neg\nomi_{1}$. It is easy to see that $V^{\nomi_{0},\nomi_{1}}_{w,w}(\overline{\phi})\subseteq V^{\nomi_{0},\nomi_{1}}_{w,w}(\overline{\psi})$. By $X\Vdash_{P}\overline{\phi}\leq\overline{\psi}\ \&\ \nomi_0\leq\alpha\ \&\ \beta\leq\neg\nomi_1 \Rightarrow\ \nomi_0\leq\neg\nomi_1$, we have that $X, V^{\nomi_{0}, \nomi_{1}}_{w, w}\Vdash\nomi_{0}\leq\neg\nomi_{1}$, so $X, V^{\nomi_{0}, \nomi_{1}}_{w, w},w\Vdash\nomi_{0}$ implies that $X, V^{\nomi_{0}, \nomi_{1}}_{w, w},w\Vdash\neg\nomi_{1}$, a contradiction. So $X\Vdash_{P}\overline{\phi}\leq\overline{\psi}\ \Rightarrow\ \alpha\leq\beta$.
\end{proof}

The next step is to show the soundness of each rule of Stage 2. For each rule, before the application of this rule we have a set of inequalities $S$ (which we call the \emph{system}), after applying the rule we get a set of inequalities $S'$, the soundness of Stage 2 is then the equivalence of the following two conditions:
\begin{itemize}
\item $X\Vdash_{P}\bigamp S\ \Rightarrow \nomi_0\leq\neg\nomi_1$;
\item $X\Vdash_{P}\bigamp S'\ \Rightarrow \nomi_0\leq\neg\nomi_1$;
\end{itemize}

where $\bigamp S$ denote the meta-conjunction of inequalities of $S$. It suffices to show the following property:

\begin{itemize}\label{condition:1:4:equivalence}
\item For any $X$, any valuation $V$, if $X,V\Vdash S$, then there is a valuation $V'$ such that $V'(\nomi_0)=V(\nomi_0)$, $V'(\nomi_1)=V(\nomi_1)$ and $X,V'\Vdash S'$;
\item For any $X$, any valuation $V'$, if $X,V'\Vdash S'$, then there is a valuation $V$ such that $V(\nomi_0)=V'(\nomi_0)$, $V(\nomi_1)=V'(\nomi_1)$ and $X,V\Vdash S$.
\end{itemize}

\begin{proposition}\label{Prop:Stage:2}
The splitting rules, the approximation rules and the residuation rules in Stage 2 are sound in both directions in $X$.
\end{proposition}

\begin{proof}
The soundness proofs are the same to the soundness of the same rules in \cite[Proposition 6.4-Lemma 6.8, Lemma 6.11-6.12]{Zh21c}.
\end{proof}

\begin{proposition}\label{Prop:Ackermann}
The Ackermann rules are sound in $X$.
\end{proposition}

\begin{proof}
The proof is similar to the soundness of the Ackermann rules in \cite[Lemma 4.2, 4.3, 8.4]{CoPa12}. We only prove it for the right-handed Ackermann rule, the left-handed Ackermann rule is similar. Without loss of generality we assume that $n=1$. By the discussion on page \pageref{condition:1:4:equivalence}, it suffices to show the following \emph{right-handed Ackermann lemma}:
\begin{lemma}[Right-Handed Ackermann Lemma]\label{Lemma:Right:Ackermann}
Let $\theta$ be a formula which does not contain $p$, let $\eta_i(p)$ (resp.\ $\iota_i(p)$) be positive (resp.\ negative) in $p$ for $1\leq i\leq m$, and let $\vec q$ (resp.\ $\vec \nomj$) be all the propositional variables (resp.\ nominals) occurring in $\eta_1(p), \ldots, \eta_m(p)$, $\iota_1(p), \ldots, \iota_m(p), \theta$ other than $p$. Then for all $\vec a\in P(X), \vec x\in X$, the following are equivalent:
\begin{enumerate}
\item
$V(\eta_i(\theta/p))\subseteq V(\iota_i(\theta/p))$ for $1\leq i\leq m$,  where $V(\vec q)=\vec a$, $V(\vec \nomj)=\vec{\{x\}}$,
\item
there exists $a_0\in P(X)$ such that $V'(\theta)\leq a_0$ and $V'(\eta_i(p))\subseteq V'(\iota_i(p))$ for $1\leq i\leq m$,  where $V'$ is the same as $V$ except that $V'(p)=a_0$.
\end{enumerate}
\end{lemma}

$\Rightarrow$: Take $V'$ such that $V'$ is the same as $V$ except that $V'(p)=V(\theta)$. Since $\theta$ does not contain $p$, it is easy to see that $V(\theta)=V'(\theta)$. Therefore we can take $a_0:=V(\theta)$ and get 2.

$\Leftarrow$: It is obvious by monotonicity.
\end{proof}

\section{Success}\label{Sec:Success}
In this section, we show the success of the algorithm on inductive quasi-inequalities. Notice that here we do not allow the input quasi-inequality to contain any black connectives $\Diamondblack,\blacksquare,{\boxdotb},{\diamdotb}$, but we allow ${\diamdot}$ and ${\boxdot}$ to occur in the input quasi-inequality.

\begin{definition}[Definite $(\Omega,\epsilon)$-inductive signed generation tree]
Given a dependence order $<_{\Omega}$, an order type $\epsilon$, $*\in\{-,+\}$, the signed generation tree $*\theta$ of the formula $\theta(p_1,\ldots, p_n)$ is \emph{definite $(\Omega,\epsilon)$-inductive} if there is no $+\lor,-\land$ occurring in the Skeleton part on an $\epsilon$-critical branch.
\end{definition}

\begin{lemma}\label{Lemma:Definite:inductive}
Given an input $(\Omega,\epsilon)$-inductive quasi-inequality $$\overline{\phi}\leq\overline{\psi}\ \&\ \overline{\gamma}\prec\overline{\delta}\ \Rightarrow\ \alpha\triangleleft\beta \mbox{ (where } \triangleleft\in\{\leq,\prec\}),$$
after the first stage, it is transformed into a quasi-inequality of the form $$\overline{\phi}\leq\overline{\psi}\ \&\ \overline{{\diamdot}\gamma}\leq\overline{\delta}\ \Rightarrow\ \overline{\alpha}\leq\overline{\beta}\ \&\ \overline{{\diamdot}\xi}\leq\overline{\chi},$$
and further of the form 
$$\overline{\phi}\leq\overline{\psi}\ \&\ \overline{{\diamdot}\gamma}\leq\overline{\delta}\ \Rightarrow\ \alpha\leq\beta$$
or
$$\overline{\phi}\leq\overline{\psi}\ \&\ \overline{{\diamdot}\gamma}\leq\overline{\delta}\ \Rightarrow\ {\diamdot}\xi\leq\chi,$$
where 
\begin{itemize}
\item each $\phi\leq\psi$ or ${\diamdot}\gamma\leq\delta$ is either $(\Omega, \epsilon)$-receiving or $(\Omega, \epsilon)$-solvable;
\item $+\alpha$ and $-\beta$ (resp.\ $+{\diamdot}\xi$ and $-\chi$) are definite $(\Omega, \epsilon)$-inductive signed generation trees;
\item each formula contains no black connectives.
\end{itemize}
\end{lemma}

\begin{proof}
It is easy to see that by applying the distribution rules, in each inequality $\theta\leq\eta$ or $\theta\prec\eta$, consider the signed generation trees $+\theta$ and $-\eta$, all occurrences of $+\lor$ and $-\land$ in the Skeleton part of an $\epsilon$-critical branch have been pushed up towards the root of the signed generation trees $+\theta$ and $-\eta$. Then by exhaustively applying the splitting rules, all such $+\lor$ and $-\land$ are eliminated. 

Since by applying the distribution rules, the splitting rules and the monotone/antitone variable elimination rules do not change the $(\Omega,\epsilon)$-inductiveness of a signed generation tree for $+\alpha$, $-\beta$ (resp.\ $+\xi$, $-\chi$), and do not change the property of being $(\Omega, \epsilon)$-receiving or $(\Omega, \epsilon)$-solvable for $\phi\leq\psi$ or $\gamma\prec\delta$, so in $$\overline{\phi}\leq\overline{\psi}\ \&\ \overline{{\diamdot}\gamma}\leq\overline{\delta}\ \Rightarrow\ \overline{\alpha}\leq\overline{\beta}\ \&\ \overline{{\diamdot}\xi}\leq\overline{\chi},$$ 
\begin{itemize}
\item each $\phi\leq\psi$ or ${\diamdot}\gamma\leq\delta$ is either $(\Omega, \epsilon)$-receiving or $(\Omega, \epsilon)$-solvable;
\item each pair of $+\alpha$ and $-\beta$ (resp.\ $+{\diamdot}\xi$ and $-\chi$) are definite $(\Omega, \epsilon)$-inductive signed generation trees.
\end{itemize}
It is easy to see that no formula contains black connectives since the input quasi-inequality contains no black connective and no rule in Stage 1 introduces them.
\end{proof}

\begin{lemma}
After first approximation, the input $(\Omega,\epsilon)$-inductive inequality $$\overline{\phi}\leq\overline{\psi}\ \&\ \overline{\gamma}\prec\overline{\delta}\ \Rightarrow\ \alpha\triangleleft\beta \mbox{ (where } \triangleleft\in\{\leq,\prec\})$$
is rewritten into the form 
$$\overline{\phi}\leq\overline{\psi}\ \&\ \overline{{\diamdot}\gamma}\leq\overline{\delta}\ \&\ \nomi_0\leq\alpha\ \&\ \beta\leq\neg\nomi_1\  \Rightarrow\ \nomi_0\leq\neg\nomi_1$$
or
$$\overline{\phi}\leq\overline{\psi}\ \&\ \overline{{\diamdot}\gamma}\leq\overline{\delta}\ \&\ \nomi_0\leq{\diamdot}\xi\ \&\ \chi\leq\neg\nomi_1\  \Rightarrow\ \nomi_0\leq\neg\nomi_1,$$
where 
\begin{itemize}
\item each $\phi\leq\psi$ or ${\diamdot}\gamma\leq\delta$ is either $(\Omega, \epsilon)$-receiving or $(\Omega, \epsilon)$-solvable;
\item $+\alpha$ and $-\beta$ (resp.\ $+{\diamdot}\xi$ and $-\chi$) are definite $(\Omega, \epsilon)$-inductive signed generation trees;
\item each formula contains no black connectives.
\end{itemize}
\end{lemma}

\begin{proof}
Straightforward.
\end{proof}

Now we proceed to Stage 2 and work with set of inequalities (i.e.\ systems).
$$\{\overline{\phi}\leq\overline{\psi}, \overline{{\diamdot}\gamma}\leq\overline{\delta}, \nomi_0\leq\alpha, \beta\leq\neg\nomi_1\}$$
or 
$$\{\overline{\phi}\leq\overline{\psi}, \overline{{\diamdot}\gamma}\leq\overline{\delta}, \nomi_0\leq{\diamdot}\xi, \chi\leq\neg\nomi_1\}.$$

\begin{lemma}\label{Lemma:Approximation:Quasi}
For each inequality $\nomi_0\leq\alpha$ or $\beta\leq\neg\nomi_1$ (here we merge the $\nomi_0\leq{\diamdot}\xi$ case into the $\nomi_0\leq\alpha$ case, and merge the $\chi\leq\neg\nomi_1$ case into the $\beta\leq\neg\nomi_1$ case), by applying the approximation rules, splitting rules and residuation rules involving $\neg$ exhaustively, we obtain a set of inequalities, each inequality of which contains no black connectives and belong to one of the three classes:
\begin{itemize}
\item it is pure;
\item it is $(\Omega,\epsilon)$-receiving;
\item it is of the form $\nomi\leq\alpha'$ or $\beta'\leq\neg\nomi$, where $+\alpha'$ and $-\beta'$ are $(\Omega,\epsilon)$-inductive and their $\epsilon$-critical branches contain only PIA nodes.
\end{itemize}
\end{lemma}

\begin{proof}
First of all, since the approximation rules, the splitting rules and the residuation rules involving $\neg$ do not introduce black connectives and $\alpha,\beta$ do not contain black connectives, so it is easy to see that the obtained inequalities do not contain black connectives.

Then we prove by induction on the complexity of $+\alpha$ and $-\beta$ to show that the obtained inequalities can be classified into three classes mentioned above:
\begin{itemize}
\item When $\alpha$ is $p$ and $\epsilon(p)=1$, then it is obvious that $+p$ is critical, $\nomi_{0}\leq\alpha$ belongs to the third class;

when $\alpha$ is $p$ and $\epsilon(p)=\partial$, then it is obvious that $+p$ is non-critical, $\nomi_{0}\leq\alpha$ belongs to the second class;

when $\beta$ is $p$ and $\epsilon(p)=\partial$, then it is obvious that $-p$ is critical, $\beta\leq\neg\nomi_{1}$ belongs to the third class;

when $\beta$ is $p$ and $\epsilon(p)=1$, then it is obvious that $-p$ is non-critical, $\beta\leq\neg\nomi_{1}$ belongs to the second class;

\item when $\alpha$ or $\beta$ is $\bot$ or $\top$, then it is already pure, so it belongs to the first class;

\item when $\alpha$ is $\neg\theta$, then since $+\alpha$ is definite $(\Omega,\epsilon)$-inductive, $-\theta$ is also definite $(\Omega,\epsilon)$-inductive, we first apply ($\neg$-Res-Right) to $\nomi_{0}\leq\neg\theta$ to obtain $\theta\leq\neg\nomi_{0}$, then we can apply the induction hypothesis to obtain the result;

when $\beta$ is $\neg\theta$, the proof is similar;

\item when $\alpha$ is $\Diamond\theta$, then since $+\alpha$ is definite $(\Omega,\epsilon)$-inductive, $+\theta$ is also definite $(\Omega,\epsilon)$-inductive, we first apply the approximation rule for $\Diamond$ to $\nomi_0\leq\Diamond\theta$ obtain $\nomi_{0}\leq\Diamond\nomj$ and $\nomj\leq\theta$. Now the first inequality belongs to the first class, and we can use the induction hypothesis to the second inequality, so we obtain the result;

when $\alpha$ is ${\diamdot}\theta$, $\beta$ is $\Box\theta$ or ${\boxdot}\theta$, the proof is similar;

\item when $\alpha$ is $\Box\theta$, then either $\alpha$ contains no critical branch or its critical branch is already a PIA branch; in the first case, $\nomi\leq\alpha$ is already in the second class; in the second case, it is already in the third class;

when $\alpha$ is ${\boxdot}\theta$, $\beta$ is $\Diamond\theta$ or ${\diamdot}\theta$, the proof is similar;

\item when $\alpha$ is $\theta\land\eta$, then since $+\alpha$ is definite $(\Omega,\epsilon)$-inductive, $+\theta,+\eta$ are also definite $(\Omega,\epsilon)$-inductive, we first apply the splitting rule for $\land$ to $\nomi_{0}\leq\theta\land\eta$ to obtain $\nomi_{0}\leq\theta$ and $\nomi_{0}\leq\eta$, and then we can use the induction hypothesis to obtain the result;

when $\beta$ is $\theta\lor\eta$, the proof is similar;

\item when $\alpha$ is $\theta\lor\eta$, since $+\alpha$ is definite $(\Omega,\epsilon)$-inductive, either $\alpha$ contains no critical branch (then $\nomi_0\leq\alpha$ belongs to the second class) or $+\lor$ is on a critical branch. Since $\Delta$-adjoint $+\lor$s are already eliminated in Stage 1, so $+\lor$ is an SRR node, so all aritical branches passing through $+\lor$ are PIA branches, so $\nomi_{0}\leq\alpha$ belongs to the third class;

when $\beta$ is $\theta\land\eta$, the proof is similar;

\item when $\alpha$ is $\theta\to\eta$, the proof is similar to the $\alpha=\theta\lor\eta$ case;

\item when $\beta$ is $\theta\to\eta$, then since $-\beta$ is definite $(\Omega,\epsilon)$-inductive, $+\theta,-\eta$ are also definite $(\Omega,\epsilon)$-inductive, we first apply the approximation rule for $\to$ to $\theta\to\eta\leq\neg\nomi_{1}$ to obtain $\nomi\leq\theta$, $\eta\leq\neg\nomj$ and $\nomi\to\neg\nomj\leq\neg\nomi_{1}$, then the last inequality belongs to the first class, and we can apply the induction hypothesis to $\nomi\leq\theta$ and $\eta\leq\neg\nomj$ obtain the result;

\item it is easy to see that $\alpha$ and $\beta$ cannot be of the form $\Diamondblack\theta$,$\blacksquare\theta$, ${\boxdotb}\theta$ or ${\diamdotb}\theta$.
\end{itemize}
\end{proof}

\begin{lemma}
For each $(\Omega,\epsilon)$-solvable inequality $\phi\leq\psi$ or ${\diamdot}\gamma\leq\delta$, by applying the residuation rules and splitting rules exhaustively, we obtain a set of inequalities, each inequality of which belong to one of the three classes:
\begin{itemize}
\item it is pure;
\item it is $(\Omega,\epsilon)$-receiving;
\item it is of the form $\theta\leq p$ when $\epsilon(p)=1$, or of the form $p\leq\theta$ when $\epsilon(p)=\partial$. In $\theta\leq p$, $-\theta$ is $\epsilon^{\partial}$-uniform, and in $p\leq\theta$, $+\theta$ is $\epsilon^{\partial}$-uniform. In addition, for all propositional variables $q$ in $\theta$, we have $q<_{\Omega} p$.
\end{itemize}
\end{lemma}

\begin{proof}
Since ${\diamdot}\gamma\leq\delta$ is the special case of $\phi\leq\psi$ when the outermost connective of $\phi$ is ${\diamdot}$, we only consider $\phi\leq\psi$ here.

Without loss of generality we consider the situation when 
\begin{itemize}
\item $-\phi$ is $\epsilon^{\partial}$-uniform;
\item $+\psi$ has $\epsilon$-critical branches and is $(\Omega, \epsilon)$-inductive, and all $\epsilon$-critical branches in $+\psi$ are PIA branches;
\item for all the $\epsilon$-critical branches in $+\psi$ ending with $p$, all propositional variables $q$ in $-\phi$, we have $q<_{\Omega} p$.
\end{itemize}
The situation where $+\psi$ is $\epsilon^{\partial}$-uniform etc. is symmetric.

We prove by induction on the complexity of $\psi$:
\begin{itemize}
\item when $\psi$ is $p$: $\phi\leq\psi$ belongs to the third class and $\epsilon(p)=1$;
\item when $\psi$ is $\bot$ or $\top$: it cannot be the case, since $+\psi$ contains $\epsilon$-critical branches;
\item when $\psi$ is $\neg\theta$: we can first apply the residuation rule for $\neg$ to $\phi\leq\neg\theta$ to obtain $\theta\leq\neg\phi$, and then we can apply the induction hypothesis for the case where
\begin{itemize}
\item $+\neg\phi$ is $\epsilon^{\partial}$-uniform;
\item $-\theta$ has $\epsilon$-critical branches and is $(\Omega, \epsilon)$-inductive, and all $\epsilon$-critical branches in $-\theta$ are PIA branches;
\item for all the $\epsilon$-critical branches in $-\theta$ ending with $p$, all propositional variables $q$ in $+\neg\phi$, we have $q<_{\Omega} p$;
\end{itemize}
\item when $\psi$ is $\Diamond\theta$: it cannot be the case, since $\Diamond$ is on an $\epsilon$-critical branch in $+\psi$ but it is not a PIA node;
\item when $\psi$ is ${\diamdot}\theta$: similar to the $\Diamond\theta$ case;
\item when $\psi$ is $\Box\theta$: we can first apply the residuation rule for $\Box$ to $\phi\leq\Box\theta$ to obtain $\Diamondblack\phi\leq\theta$, and then we can apply the induction hypothesis for the case where
\begin{itemize}
\item $-\Diamondblack\phi$ is $\epsilon^{\partial}$-uniform;
\item $+\theta$ has $\epsilon$-critical branches and is $(\Omega, \epsilon)$-inductive, and all $\epsilon$-critical branches in $+\theta$ are PIA branches;
\item for all the $\epsilon$-critical branches in $+\theta$ ending with $p$, all propositional variables $q$ in $-\Diamondblack\phi$, we have $q<_{\Omega} p$;
\end{itemize}
\item when $\psi$ is ${\boxdot}\theta$, the situation is similar to the $\Box\theta$ case;
\item when $\psi$ is $\theta\land\eta$, we first apply the splitting rule for $\land$ to $\phi\leq\theta\land\eta$ to obtain $\phi\leq\theta$ and $\phi\leq\eta$, then there are two possibilities, namely both of $+\theta$ and $+\eta$ have $\epsilon$-critical branches, and only one of $+\theta$ and $+\eta$ has $\epsilon$-critical branches;

for the first possibility, we can apply the induction hypothesis for the case where 
\begin{itemize}
\item $-\phi$ is $\epsilon^{\partial}$-uniform;
\item $+\theta$ and $+\eta$ have $\epsilon$-critical branches and are $(\Omega, \epsilon)$-inductive, and all $\epsilon$-critical branches in $+\theta$ and $+\eta$ are PIA branches;
\item for all the $\epsilon$-critical branches in $+\theta$ and $+\eta$ ending with $p$, all propositional variables $q$ in $-\phi$, we have $q<_{\Omega} p$;
\end{itemize}

for the second possibility, without loss of generality we assume that $+\theta$ has $\epsilon$-critical branches and $+\eta$ is $\epsilon^{\partial}$-uniform, then $\phi\leq\eta$ is $(\Omega, \epsilon)$-receiving, and we can apply the induction hypothesis to $\phi\leq\theta$ for the case where 
\begin{itemize}
\item $-\phi$ is $\epsilon^{\partial}$-uniform;
\item $+\theta$ have $\epsilon$-critical branches and are $(\Omega, \epsilon)$-inductive, and all $\epsilon$-critical branches in $+\theta$ are PIA branches;
\item for all the $\epsilon$-critical branches in $+\theta$ ending with $p$, all propositional variables $q$ in $-\phi$, we have $q<_{\Omega} p$;
\end{itemize}

\item when $\psi$ is $\theta\lor\eta$: since $\lor$ is an SRR node, only one of $\theta$ and $\eta$ contains $\epsilon$-critical branches (without loss of generality we assume that it is $\eta$). Then 
\begin{itemize}
\item $+\theta$ is $\epsilon^{\partial}$-uniform;
\item for each $p$ in an $\epsilon$-critical branch in $+\eta$, each $q$ that occurs in $+\theta$, we have $q<_{\Omega} p$;
\end{itemize}
Now we apply the residuation rule for $\lor$ to $\phi\leq\theta\lor\eta$ to obtain $\phi\land\neg\theta\leq\eta$, then we can apply the induction hypothesis for the case where 
\begin{itemize}
\item $-(\phi\land\neg\theta)$ is $\epsilon^{\partial}$-uniform;
\item $+\eta$ has $\epsilon$-critical branches and is $(\Omega,\epsilon)$-inductive, and all $\epsilon$-critical branches in $+\eta$ are PIA branches;
\item for all the $\epsilon$-critical branches in $+\eta$ ending with $p$, all propositional variables $q$ in $-(\phi\land\neg\theta)$, we have $q<_{\Omega} p$;
\end{itemize}
\item when $\psi$ is $\theta\to\eta$: the situation is similar to the $\theta\lor\eta$ case (in the sense of using one of the residuation rules for $\to$).
\end{itemize}
\end{proof}

\begin{lemma}
Inequalities of the form $\nomi\leq\alpha'$ or $\beta'\leq\neg\nomi$, where $+\alpha'$ and $-\beta'$ are $(\Omega,\epsilon)$-inductive and their $\epsilon$-critical branches contain only PIA nodes are  $(\Omega,\epsilon)$-solvable, and therefore the previous lemma is applicable to them.
\end{lemma}

\begin{proof}
Straightforward.
\end{proof}

Now the set of inequalities have three kinds of inequalities, namely 
\begin{itemize}
\item pure inequalities;
\item $(\Omega,\epsilon)$-receiving inequalities;
\item $\theta\leq p$ when $\epsilon(p)=1$, or $p\leq\theta$ when $\epsilon(p)=\partial$. In $\theta\leq p$, $-\theta$ is $\epsilon^{\partial}$-uniform, and in $p\leq\theta$, $+\theta$ is $\epsilon^{\partial}$-uniform. In addition, for all propositional variables $q$ in $\theta$, we have $q<_{\Omega} p$.
\end{itemize}

It is ready to repeatedly apply the Ackermann lemmas, and then all propositional variables are eliminated. Since after the monotone/antitone variable elimination rules in the first stage, each variable has both critical branches and non-critical branches in the corresponding signed generation trees, therefore in Stage 2, they have occurrences in both $(\Omega,\epsilon)$-receiving inequalities and inequalities of the form $\theta\leq p$ or $p\leq\theta$, so they all can be eliminated.

\begin{theorem}[Success Theorem]\label{Them:Success}
For all $(\Omega,\epsilon)$-inductive quasi-inequalities which do not contain black connectives, $\mathsf{ALBA}$ successfully reduce them to pure quasi-inequalities, and therefore can transform them into first-order correspondents.
\end{theorem}

\section{Canonicity}\label{Sec:Canonicity}

In the present section, we will prove that all inductive quasi-inequalities which do not contain black connectives or ${\diamdot},{\boxdot}$ (we call them \emph{restricted inductive quasi-inequalities}\label{Page:Restricted:Inductive:Quasi:Inequality}) are canonical, in the sense that if they are admissibly valid on a Stone space with two relations, then they are valid on the same space. The proof style is similar to \cite{CoPa12}.

\begin{definition}[Restricted inductive quasi-inequality]\label{Def:Restricted:Inductive:Quasi}
Given an order-type $\epsilon$ and a dependence order $<_{\Omega}$, we say that a quasi-inequality is $(\Omega,\epsilon)$-restricted inductive, if it is $(\Omega,\epsilon)$-inductive and does not contain nominals, ${\diamdot},{\boxdot}$ or black connectives $\Diamondblack,\blacksquare,{\boxdotb},{\diamdotb}$.
\end{definition}

\begin{definition}[Canonicity]
We say that a quasi-inequality is \emph{canonical} if whenever it is admissibly valid on a Stone space with two relations, it is also valid on it.
\end{definition}

By the duality theory of modal subordination algebras, the canonicity definition above is equivalent to the preservation under taking canonical extensions of modal subordination algebras.

\subsection{U-Shaped Argument}\label{SubSec:U:Shaped:Argument}

To prove the canonicity of inductive quasi-inequalities, we use the canonicity-via-correspondence argument, which is a variation of the standard U-shaped argument (cf.\ \cite{CoPa12}) represented in the diagram below. To go through the U-shaped argument, we show that \emph{topologically correct executions} of $\mathsf{ALBA}$ is sound with respect to the Stone spaces with two relations, while we replace validity (with respect to arbitrary valuations) by admissible validity (with respect to clopen/admissible valuations). 

\begin{center}
\begin{tabular}{l c l}\label{table:U:shape}
$X\Vdash_{Clop}\mathsf{Quasi}$ & &$X\Vdash_{P}\mathsf{Quasi}$\\
\ \ \ \ \ \ $\Updownarrow$ & &\ \ \ \ \ \ $\Updownarrow$\\
$X\vDash\mathsf{FO(Quasi)}$ &\ \ \ $\Leftrightarrow$ \ \ \ &$X\vDash\mathsf{FO(Quasi)}$
\end{tabular}
\end{center}

This argument starts from the top-left corner with the validity of the input quasi-inequality $\mathsf{Quasi}$ on $X$, then uses topologically correct executions of $\mathsf{ALBA}$ to transform the quasi-inequality into an equivalent set of quasi-inequalities as well as its first-order translation $\mathsf{FO(Quasi)}$. Since the validity of the first-order formulas does not depend on the admissible set, the bottom equivalence is obvious. The right half of the argument goes on the side of arbitrary valuations, the soundness of which was already shown in Section \ref{Sec:Soundness}.

We will focus on the equivalences of the left-arm, i.e.,\ the soundness of topologically correct executions of $\mathsf{ALBA}$ with respect to admissible valuations.

\subsection{Topological Correct Executions}

In the present subsection, we show that $\mathsf{ALBA}$ can be executed in a topologically correct way on restricted inductive quasi-inequalities (see Definition \ref{Def:Restricted:Inductive:Quasi}), which paves the way for the use of the topological Ackermann lemmas in the next subsection.

We first define syntactically closed and syntactically open formulas as follows:

\begin{definition}[Syntactically closed and open formulas]\label{Def:Closed:Open:Formulas}
\begin{enumerate}
\item A formula is \emph{syntactically closed} if all occurrences of nominals and $\Diamondblack,{\diamdot},{\diamdotb}$ are positive, and all occurrences of $\Boxblack,{\boxdot},{\boxdotb}$ are negative;
\item A formula is \emph{syntactically open} if all occurrences of nominals and $\Diamondblack,{\diamdot},{\diamdotb}$ are negative, and all occurrences of $\Boxblack,{\boxdot},{\boxdotb}$ are positive.
\end{enumerate}
\end{definition}

As is discussed in \cite[Section 9]{CoPa12}, the underlying idea of the definitions above is that given an admissible valuation, the truth set of a syntactically closed (resp.\ open) formula is always a closed (resp.\ open) subset in $X$ (see Definition \ref{Def:StRR'}).

We recall the right-handed and left-handed Ackermann rule:

The right-handed Ackermann rule:

The system 
$\left\{ \begin{array}{ll}
\theta_1\leq p \\
\vdots\\
\theta_n\leq p \\
\eta_1\leq\iota_1\\
\vdots\\
\eta_m\leq\iota_m\\
\end{array} \right.$ 
is replaced by 
$\left\{ \begin{array}{ll}
\eta_1((\theta_1\lor\ldots\lor\theta_n)/p)\leq\iota_1((\theta_1\lor\ldots\lor\theta_n)/p) \\
\vdots\\
\eta_m((\theta_1\lor\ldots\lor\theta_n)/p)\leq\iota_m((\theta_1\lor\ldots\lor\theta_n)/p) \\
\end{array} \right.$
where:
\begin{enumerate}
\item $p$ does not occur in $\theta_1, \ldots, \theta_n$;
\item Each $\eta_i$ is positive, and each $\iota_i$ negative in $p$, for $1\leq i\leq m$.
\end{enumerate}
The left-handed Ackermann rule:

The system
$\left\{ \begin{array}{ll}
p\leq\theta_1 \\
\vdots\\
p\leq\theta_n \\
\eta_1\leq\iota_1\\
\vdots\\
\eta_m\leq\iota_m\\
\end{array} \right.$
is replaced by
$\left\{ \begin{array}{ll}
\eta_1((\theta_1\land\ldots\land\theta_n)/p)\leq\iota_1((\theta_1\land\ldots\land\theta_n)/p) \\
\vdots\\
\eta_m((\theta_1\land\ldots\land\theta_n)/p)\leq\iota_m((\theta_1\land\ldots\land\theta_n)/p) \\
\end{array} \right.$
where:
\begin{enumerate}
\item $p$ does not occur in $\theta_1, \ldots, \theta_n$;
\item Each $\eta_i$ is negative, and each $\iota_i$ positive in $p$, for $1\leq i\leq m$.
\end{enumerate}

\begin{definition}[Topologically Correct Executions]\label{Def:Topological:Correct:Executions}
\begin{itemize}
\item We call an execution of the Ackermann rule \emph{topologically correct}, if for each non-pure inequality in the system, the left-hand side is syntactically closed, and the right-hand side is syntactically open.
\item We call an execution of $\mathsf{ALBA}$ \emph{topologically correct}, if each execution of right-handed and left-handed Ackermann lemma is topologically correct.
\end{itemize}
\end{definition}

\begin{theorem}
Given a restricted inductive quasi-inequality as input, ALBA can topologically correctly execute on it.
\end{theorem}

\begin{proof}

We basically follow the success proof in Section \ref{Sec:Success}, while pay attention to some details (on the topological correctness) of the execution. From the proof of Lemma \ref{Lemma:Definite:inductive}, given an $(\Omega, \epsilon)$-restricted inductive quasi-inequality
$$\overline{\phi}\leq\overline{\psi}\ \&\ \overline{\gamma}\prec\overline{\delta}\ \Rightarrow\ \alpha\triangleleft\beta \mbox{ (where } \triangleleft\in\{\leq,\prec\}),$$
after the first stage, it is transformed into a set of quasi-inequalities of the form 
$$\overline{\phi}\leq\overline{\psi}\ \&\ \overline{{\diamdot}\gamma}\leq\overline{\delta}\ \Rightarrow\ \alpha\leq\beta$$
or
$$\overline{\phi}\leq\overline{\psi}\ \&\ \overline{{\diamdot}\gamma}\leq\overline{\delta}\ \Rightarrow\ {\diamdot}\xi\leq\chi,$$
where 
\begin{itemize}
\item each $\phi\leq\psi$ or ${\diamdot}\gamma\leq\delta$ is either $(\Omega, \epsilon)$-receiving or $(\Omega, \epsilon)$-solvable;
\item $+\alpha$ and $-\beta$ (resp.\ $+{\diamdot}\xi$ and $-\chi$) are definite $(\Omega, \epsilon)$-inductive signed generation trees;
\item each formula contains no black connectives.
\end{itemize}

Since for the rules in the Stage 1, only the subordination rewritting rule introduce ${\diamdot}$ and ${\diamdot}$ only occurs in the outermost level, it is easy to see that $\phi,\psi,\gamma,\delta,\alpha,\beta,\xi,\chi$ do not contain ${\diamdot},{\boxdot},{\diamdotb},{\boxdotb}$, so they are both syntactically closed and open.

Therefore, if we apply the first-approximation rule to the quasi-inequality, we obtain 
$$\overline{\phi}\leq\overline{\psi}\ \&\ \overline{{\diamdot}\gamma}\leq\overline{\delta}\ \&\ \nomi_0\leq\alpha\ \&\ \beta\leq\neg\nomi_1\  \Rightarrow\ \nomi_0\leq\neg\nomi_1$$
or
$$\overline{\phi}\leq\overline{\psi}\ \&\ \overline{{\diamdot}\gamma}\leq\overline{\delta}\ \&\ \nomi_0\leq{\diamdot}\xi\ \&\ \chi\leq\neg\nomi_1\  \Rightarrow\ \nomi_0\leq\neg\nomi_1,$$
where 
\begin{itemize}
\item each $\phi\leq\psi$ or ${\diamdot}\gamma\leq\delta$ is either $(\Omega, \epsilon)$-receiving or $(\Omega, \epsilon)$-solvable;
\item $+\alpha$ and $-\beta$ (resp.\ $+{\diamdot}\xi$ and $-\chi$) are definite $(\Omega, \epsilon)$-inductive signed generation trees;
\item each of $\phi,\psi,\gamma,\delta,\alpha,\beta,\xi,\chi$ is both syntactically closed and open.
\end{itemize}
For the second case, we apply the approximation rule for ${\diamdot}$ and get
$$\overline{\phi}\leq\overline{\psi}\ \&\ \overline{{\diamdot}\gamma}\leq\overline{\delta}\ \&\ \nomi_0\leq{\diamdot}\nomj\ \&\ \nomj\leq\xi\ \&\ \chi\leq\neg\nomi_1\  \Rightarrow\ \nomi_0\leq\neg\nomi_1.$$

Now in both cases, we have a quasi-inequality of the following form:

$$\overline{\phi}\leq\overline{\psi}\ \&\ \nomi\leq\alpha'\ \&\ \beta\leq\neg\nomi_1\ \&\ \mathsf{Pure}\  \Rightarrow\ \nomi_0\leq\neg\nomi_1$$
where 
\begin{itemize}
\item each $\phi\leq\psi$ is either $(\Omega, \epsilon)$-receiving or $(\Omega, \epsilon)$-solvable;
\item $+\alpha'$ and $-\beta$ are definite $(\Omega, \epsilon)$-inductive signed generation trees;
\item each of $\phi,\beta$ is syntactically closed and each of $\psi,\alpha'$ is syntactically open.
\end{itemize}

Now we can easily check that the following property holds for the quasi-inequality:\\

In the antecedent of the quasi-inequality, in each non-pure inequality, the left-hand side is syntactically closed and the right-hand side is syntactically open.\\

It is easy to check that for each rule in Stage 2, it does not break this property, so for each execution of the Ackermann rule, it is topologically correct, and after the execution of the Ackermann rule, it still satisfies the property stated above. Therefore, the execution of $\mathsf{ALBA}$ is topologically correct.
\end{proof}

\subsection{Topological Ackermann Lemmas}\label{Subsec:Topological:Ackermann:Lemmas}

In the present section we will prove the soundness of the algorithm $\mathsf{ALBA}$ with respect to admissible validity on Stone spaces with two relations. Indeed, similar to other semantic settings (see e.g.,\ \cite{CoPa12}), the soundness proof on the admissible valuation side goes similar to that of the arbitrary valuation side (i.e.,\ Theorem \ref{Thm:Soundness}), and for every rule except for the Ackermann rules, the proof goes without modification. Thus we will only focus on the Ackermann rules here, which are justified by the topological Ackermann lemmas given below. The proof is similar to \cite{CoPa12}, so we only state the lemmas without giving proof details of the lemmas.

For the Ackermann rules, the soundness proof with respect to arbitrary valuations is justified by the Ackermann lemmas (see Lemma \ref{Prop:Ackermann}). As discussed in the literature (e.g.,\ \cite[Section 9]{CoPa12}), the soundness proof of the Ackermann rules, namely the Ackermann lemmas, cannot be straightforwardly adapted to the setting of admissible valuations, since formulas in the expanded language might be interpreted as non-clopen subsets even if all the propositional variables are interpreted as clopen subsets in $X$. Thus by taking $a_0=\bigvee_{i}V(\theta_i)$, we cannot guarantee that $a_0$ is a clopen subset of $X$. Therefore, some adaptations are necessary based on the syntactic shape of the formulas, which are defined in the previous subsection.

Now we state the topological Ackermann lemmas without proofs. Similar proofs can be found in \cite[Section 9]{CoPa12}.

\begin{lemma}[Right-handed topological Ackermann lemma]\label{aRight:Ack} 
Let $\theta$ be a syntactically closed formula which does not contain $p$, let $\eta_i(p)$ (resp.\ $\iota_i(p)$) be syntactically closed (resp.\ open) and positive (resp.\ negative) in $p$ for $1\leq i\leq m$, and let $\vec q$ (resp.\ $\vec \nomj$) be all the propositional variables (resp.\ nominals) occurring in $\eta_1(p), \ldots, \eta_m(p)$, $\iota_1(p), \ldots, \iota_m(p), \theta$ other than $p$. Then for all $\vec a\in Clop(X), \vec x\in X$, the following are equivalent:
\begin{enumerate}
\item
$V(\eta_i(\theta/p))\subseteq V(\iota_i(\theta/p))$ for $1\leq i\leq m$,  where $V(\vec q)=\vec a$, $V(\vec \nomj)=\{\vec x\}$,
\item
there exists $a_0\in Clop(X)$ such that $V'(\theta)\leq a_0$ and $V'(\eta_i(p))\subseteq V'(\iota_i(p))$ for $1\leq i\leq m$,  where $V'$ is the same as $V$ except that $V'(p)=a_0$.
\end{enumerate}
\end{lemma}

\begin{lemma}[Left-handed topological Ackermann lemma]\label{aLeft:Ack}
Let $\theta$ be a syntactically open formula which does not contain $p$, let $\eta_i(p)$ (resp.\ $\iota_i(p)$) be syntactically closed (resp.\ open) and negative (resp.\ positive) in $p$ for $1\leq i\leq m$, and let $\vec q$ (resp.\ $\vec \nomj$) be all the propositional variables (resp.\ nominals) occurring in $\eta_1(p), \ldots, \eta_m(p)$, $\iota_1(p), \ldots, \iota_m(p), \theta$ other than $p$. Then for all $\vec a\in Clop(X), \vec x\in X$, the following are equivalent:
\begin{enumerate}
\item
$V(\eta_i(\theta/p))\subseteq V(\iota_i(\theta/p))$ for $1\leq i\leq m$,  where $V(\vec q)=\vec a$, $V(\vec \nomj)=\{\vec x\}$,
\item
there exists $a_0\in Clop(X)$ such that $a_0\leq V'(\theta)$ and $V'(\eta_i(p))\subseteq V'(\iota_i(p))$ for $1\leq i\leq m$,  where $V'$ is the same as $V$ except that $V'(p)=a_0$.
\end{enumerate}
\end{lemma}

The proof of the topological Ackermann lemmas is similar to \cite[Section B]{PaSoZh15}, and therefore is omitted. We only show how these lemmas justify the soundness of the topologically correct executions of the Ackermann rules.

\begin{proposition}
The topologically correct executions of the Ackermann rules are sound with respect to the admissible valuations.
\end{proposition}

\begin{proof}
Here we prove it for the right-handed Ackermann lemma, and the left-handed Ackermann lemma is similar. Without loss of generality we assume that $n=1$, then for a topological correct execution of this rule, the system $S$ before execution is of the following shape:
\begin{center}
$\left\{ \begin{array}{ll}
\theta\leq p \\
\eta_1\leq\iota_1\\
\vdots\\
\eta_m\leq\iota_m\\
\mathsf{Pure}
\end{array} \right.$ 
\end{center}
where:
\begin{enumerate}
\item $p$ does not occur in $\theta$ and $\theta$ is syntactically closed;
\item Each $\eta_i$ is syntactically closed and positive in $p$, and each $\iota_i$ is syntactically open and negative in $p$, for $1\leq i\leq m$;
\item $\mathsf{Pure}$ is pure;
\end{enumerate}

the system $S'$ after the execution is of the following shape:
\begin{center}
$\left\{ \begin{array}{ll}
\eta_1(\theta/p)\leq\iota_1(\theta/p) \\
\vdots\\
\eta_m(\theta/p)\leq\iota_m(\theta/p) \\
\mathsf{Pure}
\end{array} \right.$
\end{center}
Similar to the discussion of page \pageref{condition:1:4:equivalence}, it suffices to show the following:

\begin{enumerate}
\item For any $X$, any admissible valuation $V_0$, if $X,V_0\Vdash S$, then there is an admissible valuation $V_1$ such that $V_1(\nomi_0)=V_0(\nomi_0)$, $V_1(\nomi_1)=V_0(\nomi_1)$ and $X,V_1\Vdash S'$;

\item For any $X$, any admissible valuation $V_1$, if $X,V_1\Vdash S'$, then there is an admissible valuation $V_0$ such that $V_0(\nomi_0)=V_1(\nomi_0)$, $V_0(\nomi_1)=V_1(\nomi_1)$ and $X,V_0\Vdash S$.
\end{enumerate}

\begin{enumerate}
\item[Proof of 1.]Consider any $X$ and admissible valuation $V_0$ on $X$. If $X,V_0\Vdash S$, then take $a_0=V_0(p)$, we have:
\begin{center}
$\left\{ \begin{array}{ll}
V_0(\theta)\leq V_0(p)=a_0 \\
V_0(\eta_i)\leq V_0(\iota_i)\mbox{ for }1\leq i\leq m\\
\mathsf{Pure}\mbox{ is true under }V_0,
\end{array} \right.$ 
\end{center}
so condition 2. holds, by the right-handed Ackermann lemma, condition 1 holds, so $V_1(\eta_i(\theta/p))\subseteq V_1(\iota_i(\theta/p))$ for $1\leq i\leq m$ and some admissible valuation $V_1$ which differs from $V_0$ at most at $p$. It is easy to see that $\mathsf{Pure}$ is true under $V_0$ iff it is true under $V_1$, so $X,V_1\Vdash S'$;

\item[Proof of 2.]Consider any $X$ and admissible valuation $V_1$ on $X$. If $X,V_1\Vdash S'$, then we have:
\begin{center}
$\left\{ \begin{array}{ll}
V_1(\eta_i(\theta/p))\leq V_1(\iota_i(\theta/p))\mbox{ for }1\leq i\leq m\\
\mathsf{Pure}\mbox{ is true under }V_1,
\end{array} \right.$
\end{center}
so condition 1. holds, by the right-handed Ackermann lemma, condition 2 holds, so there exists $a_0\in Clop(X)$ such that $V_0(\theta)\leq a_0$ and $V_0(\eta_i(p))\subseteq V_0(\iota_i(p))$ for $1\leq i\leq m$,  where $V_0$ is the same as $V_1$ except that $V_0(p)=a_0$. It is easy to see that  $\mathsf{Pure}$ is true under $V_1$ iff it is true under $V_0$,
\begin{center}
$\left\{ \begin{array}{ll}
V_0(\theta)\leq V_0(p)=a_0 \\
V_0(\eta_i)\leq V_0(\iota_i)\mbox{ for }1\leq i\leq m\\
\mathsf{Pure}\mbox{ is true under }V_0,
\end{array} \right.$ 
\end{center}
so  $X,V_0\Vdash S$.
\end{enumerate}
\end{proof}

\section{Examples}\label{Sec:Examples}
\begin{example}
Given a modal subordination algebra $(B,\prec,\Diamond)$, it is proximity preserving iff its dual Stone space with two relations $(X,\tau,R,R')$ satisfies the following interaction condition:
$$\forall w(R[R'^{-1}(w)]\subseteq R'^{-1}[R(w)]),\mbox{ i.e.\ }$$
$$\mbox{for all }u,v,w\in X\mbox{ such that }vR'w\mbox{ and }vRu,\mbox{ there exists a }t\in X\mbox{ such that }wRt\mbox{ and }uR't.$$
\end{example}

\begin{proof}
It suffices to show that $$(B,\prec,\Diamond)\Vdash(a\prec b)\Rightarrow(\Diamond a\prec\Diamond b)$$ iff its dual $$(X,\tau,R,R')\vDash\forall w(R[R'^{-1}(w)]\subseteq R'^{-1}[R(w)]).$$
Indeed,
\begin{center}
\begin{tabular}{c l}
 &$(B,\prec,\Diamond)\Vdash(a\prec b)\Rightarrow(\Diamond a\prec\Diamond b)$\\
iff &$(X,\tau,R,R')\Vdash_{Clop}\forall a\forall b((a\prec b)\Rightarrow(\Diamond a\prec\Diamond b))$\\
iff &$(X,\tau,R,R')\Vdash_{Clop}\forall a\forall b({\diamdot}a\leq b\Rightarrow {\diamdot}\Diamond a\leq\Diamond b)$\\
iff &$(X,\tau,R,R')\Vdash_{Clop}\forall a\forall b\forall\nomi\forall\nomj({\diamdot}a\leq b\ \&\ \nomi\leq{\diamdot}\Diamond a\ \&\ \Diamond b\leq\neg\nomj\ \Rightarrow\ \nomi\leq\neg\nomj)$\\
iff &$(X,\tau,R,R')\Vdash_{Clop}\forall a\forall b\forall\nomi\forall\nomj\forall\nomk({\diamdot}a\leq b\ \&\ \nomi\leq{\diamdot}\Diamond \nomk\ \&\ \nomk\leq a\ \&\ \Diamond b\leq\neg\nomj\ \Rightarrow\ \nomi\leq\neg\nomj)$\\
iff &$(X,\tau,R,R')\Vdash_{Clop}\forall b\forall\nomi\forall\nomj\forall\nomk({\diamdot}\nomk\leq b\ \&\ \nomi\leq{\diamdot}\Diamond \nomk\ \&\ \Diamond b\leq\neg\nomj\ \Rightarrow\ \nomi\leq\neg\nomj)$\\
iff &$(X,\tau,R,R')\Vdash_{Clop}\forall\nomi\forall\nomj\forall \nomk(\nomi\leq{\diamdot}\Diamond \nomk\ \&\ \Diamond {\diamdot}\nomk\leq\neg\nomj\ \Rightarrow\ \nomi\leq\neg\nomj)$\\
iff &$(X,\tau,R,R')\Vdash_{Clop}\forall \nomk({\diamdot}\Diamond\nomk\leq\Diamond {\diamdot}\nomk)$\\
iff &$(X,\tau,R,R')\vDash\forall w(R[R'^{-1}(w)]\subseteq R'^{-1}[R(w)])$.
\end{tabular}
\end{center}
\end{proof}

\begin{example}[See Lemma 2.1.12 in \cite{Sa16}]
Given a modal subordination algebra $(B,\prec,\Diamond)$ and its dual Stone space with two relations $(X,\tau,R,R')$,
\begin{enumerate}
\item $(B,\prec,\Diamond)\Vdash\forall a\forall b(a\prec b\Rightarrow a\leq b)$ iff $R$ is reflexive;
\item $(B,\prec,\Diamond)\Vdash\forall a\forall b(a\prec b\Rightarrow \neg b\prec\neg a)$ iff $R$ is symmetric.
\end{enumerate}
\end{example}

\begin{proof}
\begin{enumerate}
\item 
It suffices to show that $$(B,\prec,\Diamond)\Vdash a\prec b\Rightarrow a\leq b$$ iff its dual $$(X,\tau,R,R')\vDash\forall wRww.$$
Indeed,
\begin{center}
\begin{tabular}{c l}
 &$(B,\prec,\Diamond)\Vdash a\prec b\ \Rightarrow\ a\leq b$\\
iff &$(X,\tau,R,R')\Vdash_{Clop}\forall a\forall b(a\prec b\ \Rightarrow\ a\leq b)$\\
iff &$(X,\tau,R,R')\Vdash_{Clop}\forall a\forall b({\diamdot}a\leq b\ \Rightarrow\ a\leq b)$\\
iff &$(X,\tau,R,R')\Vdash_{Clop}\forall a\forall b\forall\nomi\forall\nomj({\diamdot}a\leq b\ \&\ \nomi\leq a\ \&\ b\leq\neg\nomj\ \Rightarrow\ \nomi\leq\neg\nomj)$\\
iff &$(X,\tau,R,R')\Vdash_{Clop}\forall b\forall\nomi\forall\nomj({\diamdot}\nomi\leq b\ \&\ b\leq\neg\nomj\ \Rightarrow\ \nomi\leq\neg\nomj)$\\
iff &$(X,\tau,R,R')\Vdash_{Clop}\forall\nomi\forall\nomj({\diamdot}\nomi\leq\neg\nomj\ \Rightarrow\ \nomi\leq\neg\nomj)$\\
iff &$(X,\tau,R,R')\Vdash_{Clop}\forall\nomi(\nomi\leq{\diamdot}\nomi)$\\
iff &$(X,\tau,R,R')\vDash\forall w(w\in R[w])$\\
iff &$(X,\tau,R,R')\vDash\forall wRww$.
\end{tabular}
\end{center}
\item 
It suffices to show that $$(B,\prec,\Diamond)\Vdash \forall a\forall b(a\prec b\Rightarrow \neg b\prec\neg a)$$ iff its dual $$(X,\tau,R,R')\vDash\forall w\forall v(Rvw\to Rwv).$$
Indeed,
\begin{center}
\begin{tabular}{c l}
 &$(B,\prec,\Diamond)\Vdash a\prec b\ \Rightarrow\ \neg b\prec\neg a$\\
iff &$(X,\tau,R,R')\Vdash_{Clop}\forall a\forall b(a\prec b\ \Rightarrow\ \neg b\prec\neg a)$\\
iff &$(X,\tau,R,R')\Vdash_{Clop}\forall a\forall b({\diamdot}a\leq b\ \Rightarrow\ {\diamdot}\neg b\leq\neg a)$\\
iff &$(X,\tau,R,R')\Vdash_{Clop}\forall a\forall b\forall\nomi\forall\nomj({\diamdot}a\leq b\ \&\ \nomi\leq {\diamdot}\neg b\ \&\ \neg a\leq\neg\nomj\ \Rightarrow\ \nomi\leq\neg\nomj)$\\
iff &$(X,\tau,R,R')\Vdash_{Clop}\forall a\forall b\forall\nomi\forall\nomj\forall\nomk({\diamdot}a\leq b\ \&\ \nomi\leq {\diamdot}\nomk\ \&\ \nomk\leq\neg b\ \&\ \neg a\leq\neg\nomj\ \Rightarrow\ \nomi\leq\neg\nomj)$\\
iff &$(X,\tau,R,R')\Vdash_{Clop}\forall a\forall b\forall\nomi\forall\nomj\forall\nomk({\diamdot}a\leq b\ \&\ \nomi\leq {\diamdot}\nomk\ \&\ \nomk\leq\neg b\ \&\ \nomj\leq a\ \Rightarrow\ \nomi\leq\neg\nomj)$\\
iff &$(X,\tau,R,R')\Vdash_{Clop}\forall b\forall\nomi\forall\nomj\forall\nomk({\diamdot}\nomj\leq b\ \&\ \nomi\leq {\diamdot}\nomk\ \&\ \nomk\leq\neg b\ \Rightarrow\ \nomi\leq\neg\nomj)$\\
iff &$(X,\tau,R,R')\Vdash_{Clop}\forall\nomi\forall\nomj\forall\nomk(\nomi\leq {\diamdot}\nomk\ \&\ \nomk\leq\neg {\diamdot}\nomj\ \Rightarrow\ \nomi\leq\neg\nomj)$\\
iff &$(X,\tau,R,R')\Vdash_{Clop}\forall\nomi\forall\nomj(\nomi\leq {\diamdot}\neg{\diamdot}\nomj\ \Rightarrow\ \nomi\leq\neg\nomj)$\\
iff &$(X,\tau,R,R')\Vdash_{Clop}\forall\nomj({\diamdot}\neg{\diamdot}\nomj\leq\neg\nomj)$\\
iff &$(X,\tau,R,R')\Vdash_{Clop}\forall\nomj(\nomj\leq\neg{\diamdot}\neg{\diamdot}\nomj)$\\
iff &$(X,\tau,R,R')\Vdash_{Clop}\forall\nomj(\nomj\leq{\boxdot}{\diamdot}\nomj)$\\
iff &$(X,\tau,R,R')\Vdash_{Clop}\forall\nomj({\diamdotb}\nomj\leq{\diamdot}\nomj)$\\
iff &$(X,\tau,R,R')\vDash\forall w(R^{-1}[w]\subseteq R[w])$\\
iff &$(X,\tau,R,R')\vDash\forall w\forall v(Rvw\to Rwv)$.
\end{tabular}
\end{center}
\end{enumerate}
\end{proof}

\section*{Part II: Correspondence and Canonicity for $\Pi_2$-statements}\label{Section:Pi_2}

In this part, we focus on \emph{$\Pi_2$-statements} of the form $\overline{\phi}\triangleleft\overline{\psi}\  \Rightarrow\ \exists\vec{q}(\overline{\gamma}\triangleleft\overline{\delta})$,\footnote{Here we call it a $\Pi_2$-statement is because the statement is implicitly universally quantified when talking about its validity, therefore it is of the $\forall\exists$-form.} and develop their correspondence and canonicity theory. We will define inductive $\Pi_2$-statements and restricted inductive $\Pi_2$-statements (Section \ref{Subsec:Inductive:Pi_2}), give a version of the algorithm $\mathsf{ALBA}^{\Pi_2}$ for $\Pi_2$-statements (Section \ref{SubSec:ALBA:Pi_2}), state its soundness with respect to arbitrary valuations (Section \ref{Subsec:Soundness:Pi_2}), success on inductive and restricted inductive $\Pi_2$-statements (Section \ref{Subsec:Success:Pi_2}), and canonicity of restricted inductive $\Pi_2$-statements (Section \ref{Subsec:Canonicity:Pi_2}).

\section{Syntax and Semantics}\label{Subsec:Syn:Sem:Pi_2}

For $\Pi_2$-statements, the special feature of its syntax is the propositional quantifiers of the form $\exists q$. The semantics of the propositional quantifiers is already different between admissible valuations and arbitrary valuations. The semantics of propositional quantifiers are given as follows: for any Stone space with two relations $X$,

\begin{itemize}
\item for any admissible valuation $V_0$, $(X,V_0)\Vdash\exists q\exists\vec{q}(\overline{\gamma}\triangleleft\overline{\delta})$ iff there exist $A\in Clop(X)$ such that $X,(V_0)^{q}_{A}\Vdash\exists\vec{q}(\overline{\gamma}\triangleleft\overline{\delta})$;
\item for any arbitrary valuation $V_1$, $(X,V_1)\Vdash\exists q\exists\vec{q}(\overline{\gamma}\triangleleft\overline{\delta})$ iff there exist $A\in P(X)$ such that $X,(V_1)^{q}_{A}\Vdash\exists\vec{q}(\overline{\gamma}\triangleleft\overline{\delta})$.
\end{itemize}

For the semantics of $\Pi_2$-statements, the global satisfaction relation for $\Pi_2$-statements is defined differently between admissible valuations and arbitrary valuations. It is given as follows:

\begin{definition}
Given a $\Pi_2$-statement $\overline{\phi}\triangleleft\overline{\psi}\  \Rightarrow\ \exists\vec{q}(\overline{\gamma}\triangleleft\overline{\delta})$, a Stone space with two relations $X$, an admissible valuation $V_0$, an arbitrary valuation $V_1$, 

\begin{itemize}
\item $X,V_0\Vdash\overline{\phi}\triangleleft\overline{\psi}\  \Rightarrow\ \exists\vec{q}(\overline{\gamma}\triangleleft\overline{\delta})$ iff whenever $X,V_0\Vdash\overline{\phi}\triangleleft\overline{\psi}$, there exist $\vec{A}\in Clop(X)$ such that $X,(V_0)^{\vec{q}}_{\vec{A}}\Vdash\overline{\gamma}\triangleleft\overline{\delta}$;

\item $X,V_1\Vdash\overline{\phi}\triangleleft\overline{\psi}\  \Rightarrow\ \exists\vec{q}(\overline{\gamma}\triangleleft\overline{\delta})$ iff whenever $X,V_1\Vdash\overline{\phi}\triangleleft\overline{\psi}$, there exist $\vec{A}\in P(X)$ such that $X,(V_1)^{\vec{q}}_{\vec{A}}\Vdash\overline{\gamma}\triangleleft\overline{\delta}$.

\item $\overline{\phi}\triangleleft\overline{\psi}\  \Rightarrow\ \exists\vec{q}(\overline{\gamma}\triangleleft\overline{\delta})$ is admissibly valid in $X$, i.e.\ $X\Vdash_{Clop}\overline{\phi}\triangleleft\overline{\psi}\  \Rightarrow\ \exists\vec{q}(\overline{\gamma}\triangleleft\overline{\delta})$ iff $X,V_0\Vdash\overline{\phi}\triangleleft\overline{\psi}\  \Rightarrow\ \exists\vec{q}(\overline{\gamma}\triangleleft\overline{\delta})$ for all admissible valuations $V_0$;

\item $\overline{\phi}\triangleleft\overline{\psi}\  \Rightarrow\ \exists\vec{q}(\overline{\gamma}\triangleleft\overline{\delta})$ is valid in $X$, i.e.\ $X\Vdash_{P}\overline{\phi}\triangleleft\overline{\psi}\  \Rightarrow\ \exists\vec{q}(\overline{\gamma}\triangleleft\overline{\delta})$ iff $X,V_1\Vdash\overline{\phi}\triangleleft\overline{\psi}\  \Rightarrow\ \exists\vec{q}(\overline{\gamma}\triangleleft\overline{\delta})$ for all arbitrary valuations $V_1$.
\end{itemize}
\end{definition}

\section{Inductive and Restricted Inductive $\Pi_2$-Statements}\label{Subsec:Inductive:Pi_2}

In the present section, we define inductive and restricted inductive $\Pi_2$-statements of the shape $$\overline{\phi}\leq\overline{\psi}\ \&\ \overline{\chi}\prec\overline{\xi}\Rightarrow\exists\vec{q}(\overline{\alpha}\leq\overline{\beta}\ \&\ \overline{\gamma}\prec\overline{\delta})$$

where each of $\phi,\psi,\chi,\xi,\alpha,\beta,\gamma,\delta$ has certain syntactic restrictions.

\subsection{Intuitive ideas behind the definition}\label{SubSubSec:Intuition:Pi_2}
\begin{enumerate}
\item The basic idea behind defining inductive $\Pi_2$-statements is as follows: We divide the algorithm into two parts, the first part of the algorithm equivalently transforms $\exists\vec{q}(\overline{\alpha}\leq\overline{\beta}\ \&\ \overline{\gamma}\prec\overline{\delta})$ part into a meta-conjunction of inequalities $\mathsf{MetaConIneq}$, such that 
$$\overline{\phi}\leq\overline{\psi}\ \&\ \overline{\chi}\prec\overline{\xi}\Rightarrow\ \mathsf{MetaConIneq}$$
is an inductive quasi-inequality, and we can then use the second part of the algorithm (which is the same as the algorithm $\mathsf{ALBA}$ for quasi-inequalities) to transform it into a set of pure quasi-inequalities. The first part of the algorithm eliminate all the existential propositional quantifiers by application of Ackermann rules, but this time we do not have nominals in the elimination of existential propositional quantifiers like in the elimination of the (implicit) universal propositional quantifiers (which is what will be done in the second part of the algorithm).

\item For the elimination of the existential propositional quantifiers, we again divide the set of inequalities in $$\exists\vec{q}(\overline{\alpha}\leq\overline{\beta}\ \&\ \overline{\gamma}\prec\overline{\delta})$$ into two parts:

\begin{itemize}
\item One part is the solvable inequalities, which will finally be transformed into inequalities of the form $\theta\leq q$ or $q\leq\theta$. The solvable inequalities will be used to compute the minimal valuations of the propositional variables in $\exists\vec{q}$.
\item The other part is the receiving inequalities, which will be used to receive minimal valuations.
\end{itemize}

\item For the canonicity proof, we again need to guarantee certain syntactic topological properties, which will be further explained in the canonicity proof. These restrictions lead to the definition of restricted inductive $\Pi_2$-statements.
\end{enumerate}

\subsection{The Definition of Inductive $\Pi_2$-Statements}

In this section, we will define inductive $\Pi_2$-statements in the language without black connectives $\Diamondblack,\blacksquare,{\boxdotb},{\diamdotb}$ or nominals, which will be shown to have first-order correspondents over arbitrary valuations on Stone spaces with two relations. This definition will not be given in a syntactic description, for a syntactic definition we will leave to future work.

Before defining inductive $\Pi_2$-statements, we first define the so-called first-round good $\exists$-statements of the form $\exists\vec{q}(\overline{\alpha}\leq\overline{\beta}\ \&\ \overline{\gamma}\prec\overline{\delta})$, which are the ones that can be equivalently transformed into a ``good'' meta-conjunction of inequalities without existential propositional quantifiers:

\begin{definition}[$(\Omega,\epsilon_{\vec p})$-First-Round Good $\exists$-Statements]\label{Def:Exists:Statements}
Given an order-type $\epsilon_{\vec p}$ for $\vec p$, and a dependence order $<_{\Omega}$ on $\vec p$, an $\exists$-statement $\exists\vec{q}(\overline{\alpha}\leq\overline{\beta}\ \&\ \overline{\gamma}\prec\overline{\delta})$ is $(\Omega,\epsilon_{\vec p})$-first-round good, if by the first part of the algorithm $\mathsf{ALBA}^{\Pi_2}$ defined in Section \ref{SubSec:ALBA:Pi_2}, it can be transformed into a meta-conjunction of inequalities $\mathsf{MetaConIneq}$, each inequality in which is $(\Omega,\epsilon_{\vec p})$-inductive (i.e.\ the positive tree of the left-hand side and the negative tree of the right-hand side are both $(\Omega,\epsilon_{\vec p})$-inductive, which might contain black connectives) for the same $(\Omega,\epsilon_{\vec p})$.
\end{definition}

\begin{definition}[Inductive $\Pi_2$-Statements]\label{Def:Pi_2:Inductive}
Given an order-type $\epsilon_{\vec p}$ for $\vec p$, and a dependence order $<_{\Omega}$ on $\vec p$, a $\Pi_2$-statement $$\overline{\phi}\leq\overline{\psi}\ \&\ \overline{\chi}\prec\overline{\xi}\Rightarrow\exists\vec{q}(\overline{\alpha}\leq\overline{\beta}\ \&\ \overline{\gamma}\prec\overline{\delta})$$ is $(\Omega,\epsilon_{\vec p})$-inductive, if
\begin{itemize}
\item it does not contain nominals or black connectives $\Diamondblack,\blacksquare,{\boxdotb},{\diamdotb}$;
\item $\exists\vec{q}(\overline{\alpha}\leq\overline{\beta}\ \&\ \overline{\gamma}\prec\overline{\delta})$ is $(\Omega,\epsilon_{\vec p})$-first-round good;
\item $\overline{\phi}\leq\overline{\psi}\ \&\ \overline{\chi}\prec\overline{\xi}\Rightarrow\mathsf{MetaConIneq}$ is an $(\Omega,\epsilon_{\vec p})$-inductive quasi-inequality (which might contain black connectives), where $\mathsf{MetaConIneq}$ is defined as in Definition \ref{Def:Exists:Statements}.
\end{itemize}
A $\Pi_2$-statement is inductive, if it is $(\Omega,\epsilon_{\vec p})$-inductive for some $\epsilon_{\vec p}$ and $<_\Omega$.
\end{definition}

\subsection{The Definition of Restricted Inductive $\Pi_2$-Statements}

In this section, we will define restricted inductive $\Pi_2$-statements, which will be shown to be canonical. For this class of $\Pi_2$-statements, we will give a syntactic description.

Again, before defining restricted inductive $\Pi_2$-statements, we first defined the restricted first-round good $\exists$-statements.

\begin{definition}[$(\Omega,\epsilon_{\vec q})$-Restricted First-Round Good $\exists$-Statements]\label{Def:Exists:Statements}
Given order-types $\epsilon_{\vec p}$ and $\epsilon_{\vec q}$ for $\vec p$ and $\vec q$, and a dependence order $<_{\Omega}$ on $\vec p$ and $\vec q$ such that $p<_{\Omega}q$ for all $p$ in $\vec p$ and $q$ in $\vec q$, an $\exists$-statement $\exists\vec{q}(\overline{\alpha}\leq\overline{\beta}\ \&\ \overline{\gamma}\prec\overline{\delta})$ is $(\Omega,\epsilon_{\vec q})$-restricted first-round good, if inequalities in $\overline{\alpha}\leq\overline{\beta}\ \&\ \overline{\gamma}\prec\overline{\delta}$ can be divided into the following two kinds:

\begin{itemize}
\item $\theta\triangleleft\eta$, which we call $(\Omega,\epsilon_{\vec q})$-restricted receiving inequality, where 
\begin{itemize}
\item all branches in $-\theta$ and $+\eta$ ending with propositional variables in $\vec q$ are $\epsilon_{\vec q}^{\partial}$-critical (i.e.\ in the first half of the algorithm, the inequality $\theta\triangleleft\eta$ is a ``receiving'' inequality);
\item all branches of $+\theta$ and $-\eta$ are Skeleton branches (i.e.\ in the second half of the algorithm, their branches behave like Skeleton branches);
\end{itemize}
\item $\iota\triangleleft\zeta$, which we call $(\Omega,\epsilon_{\vec q})$-restricted solvable inequality, where 
\begin{itemize}
\item exactly one of $-\iota,+\zeta$ is such that all branches ending with propositional variables in $\vec q$ are $\epsilon_{\vec q}^{\partial}$-critical (without loss of generality we denote this generation tree $\star\rho$ and the other one $*\kappa$), and the other contains $\epsilon_{\vec q}$-critical branches;
\item $*\kappa$ is $(\Omega, \epsilon_{\vec q})$-inductive, and all $\epsilon_{\vec q}$-critical branches in $*\kappa$ are PIA branches;
\item for all the $\epsilon_{\vec q}$-critical branches in $*\kappa$ ending with $q$, all propositional variables $r$ in $\star\rho$, we have $r<_{\Omega} q$.
\item for every branch in $-\iota$ and $+\zeta$ which is not $\epsilon_{\vec q}$-critical, it is a Skeleton branch in $+\iota$ and $-\zeta$.
\end{itemize}
\end{itemize}
\end{definition}

The basic idea behind the ``solvable inequality'' is that after the application of residuation rules exhaustively, the resulting inequality is of the form $\upsilon\leq r$ (when $\epsilon(r)=1$) or of the form $r\leq\upsilon$ (when $\epsilon(r)=\partial$) where all branches in $+\upsilon$ (resp.\ $-\upsilon$) ending with $p$ in $\vec p$ is a Skeleton branch. 

\begin{definition}[Restricted Inductive $\Pi_2$-Statements]\label{Def:Pi_2:Restricted:Inductive}
Given order-types $\epsilon_{\vec p}$, $\epsilon_{\vec q}$ for $\vec p$, $\vec q$, and a dependence order $<_{\Omega}$ on $\vec p$ and $\vec q$ such that $p<_{\Omega}q$ for all $p$ in $\vec p$ and $q$ in $\vec q$, a $\Pi_2$-statement $$\overline{\phi}\leq\overline{\psi}\ \&\ \overline{\chi}\prec\overline{\xi}\Rightarrow\exists\vec{q}(\overline{\alpha}\leq\overline{\beta}\ \&\ \overline{\gamma}\prec\overline{\delta})$$ is $(\Omega,\epsilon_{\vec p})$-restricted inductive,
\begin{itemize}
\item it does not contain nominals or black connectives $\Diamondblack,\blacksquare,{\boxdotb},{\diamdotb}$ or ${\diamdot},{\boxdot}$;
\item $\exists\vec{q}(\overline{\alpha}\leq\overline{\beta}\ \&\ \overline{\gamma}\prec\overline{\delta})$ is $(\Omega,\epsilon_{\vec q})$-restricted first-round good;
\item each inequality in $\phi\leq\psi$ and $\chi\prec\xi$ is either $(\Omega,\epsilon_{\vec p})$-receiving or $(\Omega,\epsilon_{\vec p})$-solvable.
\end{itemize}
A $\Pi_2$-statement is restricted inductive, if it is $(\Omega,\epsilon_{\vec p})$-restricted inductive for some $\epsilon_{\vec p}$ and $<_{\Omega}$.
\end{definition}

\section{The Algorithm $\mathsf{ALBA}^{\Pi_2}$}\label{SubSec:ALBA:Pi_2}

The algorithm $\mathsf{ALBA}^{\Pi_2}$ transforms the input $\Pi_2$-statement $$\overline{\phi}\leq\overline{\psi}\ \&\ \overline{\chi}\prec\overline{\xi}\Rightarrow\exists\vec{q}(\overline{\alpha}\leq\overline{\beta}\ \&\ \overline{\gamma}\prec\overline{\delta})$$ into an equivalent set of pure quasi-inequalities which does not contain occurrences of propositional variables or propositional quantifiers, and therefore can be translated into the first-order correspondence language via the standard translation of the expanded language (see page \pageref{Sub:FOL:ST}). The language on which we manipulate the algorithm is almost the same as the algorithm for quasi-inequalities, except that we have additional existential propositional quantifiers $\exists q$.

Now we define the algorithm in two halves, the first half aims at reducing $$\exists\vec{q}(\overline{\alpha}\leq\overline{\beta}\ \&\ \overline{\gamma}\prec\overline{\delta})$$ into a meta-conjunction of inequalities, and the second half is just the algorithm $\mathsf{ALBA}$ for quasi-inequalities.

\paragraph{First Half}
The algorithm receives a $\Pi_2$-statement $$\overline{\phi}\leq\overline{\psi}\ \&\ \overline{\chi}\prec\overline{\xi}\Rightarrow\exists\vec{q}(\overline{\alpha}\leq\overline{\beta}\ \&\ \overline{\gamma}\prec\overline{\delta})$$ as input. The first half of the algorithm operates on the $\exists$-statement $\exists\vec{q}(\overline{\alpha}\leq\overline{\beta}\ \&\ \overline{\gamma}\prec\overline{\delta})$, and executes in three stages:

\begin{enumerate}

\item \textbf{Preprocessing}\footnote{Notice that here we do not have first-approximation anymore.}:

\begin{enumerate}

\item In each inequality $\theta\triangleleft\eta$ (where $\triangleleft\in\{\leq,\prec\}$) in the $\exists$-statement $\exists\vec{q}(\overline{\alpha}\leq\overline{\beta}\ \&\ \overline{\gamma}\prec\overline{\delta})$, consider the signed generation trees $+\theta$ and $-\eta$, apply the distribution rules:

\begin{enumerate}
\item Push down $+\Diamond,+{\diamdot}, -\neg, +\land, -\to$ by distributing them over nodes labelled with $+\lor$ which are Skeleton nodes (see Figure \ref{Figure:distribution:rules} on page \pageref{Figure:distribution:rules}), and

\item Push down $-\Box,-{\boxdot},+\neg, -\lor, -\to$ by distributing them over nodes labelled with $-\land$ which are Skeleton nodes (see Figure \ref{Figure:distribution:rules:2} on page \pageref{Figure:distribution:rules:2}).

\end{enumerate}

\item Apply the splitting rules to each inequality occurring in the $\exists$-statement:

$$\infer{\theta\leq\eta\ \&\ \theta\leq\iota}{\theta\leq\eta\land\iota}
\qquad
\infer{\theta\leq\iota\ \&\ \eta\leq\iota}{\theta\lor\eta\leq\iota}
$$

$$\infer{\theta\prec\eta\ \&\ \theta\prec\iota}{\theta\prec\eta\land\iota}
\qquad
\infer{\theta\prec\iota\ \&\ \eta\prec\iota}{\theta\lor\eta\prec\iota}
$$

\item Apply the monotone and antitone variable-elimination rules for variables in $\vec q$ to the whole $\exists$-statement:

$$\infer{\exists\vec{q}(\overline{\alpha}(\bot)\leq\overline{\beta}(\bot)\ \&\ \overline{\gamma}(\bot)\prec\overline{\delta}(\bot))}{\exists q\exists\vec{q}(\overline{\alpha}(q)\leq\overline{\beta}(q)\ \&\ \overline{\gamma}(q)\prec\overline{\delta}(q))}$$

if $q$ is positive in $\overline{\alpha},\overline{\gamma}$ and negative in $\overline{\beta},\overline{\delta}$;

$$\infer{\exists\vec{q}(\overline{\alpha}(\top)\leq\overline{\beta}(\top)\ \&\ \overline{\gamma}(\top)\prec\overline{\delta}(\top))}{\exists q\exists\vec{q}(\overline{\alpha}(q)\leq\overline{\beta}(q)\ \&\ \overline{\gamma}(q)\prec\overline{\delta}(q))}$$

if $q$ is negative in $\overline{\alpha},\overline{\gamma}$ and positive in $\overline{\beta},\overline{\delta}$;

\item Apply the subordination rewritting rule to each inequality with $\prec$ in order to turn it into $\leq$:
$$\infer{{\diamdot}\theta\leq\eta}{\theta\prec\eta}$$
\end{enumerate}

Now we have a $\exists$-statement of the form $\exists\vec{q}(\overline{\alpha}\leq\overline{\beta})$.

\item \textbf{The reduction-elimination cycle}:

In this stage, for each inequality $\alpha\leq\beta$ in the $\exists$-statement $\exists\vec{q}(\overline{\alpha}\leq\overline{\beta})$, we apply the following rules together with the splitting rules in the previous stage to eliminate all the propositional variables in the set of inequalities\footnote{Notice that since we do not have nominals in this stage, we do not have approximation rules anymore.}:

\begin{enumerate}

\item Residuation rules, as described on page \pageref{Page:Residuation:Rules}:

\item The Ackermann rules. These rules are executed on the whole $\exists$-statement $\exists\vec{q}(\overline{\alpha}\leq\overline{\beta})$. We take all the inequalities in this $\exists$-statement.\\

The right-handed Ackermann rule:

The $\exists$-statement $\exists q\exists{\vec q}$ 
$\left\{ \begin{array}{ll}
\theta_1\leq q \\
\vdots\\
\theta_n\leq q \\
\eta_1\leq\iota_1\\
\vdots\\
\eta_m\leq\iota_m\\

\end{array} \right.$ 
is replaced by \\

$\exists{\vec q}$
$\left\{ \begin{array}{ll}
\eta_1((\theta_1\lor\ldots\lor\theta_n)/q)\leq\iota_1((\theta_1\lor\ldots\lor\theta_n)/q) \\
\vdots\\
\eta_m((\theta_1\lor\ldots\lor\theta_n)/q)\leq\iota_m((\theta_1\lor\ldots\lor\theta_n)/q) \\

\end{array} \right.$

where:

\begin{enumerate}

\item $q$ does not occur in $\theta_1, \ldots, \theta_n$;
\item Each $\eta_i$ is positive, and each $\iota_i$ negative in $q$, for $1\leq i\leq m$.

\end{enumerate}

The left-handed Ackermann rule:

The $\exists$-statement $\exists q\exists{\vec q}$ 
$\left\{ \begin{array}{ll}
q\leq\theta_1 \\
\vdots\\
q\leq\theta_n \\
\eta_1\leq\iota_1\\
\vdots\\
\eta_m\leq\iota_m\\

\end{array} \right.$
is replaced by\\

$\exists{\vec q}$ $\left\{ \begin{array}{ll}
\eta_1((\theta_1\land\ldots\land\theta_n)/q)\leq\iota_1((\theta_1\land\ldots\land\theta_n)/q) \\
\vdots\\
\eta_m((\theta_1\land\ldots\land\theta_n)/q)\leq\iota_m((\theta_1\land\ldots\land\theta_n)/q) \\

\end{array} \right.$

where:

\begin{enumerate}

\item $q$ does not occur in $\theta_1, \ldots, \theta_n$;
\item Each $\eta_i$ is negative, and each $\iota_i$ positive in $q$, for $1\leq i\leq m$.

\end{enumerate}
\end{enumerate}

\item \textbf{Output}: If in the previous stage, for some existential propositional quantifier $\exists q$, the algorithm gets stuck, i.e.\ some propositional quantifiers cannot be eliminated by the application of the reduction rules, then the algorithm halts and output ``failure''. Otherwise, we get a meta-conjunction of inequalities of the form $\alpha_1\leq\beta_1\ \&\ \ldots \ \&\ \alpha_n\leq\beta_n$, and the algorithm proceeds in the second half of the algorithm with input quasi-inequality $$\overline{\phi}\leq\overline{\psi}\ \&\ \overline{\chi}\prec\overline{\xi}\Rightarrow\alpha_1\leq\beta_1\ \&\ \ldots \ \&\ \alpha_n\leq\beta_n.$$
\end{enumerate}

\paragraph{Second Half} The second half of the algorithm receives the quasi-inequality $$\overline{\phi}\leq\overline{\psi}\ \&\ \overline{\chi}\prec\overline{\xi}\Rightarrow\alpha_1\leq\beta_1\ \&\ \ldots \ \&\ \alpha_n\leq\beta_n.$$ as input, and the proceeds as the algorithm $\mathsf{ALBA}$ defined in Section \ref{Sec:ALBA}.

\section{Soundness}\label{Subsec:Soundness:Pi_2}

The soundness proof with respect to arbitrary valuations of the algorithm $\mathsf{ALBA}^{\Pi_2}$ in this section is similar to Section \ref{Sec:Soundness}, and hence is omitted.

\section{Success}\label{Subsec:Success:Pi_2}

\subsection{Success of $\mathsf{ALBA}^{\Pi_2}$ on Inductive $\Pi_2$-Statements}

For inductive $\Pi_2$-statements, their definition is given by the execution result of the algorithm, therefore after the first half of the algorithm, the $\exists$-statement part is transformed into a meta-conjunction of inequalities, each of which is $(\Omega,\epsilon_{\vec p})$-inductive, and therefore the input of the second half is an $(\Omega,\epsilon_{\vec p})$-inductive quasi-inequality, by the success of $\mathsf{ALBA}$ on inductive quasi-inequalities, we have that the algorithm succeeds on inductive $\Pi_2$-statements.

\subsection{Success of $\mathsf{ALBA}^{\Pi_2}$ on Restricted Inductive $\Pi_2$-Statements}

For restricted inductive $\Pi_2$-statements, what we need to show is that the algorithm $\mathsf{ALBA}^{\Pi_2}$ succeeds on restricted inductive $\Pi_2$-statements.

\begin{lemma}\label{Lemma:First:Stage:First:Half:Pi_2}
Given an $(\Omega,\epsilon_{\vec p})$-restricted inductive $\Pi_2$-statement $$\overline{\phi}\leq\overline{\psi}\ \&\ \overline{\chi}\prec\overline{\xi}\Rightarrow\exists\vec{q}(\overline{\alpha}\leq\overline{\beta}\ \&\ \overline{\gamma}\prec\overline{\delta})$$ in which the $\exists$-statement $\exists\vec{q}(\overline{\alpha}\leq\overline{\beta}\ \&\ \overline{\gamma}\prec\overline{\delta})$ part is $(\Omega,\epsilon_{\vec q})$-restricted first-round good, after the first stage of the first half of the algorithm, the $\exists$-statement is transformed into another $(\Omega,\epsilon_{\vec q})$-restricted first-round good $\exists$-statement of the shape $\exists \vec{q}(\overline{\alpha}\leq\overline{\beta})$, which contains no black connective. 
\end{lemma}

\begin{proof}
Straightforward checking.
\end{proof}

For each $(\Omega,\epsilon_{\vec q})$-restricted first-round good $\exists$-statement $\exists \vec{q}(\overline{\alpha}\leq\overline{\beta})$, it can be written in the form $\exists \vec{q}(\overline{\theta}\leq\overline{\eta}\ \&\ \overline{\iota}\leq\overline{\zeta})$, where each inequality $\theta\leq\eta$ is $(\Omega,\epsilon_{\vec q})$-restricted receiving, and each inequality $\iota\leq\zeta$ is $(\Omega,\epsilon_{\vec q})$-restricted solvable.

\begin{lemma}
For each $(\Omega,\epsilon_{\vec q})$-restricted solvable inequality $\iota\leq\zeta$ in the $\exists$-statement above, it can be transformed into a meta-conjunction of inequalities of the following kinds:
\begin{itemize}
\item $\gamma\leq\delta$, where $\gamma\leq\delta$ is a $(\Omega,\epsilon_{\vec q})$-restricted receiving inequality;
\item $\kappa\leq q$ (if $\epsilon(q)=1$) or $q\leq\kappa$ (if $\epsilon(q)=\partial$), where every branch in $+\kappa$ (resp.\ $-\kappa$) is a Skeleton branch.
\end{itemize}
\end{lemma}

\begin{proof}

Without loss of generality we consider the situation when 

\begin{itemize}
\item $-\iota$ is such that all branches ending with propositional variables in $\vec q$ are $\epsilon_{\vec q}^{\partial}$-critical, and $+\zeta$ contains $\epsilon_{\vec q}$-critical branches;
\item $+\zeta$ is $(\Omega, \epsilon_{\vec q})$-inductive, and all $\epsilon_{\vec q}$-critical branches in $+\zeta$ are PIA branches;
\item for all the $\epsilon_{\vec q}$-critical branches in $+\zeta$ ending with $q$, all propositional variables $r$ in $-\iota$, we have $r<_{\Omega} q$.
\item for every branch in $-\iota$ and $+\zeta$ which is not $\epsilon_{\vec q}$-critical, it is a Skeleton branch in $+\iota$ and $-\zeta$.
\end{itemize}

The situation where $+\zeta$ is such that all branches ending with $\vec q$ are $\epsilon_{\vec q}^{\partial}$-critical etc. is symmetric.

We prove by induction on the complexity of $\zeta$. It is easy to see that $\zeta$ does not contain any black connective.

\begin{itemize}
\item when $\zeta$ is $q$: $\iota\leq\zeta$ belongs to the second class and $\epsilon_{q}=1$;
\item when $\zeta$ is $p$ or $\bot$ or $\top$: it cannot be the case, since $+\zeta$ contains $\epsilon_{\vec q}$-critical branches;
\item when $\zeta$ is $\neg\gamma$: we can first apply the residuation rule for $\neg$ to $\iota\leq\neg\gamma$ to obtain $\gamma\leq\neg\iota$, and then we can apply the induction hypothesis for the case where
\begin{itemize}
\item $+\neg\iota$ is such that all branches ending with propositional variables in $\vec q$ are $\epsilon_{\vec q}^{\partial}$-critical;
\item $-\gamma$ contains $\epsilon_{\vec q}$-critical branches, and is $(\Omega, \epsilon_{\vec q})$-inductive, and all $\epsilon_{\vec q}$-critical branches in $-\gamma$ are PIA branches;
\item for all the $\epsilon_{\vec q}$-critical branches in $-\gamma$ ending with $q$, all propositional variables $r$ in $+\neg\iota$, we have $r<_{\Omega} q$;
\item for every branch in $+\neg\iota$ and $-\gamma$ which is not $\epsilon_{\vec q}$-critical, it is a Skeleton branch in $-\neg\iota$ and $+\gamma$;
\end{itemize}
\item when $\zeta$ is $\Diamond\gamma$ or ${\diamdot}\gamma$: it cannot be the case, since $\Diamond$ or ${\diamdot}$ is on an $\epsilon_{\vec q}$-critical branch in $+\zeta$ but it is not a PIA node;
\item when $\zeta$ is $\Box\gamma$: we can first apply the residuation rule for $\Box$ to $\iota\leq\Box\gamma$ to obtain $\Diamondblack\iota\leq\gamma$, and then we can apply the induction hypothesis for the case where
\begin{itemize}
\item $-\Diamondblack\iota$ is such that all branches ending with propositional variables in $\vec q$ are $\epsilon_{\vec q}^{\partial}$-critical;
\item $+\gamma$ contains $\epsilon_{\vec q}$-critical branches, and is $(\Omega, \epsilon_{\vec q})$-inductive, and all $\epsilon_{\vec q}$-critical branches in $+\gamma$ are PIA branches;
\item for all the $\epsilon_{\vec q}$-critical branches in $+\gamma$ ending with $q$, all propositional variables $r$ in $-\Diamondblack\iota$, we have $r<_{\Omega} q$;
\item for every branch in $-\Diamondblack\iota$ and $+\gamma$ which is not $\epsilon_{\vec q}$-critical, it is a Skeleton branch in $+\Diamondblack\iota$ and $-\gamma$;
\end{itemize}
\item when $\zeta$ is ${\boxdot}\gamma$, the situation is similar to the $\Box\gamma$ case;
\item when $\zeta$ is $\gamma\land\delta$, we first apply the splitting rule for $\land$ to $\iota\leq\gamma\land\delta$ to obtain $\iota\leq\gamma$ and $\iota\leq\delta$, then there are two possibilities, namely both of $+\gamma$ and $+\delta$ have $\epsilon_{\vec q}$-critical branches, and only one of $+\gamma$ and $+\delta$ has $\epsilon_{\vec q}$-critical branches;

for the first possibility, we can apply the induction hypothesis for the case where 
\begin{itemize}
\item $-\iota$ is such that all branches ending with propositional variables in $\vec q$ are $\epsilon_{\vec q}^{\partial}$-critical;
\item $+\gamma$ and $+\delta$ contain $\epsilon_{\vec q}$-critical branches, and are $(\Omega, \epsilon_{\vec q})$-inductive, and all $\epsilon_{\vec q}$-critical branches in $+\gamma$ and $+\delta$ are PIA branches;
\item for all the $\epsilon_{\vec q}$-critical branches in $+\gamma$ and $+\delta$ ending with $q$, all propositional variables $r$ in $-\iota$, we have $r<_{\Omega} q$;
\item for every branch in $-\iota$ and $+\gamma$ and $+\delta$ which is not $\epsilon_{\vec q}$-critical, it is a Skeleton branch in $+\iota$ and $-\gamma$ and $-\delta$;
\end{itemize}

for the second possibility, without loss of generality we assume that $+\gamma$ has $\epsilon_{\vec q}$-critical branches and $+\delta$ contains no $\epsilon_{\vec q}$-critical branches, then $\iota\leq\delta$ is $(\Omega, \epsilon_{\vec q})$-receiving (it is easy to check the Skeleton condition in $+\iota$ and $-\delta$), and we can apply the induction hypothesis to $\iota\leq\gamma$ for the case where 
\begin{itemize}
\item $-\iota$ is such that all branches ending with propositional variables in $\vec q$ are $\epsilon_{\vec q}^{\partial}$-critical;
\item $+\gamma$ contains $\epsilon_{\vec q}$-critical branches, and is $(\Omega, \epsilon_{\vec q})$-inductive, and all $\epsilon_{\vec q}$-critical branches in $+\gamma$ are PIA branches;
\item for all the $\epsilon_{\vec q}$-critical branches in $+\gamma$ ending with $q$, all propositional variables $r$ in $-\iota$, we have $r<_{\Omega} q$;
\item for every branch in $-\iota$ and $+\gamma$ which is not $\epsilon_{\vec q}$-critical, it is a Skeleton branch in $+\iota$ and $-\gamma$;
\end{itemize}
\item when $\zeta$ is $\gamma\lor\delta$: since $\lor$ is an SRR node, only one of $\gamma$ and $\delta$ contains $\epsilon_{\vec q}$-critical branches (without loss of generality we assume that it is $\delta$). Then 
\begin{itemize}
\item $+\gamma$ is such that all branches ending with propositional variables in $\vec q$ are $\epsilon_{\vec q}^{\partial}$-critical;
\item for each $q$ in an $\epsilon_{\vec q}$-critical branch in $+\delta$, each $r$ that occurs in $+\gamma$, we have $r<_{\Omega} q$;
\end{itemize}
Now we apply the residuation rule for $\lor$ to $\iota\leq\gamma\lor\delta$ to obtain $\iota\land\neg\gamma\leq\delta$, then we can apply the induction hypothesis for the case where 
\begin{itemize}
\item $-(\iota\land\neg\gamma)$ is such that all branches ending with propositional variables in $\vec q$ are $\epsilon_{\vec q}^{\partial}$-critical;
\item $+\delta$ has $\epsilon_{\vec q}$-critical branches and is $(\Omega,\epsilon_{\vec q})$-inductive, and all $\epsilon_{\vec q}$-critical branches in $+\delta$ are PIA branches;
\item for all the $\epsilon_{\vec q}$-critical branches in $+\delta$ ending with $q$, all propositional variables $r$ in $-(\iota\land\neg\gamma)$, we have $r<_{\Omega} q$;
\item for every branch in $-(\iota\land\neg\gamma)$ and $+\delta$ which is not $\epsilon_{\vec q}$-critical, it is a Skeleton branch in $+(\iota\land\neg\gamma)$ and $-\delta$;
\end{itemize}
\item when $\zeta$ is $\gamma\to\delta$: similar to the $\gamma\lor\delta$ case (in the sense of using one of the residuation rules for $\to$).
\end{itemize}
\end{proof}

\begin{lemma}\label{Lemma:Success:Ackermann:Pi_2}
Given a $\exists$-statement $\exists(\overline{\gamma}\leq\overline{\delta}\ \&\ \overline{\kappa}\leq \overline{q})$ where each $\gamma\leq\delta$ is $(\Omega,\epsilon_{\vec q})$-restricted receiving, and each $\kappa\leq q$ (when $\epsilon(q)=1$) or $q\leq\kappa$ (when $\epsilon(q)=\partial$) is such that every branch in $+\kappa$ (resp.\ $-\kappa$) is a Skeleton branch, after the Ackermann rule, it is either again a $\exists$-statement divided into the two kinds of inequalities, or all inequalities of the form $\kappa\leq q$ or $q\leq\kappa$ are eliminated.
\end{lemma}

\begin{proof}
Without loss of generality we consider the application of the right-handed Ackermann rule where the $\exists$-statement $\exists q\exists{\vec q}$ 
$\left\{ \begin{array}{ll}
\kappa_1\leq q \\
\vdots\\
\kappa_n\leq q \\
\gamma_1\leq\delta_1\\
\vdots\\
\gamma_m\leq\delta_m\\
\iota_1\leq\zeta_1\\
\vdots\\
\iota_k\leq\zeta_k\\
\end{array} \right.$ 
is replaced by \\

$\exists{\vec q}$
$\left\{ \begin{array}{ll}
\gamma_1((\kappa_1\lor\ldots\lor\kappa_n)/q)\leq\delta_1((\kappa_1\lor\ldots\lor\kappa_n)/q) \\
\vdots\\
\gamma_m((\kappa_1\lor\ldots\lor\kappa_n)/q)\leq\delta_m((\kappa_1\lor\ldots\lor\kappa_n)/q) \\
\iota_1\leq\zeta_1\\
\vdots\\
\iota_k\leq\zeta_k\\
\end{array} \right.$

where:

\begin{enumerate}
\item $q$ does not occur in $\kappa_1, \ldots, \kappa_n$;
\item Each $\gamma_i$ is positive, and each $\delta_i$ negative in $q$, for $1\leq i\leq m$;
\item Each $\iota_i$, $\zeta_i$ does not contain $q$.
\end{enumerate}

Here each $\gamma_i\leq\delta_i$ is $(\Omega,\epsilon_{\vec q})$-restricted receiving, $\epsilon(q)=1$, and each branch in $+\kappa_i$ is a Skeleton branch. Since each branch in $+\gamma_i$ and $-\delta_i$ is Skeleton branch, and $q$ is positive in each $+\gamma_i$ and $-\delta_i$, so replacing $q$ by $\kappa_1\lor\ldots\lor\kappa_n$ makes $+\gamma_i((\kappa_1\lor\ldots\lor\kappa_n)/q)$ and $-\delta_i((\kappa_1\lor\ldots\lor\kappa_n)/q)$ again trees where each branch is Skeleton, and each $\iota_i\leq\zeta_i$ is kept the same, so the $\exists$-statement can still be divided into the two kinds of inequalities, or all inequalities of the form $\kappa\leq q$ or $q\leq\kappa$ are eliminated.
\end{proof}

By repeatedly applying Lemma \ref{Lemma:Success:Ackermann:Pi_2}, since after the monotone and antitone variable elimination rules, every propositional variable in $\vec q$ has a critical branch in the $\exists$-statement, so there will be an inequality of the form $\kappa\leq q$ or $q\leq\kappa$ after application of the residuation rules, therefore all propositional variables appearing in the existential propositional quantifiers can be eliminated. 

Summarizing the proofs above, we have the following result:

\begin{theorem}
The first half of $\mathsf{ALBA}^{\Pi_2}$ succeeds on $(\Omega,\epsilon_{\vec q})$-restricted first-round good $\exists$-statements, and output a meta-conjunction of the form $\overline{\gamma}\leq\overline{\delta}$, where each branch of $+\gamma_i$ and $-\delta_i$ is Skeleton.
\end{theorem}

Since each branch of $+\gamma_i$ and $-\delta_i$ is Skeleton, we have that $\gamma_i\leq\delta_i$ is $(\Omega,\epsilon_{\vec p})$-inductive for all order-types and all dependence orders, so the second half of the algorithm succeeds on the input quasi-inequality, so we have the following final result:

\begin{theorem}
$\mathsf{ALBA}^{\Pi_2}$ succeeds on all restricted inductive $\Pi_2$-statements.
\end{theorem}

\section{Canonicity}\label{Subsec:Canonicity:Pi_2}

In this section we prove the canonicity for restricted inductive $\Pi_2$-statements. The basic idea is again to check the topological correctness of the running of $\mathsf{ALBA}^{\Pi_2}$, and the focus here is again the Ackermann lemmas.

\begin{theorem}
Given a restricted inductive $\Pi_2$-statement as input, $\mathsf{ALBA}^{\Pi_2}$ can topologically correctly execute on it.
\end{theorem}

\begin{proof}
We basically follow the success proof in Section \ref{Subsec:Success:Pi_2}, while pay attention to the topological correctness of the execution. 

From Lemma \ref{Lemma:First:Stage:First:Half:Pi_2}, given a restricted inductive $(\Omega,\epsilon_{\vec p})$-$\Pi_2$-statement $$\overline{\phi}\leq\overline{\psi}\ \&\ \overline{\chi}\prec\overline{\xi}\Rightarrow\exists\vec{q}(\overline{\alpha}\leq\overline{\beta}\ \&\ \overline{\gamma}\prec\overline{\delta})$$ in which the $\exists$-statement $\exists\vec{q}(\overline{\alpha}\leq\overline{\beta}\ \&\ \overline{\gamma}\prec\overline{\delta})$ part is $(\Omega,\epsilon_{\vec q})$-restricted first-round good, after the first stage of the first half of the algorithm, the $\exists$-statement is transformed into another $(\Omega,\epsilon_{\vec q})$-restricted first-round good $\exists$-statement of the shape $\exists \vec{q}(\overline{\alpha}\leq\overline{\beta})$, which contains no black connective, and it is easy to check that each inequality is syntactically closed on the left-hand side, and syntactically open on the right-hand side.

Now we can easily check that the following property holds for the $\exists$-statement:\\

in each inequality inside the $\exists$-statement, the left-hand side is syntactically closed and the right-hand side is syntactically open.\\

It is easy to check that for each rule in Stage 2 of the first half, it does not break this property, so for each execution of the Ackermann rule, it is topologically correct, and after the execution of the Ackermann rule, it still satisfies the property stated above. Therefore, the execution of $\mathsf{ALBA}^{\Pi_2}$ is topologically correct in the first half.\\

Now for the second half, the input quasi-inequality is of the form $$\overline{\phi}\leq\overline{\psi}\ \&\ \overline{\chi}\prec\overline{\xi}\Rightarrow\alpha_1\leq\beta_1\ \&\ \ldots \ \&\ \alpha_n\leq\beta_n,$$

where 
\begin{itemize}
\item each inequality in $\phi\leq\psi$ and $\chi\prec\xi$ is either $(\Omega,\epsilon_{\vec p})$-receiving or $(\Omega,\epsilon_{\vec p})$-solvable;
\item each branch in $+\alpha_i$ and $-\beta_i$ is a Skeleton branch.
\end{itemize}

Now by applying the subordination rewritting rule and then splitting the quasi-inequality into a meta-conjunction of quasi-inequalities of the form 

$$\overline{\phi}\leq\overline{\psi}\Rightarrow\alpha_i\leq\beta_i.$$

By applying the first-approximation rule, we get 

$$\overline{\phi}\leq\overline{\psi}\ \&\ \nomi\leq\alpha_i\ \&\ \beta_i\leq\neg\nomj\ \Rightarrow\ \nomi\leq\neg\nomj,$$

here each $\phi$ is syntactically closed, each $\psi$ is syntactically open, every branch in $+\alpha_i$ and $-\beta_i$ are Skeleton branches.

Now by an argument similar to Lemma \ref{Lemma:Approximation:Quasi}, we have that $\nomi\leq\alpha_i$ and $\beta_i\leq\neg\nomj$ can be reduced to a meta-conjunction of the following kinds of inequalities:

\begin{itemize}
\item pure inequalities;
\item inequalities of the form $\nomk\leq p$ or $p\leq\neg\nomk$ (notice that they are syntactically closed on the left-hand side and syntactically open on the right-hand side);
\end{itemize}

Now the antecedent system has two kinds of inequalities, one kind is pure, the other kind is syntactically closed on the left-hand side and syntactically open on the right-hand side.

It is easy to check that for each rule in Stage 2 of the second half, it does not break this property, so for each execution of the Ackermann rule, it is topologically correct, and after the execution of the Ackermann rule, it still satisfies the property stated above. Therefore, the execution of $\mathsf{ALBA}^{\Pi_2}$ is topologically correct in the second half.\\
\end{proof}

\begin{corollary}
Restricted inductive $\Pi_2$-statements are canonical.
\end{corollary}

\section{Examples}\label{Subsec:Examples:Pi_2}

\begin{example}[See Lemma 2.1.12 in \cite{Sa16}]
Given a modal subordination algebra $(B,\prec,\Diamond)$ and its dual Stone space with two relations $(X,\tau,R,R')$,
$$(B,\prec,\Diamond)\Vdash\forall a\forall b(a\prec b\Rightarrow\exists c(a\prec c\prec b))\mbox{ iff }R\mbox{ is transitive.}$$
\end{example}

\begin{proof}
It suffices to show that $$(B,\prec,\Diamond)\Vdash \forall a\forall b(a\prec b\Rightarrow\exists c(a\prec c\prec b))$$ iff its dual $$(X,\tau,R,R')\vDash\forall w\forall v\forall u(Rwv\land Rvu\to Rwu).$$

Indeed,
\begin{center}
\begin{tabular}{c l}
 &$(B,\prec,\Diamond)\Vdash \forall a\forall b(a\prec b\Rightarrow\exists c(a\prec c\prec b))$\\
iff &$(X,\tau,R,R')\Vdash_{Clop}\forall a\forall b(a\prec b\Rightarrow\exists c(a\prec c\prec b))$\\
iff &$(X,\tau,R,R')\Vdash_{Clop}\forall a\forall b({\diamdot}a\leq b\Rightarrow\exists c({\diamdot}a\leq c\ \&\ {\diamdot}c\leq b))$\\
iff &$(X,\tau,R,R')\Vdash_{Clop}\forall a\forall b({\diamdot}a\leq b\Rightarrow\ {\diamdot}{\diamdot}a\leq b)$\\
iff &$(X,\tau,R,R')\Vdash_{Clop}\forall a\forall b\forall\nomi\forall\nomj({\diamdot}a\leq b\ \&\ \nomi\leq{\diamdot}{\diamdot}a\ \&\ b\leq\neg\nomj\ \Rightarrow\ \nomi\leq\neg\nomj)$\\
iff &$(X,\tau,R,R')\Vdash_{Clop}\forall a\forall b\forall\nomi\forall\nomj\forall\nomk({\diamdot}a\leq b\ \&\ \nomi\leq{\diamdot}{\diamdot}\nomk\ \&\ \nomk\leq a\ \&\ b\leq\neg\nomj\ \Rightarrow\ \nomi\leq\neg\nomj)$\\
iff &$(X,\tau,R,R')\Vdash_{Clop}\forall b\forall\nomi\forall\nomj\forall\nomk({\diamdot}\nomk\leq b\ \&\ \nomi\leq{\diamdot}{\diamdot}\nomk\ \&\ b\leq\neg\nomj\ \Rightarrow\ \nomi\leq\neg\nomj)$\\
iff &$(X,\tau,R,R')\Vdash_{Clop}\forall\nomi\forall\nomj\forall\nomk({\diamdot}\nomk\leq\neg\nomj\ \&\ \nomi\leq{\diamdot}{\diamdot}\nomk\ \Rightarrow\ \nomi\leq\neg\nomj)$\\
iff &$(X,\tau,R,R')\Vdash_{Clop}\forall\nomk({\diamdot}{\diamdot}\nomk\leq{\diamdot}\nomk)$\\
iff &$(X,\tau,R,R')\vDash\forall w(R[R[w]]\subseteq R[w])$\\
iff &$(X,\tau,R,R')\vDash\forall w\forall v\forall u(Rwv\land Rvu\to Rwu)$.
\end{tabular}
\end{center}
\end{proof}

\section{Comparison with Existing Works}\label{Sec:Comparison}

In this section, we compare the present paper's method with existing works on the correspondence theory of subordination algebras and subordination spaces. In general, the three approaches we mention here treat subordination algebras and subordination spaces in different forms, while they do not treat modal subordination algebras and Stone spaces with two relations as we deal with.

\subsection{de Rudder et al.'s approach}

In \cite{dR20,dRHaSt20,dRPa21}, de Rudder et al.\ studied correspondence theory of subordination algebras in the perspective of quasi-modal operators and slanted algebras. The basic idea behind their work is as follows (we follow the notation in \cite{dRPa21}):

\begin{definition}[Definition 6.4 in \cite{dRPa21}]
A \emph{tense slanted BAE} is defined as $(B,\Diamond, \blacksquare)$, where $B$ is a Boolean algebra and $\Diamond:B\to\mathsf{K}(B^{\delta})$, $\blacksquare:B\to\mathsf{O}(B^{\delta})$ are such that 
\begin{itemize}
\item $\Diamond\bot=\bot$ and $\Diamond(a\lor b)=\Diamond a\lor \Diamond b$;
\item $\blacksquare\top=\top$ and $\blacksquare(a\land b)=\blacksquare a\land\blacksquare b$;
\item $\Diamond a\leq b$ iff $a\leq\blacksquare b$.
\end{itemize}
\end{definition}

For any subordination algebra $(B,\prec)$, define $\Diamond_{\prec}$ and $\blacksquare_{\prec}$ as follows:

$$\Diamond_{\prec}a:=\bigwedge\{b\in B\mid a\prec b\}$$
$$\blacksquare_{\prec}a:=\bigvee\{b\in B\mid b\prec a\}$$

then the associated tense slanted BAE is defined as $(B,\Diamond_{\prec}, \blacksquare_{\prec})$.

Given a tense slanted BAE $(B,\Diamond, \blacksquare)$, define $\prec_{\Diamond}$ as $$a\prec_{\Diamond}b\mbox{ iff }\Diamond a\leq b\mbox{ iff }a\leq\blacksquare b,$$

then the associated subordination algebra is $(B,\prec_{\Diamond})$.

Then it is easy to see that there is a 1-1 correspondence between subordination algebras and tense slanted BAEs (see Proposition 6.7 in \cite{dRPa21}).

The Sahlqvist formulas are then defined in the language of tense slanted BAEs by using the bimodal language $\Diamond$, $\blacksquare$ and there duals $\Box:=\neg\Diamond\neg$ and $\Diamondblack:=\neg\blacksquare\neg$:

\begin{definition}[Definition 3.3 in \cite{dRHaSt20}]
\begin{itemize}
\item A bimodal formula $\phi$ is \emph{closed} (resp.\ \emph{open}) if it is obtained from $\top,\bot$, propsitional variables and their negations by applying $\land,\lor,\Diamond,\Diamondblack$ (resp. $\land,\lor,\Box,\blacksquare$).
\item A bimodal formula $\phi$ is \emph{positive} (resp.\ \emph{negative}) if it is obtained from $\top,\bot$ and propositional variables (resp.\ negations of propositional variables) by applying $\land,\lor,\Diamond,\Diamondblack,\Box,\blacksquare$.
\item A bimodal formula $\phi$ is \emph{s-positive} (resp.\ \emph{s-negative}) if it is obtained from closed positive formulas (resp.\ open negative formulas) by applying $\land,\lor,\Box,\blacksquare$ (resp.\ $\land,\lor,\Diamond,\Diamondblack$).
\item A bimodal formula $\phi$ is \emph{g-closed} (resp.\ \emph{g-open}) if it is obtained from closed (resp.\ open) formulas by applying $\land,\lor,\Box,\blacksquare$ (resp.\ $\land,\lor,\Diamond,\Diamondblack$).
\item A \emph{strongly positive} bimodal formula $\phi$ is a conjunction of formulas of the form $\Box^{\mu_1}\blacksquare^{\mu_2}\ldots\Box^{\mu_k}p$ where $\mu_i$ is a natural number for each $i$.
\item An \emph{s-untied} bimodal formula $\phi$ is obtained from strongly positive and s-negative formulas by applying only $\land,\Diamond,\Diamondblack$.
\item An \emph{s-Sahlqvist} bimodal formula $\phi$ is of the form $\Box^{\mu_1}\blacksquare^{\mu_2}\ldots\Box^{\mu_k}(\phi_1\to\phi_2)$ where $\phi_1$ is s-untied and $\phi_2$ is s-positive.
\end{itemize}
\end{definition}

\begin{remark}
The difference and similarity between de Rudder et al.'s approach and our approach can be summarized as follows:
\begin{enumerate}
\item de Rudder et al. are using our ${\boxdot},{\diamdot},{\boxdotb},{\diamdotb}$ as their signature, and they allow nested occurrences of these modalities, while in our approach, ${\diamdot}$ only occurs in the form ${\diamdot}\phi\leq\psi$ which stands for $\phi\prec\psi$. 
\item de Rudder et al.'s modalities are slanted, i.e.\ the diamond (resp.\ open) modalities are mapping clopen elements to closed (resp.\ open) elements since they correspond to the subordination relation, while we have two kinds of modalities, namely modalities corresponding to the subordination relation and ordinary modalities which map clopens to clopens.
\item In de Rudder et al.'s approach, they only treat formulas, while our first-class citizens are quasi-inequalties and $\Pi_2$-statements.
\end{enumerate}
\end{remark}

\subsection{Santoli's approach}

In \cite{Sa16}, Santoli studied the topological correspondence theory between conditions on algebras and first-order conditions on the dual subordination spaces, in the language of a binary connective definable from the squigarrow associated with the subordination relation, using the so-called $\forall\exists$-statements \cite{BeBeSaVe19,BeCaGhLa22,BeGhLa20} (which we call $\Pi_2$-statements in the present paper).

In \cite{Sa16}, Santoli uses a binary modality $\diamond$ which is defined as $\phi\diamond\psi:=\neg(\phi\rightsquigarrow\psi)$ (see page \pageref{page:squigarrow}), and its dual $\square$ which is defined as $\phi\square\psi:=\neg(\neg\phi\diamond\neg\psi)$. His definition of Sahlqvist formulas and $\forall\exists$-statements are given as follows:

\begin{definition}[Definition 6.1.1 in \cite{Sa16}]
\begin{itemize}
\item A formula $\phi$ is \emph{positive $\diamond$-free} if it is obtained from $\top$ and propositional variables by applying $\lor, \land$.
\item A formula $\theta$ is a \emph{Sahlqvist antecedent} if it is obtained from $\top$ and $\phi\diamond\psi$ using $\lor,\land$ where $\phi$ and $\psi$ are positive $\diamond$-free.
\item A formula $\chi$ is \emph{positive} if it is obtained from $\top$ and $\phi\diamond\psi$ and $\phi\square\psi$ using $\lor,\land$ where $\phi$ and $\psi$ are positive $\diamond$-free.
\item A \emph{non-separating} formula $S(p)$ is $F(p)\lor G(\neg p)$ where there are positive $\diamond$-free formulas $\phi,\psi$ such that $F(p)$ is equal to either $\phi\diamond p$ or $p\diamond\phi$ and $G(\neg p)$ is euqal to either $\psi\diamond\neg p$ or $\neg p\diamond\psi$.
\item A \emph{general positive} formula is a formula $\chi(\overline{p})$ which is a conjunction of
non-separating formulas $S(p_1),\ldots,S(p_n)$ and positive formulas, where $\overline{p}=p_1,\ldots,p_n$ are propositional variables.
\item A \emph{Sahlqvist} formula is a formula $\theta\to\chi(\overline{p})$ where $\theta$ is a Sahlqivst antecedent, $\chi(\overline{p})$ is a general positive formula, and the propositional variables $\overline{p}=p_1,\ldots,p_n$ do not occur in $\theta$.
\end{itemize}
\end{definition}

\begin{definition}[Definition 6.1.1 in \cite{Sa16}]
A Sahlqvist statement is $\Psi$ in the signature of $(\land,\neg,\top,\diamond)$ of the form
$$\Psi:=\forall\overline{q}(\forall\overline{p}(\theta\land(\bigwedge^{k}_{l=1}S_l(p_l))=\top)\Rightarrow\forall\overline{r}(\chi(\overline{r})=\top))$$ where
\begin{itemize}
\item $\theta$ is a Sahlqvist antecedent;
\item $\chi(\overline{r})$ is a general positive formula;
\item the $S_l(p_l)$'s are non-separating formulas;
\item $\overline{q}$ are all propositional variabless not among $\overline{p}=p_1,\ldots,p_k$ and $\overline{r}$ which occur in the formula;
\item the proposition variables $\overline{p}$ and $\overline{r}$ do not occur anywhere but in their respective non-separating formulas.
\end{itemize}
\end{definition}

\begin{remark}
The difference and similarity between Santoli's approach and our approach can be summarized as follows:
\begin{enumerate}
\item Santoli treats both Sahlqvist formulas and Sahlqivst $\forall\exists$-statements, which is similar to our approach. However, our treatment of $\forall\exists$-statements (i.e.\ $\Pi_2$-statements) are based on quasi-inequalities, which is different from Santoli's approach. 
\item Santoli uses a diamond-type binary modality as signature, while we use both the subordination relation and the dotted modalities to interpret the subordination relation and its corresponding closed relation.
\item Santoli uses the non-separating formulas, which is a special feature in his approach and not used anywhere else.
\item Santoli does not allow nested occurrences of modalities, which is similar to our usage of dotted modalities, while we allow nested occurrences of ordinary modalities.
\end{enumerate}
\end{remark}

\subsection{Balbiani and Kikot's approach}

Balbiani and Kikot \cite{BaKi12} investigated the Sahlqvist theory in the language of region-based modal logics of space, which uses a contact relation. 

The syntax of \cite{BaKi12} is two-layered. The inner layer is the layer of Boolean terms, and is defined as follows:

$$a::=p\mid 0\mid -a \mid(a \cup b)$$

and $1$ is defined as $-0$, $(a \cap b)$ is defined as $-(-a \cup-b)$. A term $a$ is \emph{positive} if $a$ is obtained from variables and $1$ by applying $\cup,\cap$. The second layer is defined as follows:
$$\phi::=a\equiv b\mid C(a,b)\mid\bot\mid\neg\phi\mid(\phi\vee\psi)$$

and $\top$ is defined as $\neg\bot$, $(\phi\land\psi)$ is defined as $\neg(\neg\phi\lor\neg\psi),(\phi\to\psi)$ is defined as $(\neg\phi\lor\psi)$ and $(\phi\leftrightarrow\psi)$ for $(\neg(\phi \lor\psi)\lor\neg(\neg\phi\lor\neg\psi))$, $a\not\equiv b$ is defined as $\neg a\equiv b$, $\bar{C}(a,b)$ is defined as $\neg C(a,b)$.

The semantics of the language is given as follows:

\begin{definition}[Section 3 in \cite{BaKi12}]
A \emph{Kripke frame} is $\mathcal{F}=(W,R)$ where $W$ is non-empty and $R$ is a binary relation on $W$. A \emph{valuation} on $\mathcal{F}$ is a function $V$ assigning to each Boolean variable $p$ a subset $V(p)$ of $W$. The interpretation of Boolean terms are given as follows:
\begin{itemize}
\item $V(0)=\emptyset$;
\item $V(-a)=W-V(a)$;
\item $V(a\cup b)=V(a)\cup V(b)$.
\end{itemize}
A \emph{Kripke model} is $\mathcal{M}=(W,R,V)$ where $\mathcal{F}=(W,R)$ is a Kripke frame and $V$ is a valuation on $\mathcal{F}$. The satisfaction relation is defined as follows:
\begin{itemize}
\item $\mathcal{M}\Vdash a\equiv b$ iff $V(a)=V(b)$;
\item $\mathcal{M}\Vdash C(a,b)$ iff there exist $x,y\in W$ such that $xRy$, $x\in V(a)$ and $y\in  V(b)$;
\item $\mathcal{M}\nVdash\bot$;
\item $\mathcal{M}\Vdash\neg\phi$ iff $\mathcal{M}\nVdash \phi$;
\item $\mathcal{M}\Vdash\phi\vee\psi$ iff $\mathcal{M}\Vdash\phi$ or $\mathcal{M}\Vdash\psi$.
\end{itemize}

\end{definition}

The definition of Sahlqvist formulas is given as follows:

\begin{definition}[Section 2 in \cite{BaKi12}]
\begin{itemize}
\item A formula $\phi$ is \emph{negation-free} if it is obtained from $\top$, $a\not\equiv 0$ and $C(a,b)$ (where $a$ and $b$ are positive terms) by applying $\land,\lor$.
\item A formula $\phi$ is \emph{positive} if it is obtained from $\top$, $a\not\equiv 0$, $-a \equiv 0$, $C(a,b)$ and $\bar{C}(-a,-b)$ (where $a$ and $b$ are positive terms) by applying $\land,\lor$.
\item A formula $\phi$ is \emph{Sahlqvist} if it is of the form $\psi\to\chi$ where $\psi$ is negation-free and $\chi$ is positive.
\end{itemize}
\end{definition}

According to \cite{BaKi12}, precontact logics contains all instances of the following formulas:

\begin{itemize}
\item $C(a,b)\\to a\not\equiv 0\land b\not\equiv 0$;
\item $C(a\cup b,c)\leftrightarrow C(a,c)\lor C(b,c)$;
\item $C(a,b\cup c)\leftrightarrow C(a,b)\lor C(a,c)$.
\end{itemize}

which means that the relation $C$ on a Boolean algebra is a \emph{precontact relation} or \emph{proximity} (see Definition 2.1.2 in \cite{Sa16} or \cite{DuVa07}). Given a subordination $\prec$, the relation $a C_{\prec} b:=a\not\prec\neg b$ is a proximity. Given a proximity $C$, the relation $a\prec_{C}b:=a\bar{C}\neg b$ is a subordination. Indeed, there is a 1-1 correspondence between proximities and subordinations.

\begin{remark}
The similarity and difference between Balbiani and Kikot's approach and our approach can be summarized as follows:
\begin{itemize}
\item Balbiani and Kikot use a two-layered syntax, which is similar to our usage of formulas as one layer and inequalities of the form $\phi\leq\psi$, $\phi\prec\psi$, quasi-inequalities and $\Pi_2$-statements as another layer.
\item Balbiani and Kikot uses the precontact relation, while we use the subordination relation.
\item Balbiani and Kikot do not allow nested occurrences of $C$, which is similar to our approach where the dotted modalities are not nested, although we allow nested occurrences of ordinary modalities.
\item Balbiani and Kikot do not have $\Pi_2$-statements, which is what we use substantially.
\end{itemize}
\end{remark}

\paragraph{Acknowledgement} The research of the author is supported by Taishan University Starting Grant ``Studies on Algebraic Sahlqvist Theory'' and the Taishan Young Scholars Program of the Government of Shandong Province, China (No.tsqn201909151). The author would like to thank Nick Bezhanishvili for his suggestions and comments on this project.

\bibliographystyle{abbrv}
\bibliography{Modal_Subord_Algebra}
\end{document}